\documentclass[reqno]{amsart}
\usepackage{graphicx, amsmath,amsfonts,amsthm,amssymb, paralist, fullpage}
\usepackage{hyperref}
\newtheorem{theorem}{Theorem}[section]
\newtheorem{corollary}[theorem]{Corollary}
\newtheorem{lemma}[theorem]{Lemma}
\newtheorem{proposition}[theorem]{Proposition}
\theoremstyle{definition}

\theoremstyle{remark}
\newtheorem{remark}[theorem]{Remark}
\numberwithin{equation}{section}

\newcommand{\g}{\geqslant}
\newcommand{\ar}{\rangle}
\newcommand{\al}{\langle}

\newcommand{\RR}{\mathbb{R}}\newcommand{\R}{\mathbb{R}}
\newcommand{\ZZ}{\mathbb{Z}}
\newcommand{\CC}{\mathbb{C}}

\newcommand{\NN}{\mathbb{N}}
\newcommand{\p}{\partial}

\newcommand{\les}{\leqslant}
\newcommand{\lesa}{\lesssim}

\newcommand{\mc}[1]{\mathcal{#1}}

\newcommand{\lr}[1]{ \langle #1 \rangle}
\newcommand{\ind}{\mathbbold{1}}

\DeclareSymbolFont{bbold}{U}{bbold}{m}{n}
\DeclareSymbolFontAlphabet{\mathbbold}{bbold}

\DeclareMathOperator*{\supp}{supp}

\newcommand{\N}{\mathbb{N}}

\newcommand{\EQ}[1]{\begin{equation}\begin{split} #1 \end{split}\end{equation}}

\newcommand{\pq}{\quad}

\newcommand{\la}{\lambda}

\newcommand{\PN}[1]{P^N_{#1}}
\newcommand{\PF}[1]{P^F_{#1}}

\setdefaultenum{(i)}{(a)}{1}{A}

\begin{document}

\title{The Zakharov system in dimension $d\g 4$}%

\author[T.~Candy]{Timothy Candy}
\address[T.~Candy]{Department of Mathematics and Statistics, University of Otago, PO Box 56, Dunedin 9054, New Zealand}
\email{tcandy@maths.otago.ac.nz}

\author[S.~Herr]{Sebastian Herr}
\address[S.~Herr]{Fakult\"at f\"ur
  Mathematik, Universit\"at Bielefeld, Postfach 10 01 31, 33501
  Bielefeld, Germany}
\email{herr@math.uni-bielefeld.de}

\author[K.~Nakanishi]{Kenji Nakanishi}
\address[K.~Nakanishi]{Research Institute for Mathematical Sciences, Kyoto University, Kyoto 606-8502, Japan}
\email{kenji@kurims.kyoto-u.ac.jp}

\begin{abstract}
 The sharp range of Sobolev spaces is determined in which the Cauchy problem for the classical Zakharov system is well-posed, which includes existence of solutions, uniqueness, persistence of initial regularity, and real-analytic dependence on the initial data. In addition, under a condition on the data for the Schr\"odinger equation at the lowest admissible regularity, global well-posedness and scattering is proved.
The results cover energy-critical and energy-supercritical dimensions $d\g 4$.
\end{abstract}


\maketitle

\section{Introduction}
Consider an at most weakly magnetized plasma with ion density fluctuation $v:\RR^{1+d} \to \RR$ and complex envelope $u:\RR^{1+d} \to \CC$ of the electric field. In \cite{zakharov_collapse_1972} Zakharov derived the equations for the dynamics of Langmuir waves, which are rapid oscillations of the electric field in a conducting plasma. A scalar version of his model, called the Zakharov system, is given by
    \begin{equation}\label{eq:Zakharov}
        \begin{split}
           i \p_t u + \Delta u &= v u \\
           \Box v &= \Delta |u|^2
        \end{split}
      \end{equation}
      with the d'Alembertian $\Box = \p_t^2 - \Delta$. We refer to  \cite{zakharov_collapse_1972,Colin2004,Texier2007} and the books \cite{GuoGanKongZhang2016, Sulem1999} for more details of the model and its derivation.

      The Zakharov system is Lagrangian, and formally the $L^2$-norm of $u$ and the energy
      \[
E_Z(u(t),v(t),\partial_tv(t)):=\int_{\RR^d}\frac12|\nabla u(t)|^2+\frac14||\nabla|^{-1}\partial_t v(t)|^2+\frac14| v(t)|^2 +\frac12 v(t)|u(t)|^2dx
      \]
      are constant in time.

      The Zakharov system \eqref{eq:Zakharov} is typically  studied as a Cauchy problem by prescribing initial data in Sobolev spaces, i.e.
      \begin{equation}\label{eq:ic}u(0) = f \in H^s(\RR^d)\quad \text{ and }\quad (v, |\nabla|^{-1}\p_t v)(0) = (g_0, g_1) \in H^\ell(\RR^d)\times H^\ell(\RR^d).
      \end{equation}
      In recent years, this initial value problem has attracted considerable attention, partly driven by the close connection to the focusing cubic nonlinear Schr\"odinger equation (NLS) which arises as a subsonic limit of the Zakharov system \eqref{eq:Zakharov} \cite{Schochet-Weinstein,added1988,ozawa_nonlinear_1992,Kenig1995,Mas08}. In addition, bound states for the focusing cubic NLS are closely intertwined with the global dynamics of \eqref{eq:Zakharov}. More precisely, if $Q_\omega:\R^d\to \R$ is a bound state for the  focusing cubic NLS, in other words if $Q_\omega$ solves
      \[
-\Delta Q_\omega+\omega Q_\omega=Q_\omega^3,
\]
then $(u,v)=(e^{it\omega}Q_\omega,-Q_\omega^2)$ is a global (non-dispersive) solution of \eqref{eq:Zakharov}. This connection has been used to analyze the blow-up behaviour \cite{glangetas_concentration_1994,glangetas_existence_1994,merle_blowup_1998} in dimension $d=2$, and also in the periodic case \cite{KishiMaeda2013}. Furthermore, we can write the Zakharov energy as
    \[E_Z(u(t),v(t),\partial_tv(t))=E_{S}(u(t))+\frac14 \int_{\RR^d} |(1-i|\nabla|^{-1}\partial_t) v(t)+|u|^2|^2dx \]
where
     \[E_{S}(u(t)):=\int_{\RR^d}\frac12|\nabla u(t) |^2-\frac14|u(t)|^4dx\]
is the energy for the focusing cubic NLS. As the cubic NLS is energy-critical in $d=4$, the Zakharov system is also frequently referred to as energy-critical in dimension $d=4$ although, in contrast to the cubic NLS, the Zakharov system lacks scale-invariance, see \cite{Guo2018} for further discussion.

In the Zakharov system, the interplay between the different dispersive effects of solutions to Schr\"odinger and  wave equations leads to a rich local and global well-posedness theory \cite{added1988,ozawa_existence_1992,Kenig1995,bourgain_wellposedness_1996,ginibre_cauchy_1997,colliander_low_2008,Fang2008,Bejenaru2009,Bejenaru2011,Kishi2013,Bejenaru2015}. In particular, it turned out that the required regularity of the Schr\"odinger component can go below the scaling critical one ($s=d/2-1$) for the cubic nonlinear Schr\"odinger equation. Concerning the asymptotic behaviour of global solutions, scattering results have been proven in certain cases \cite{Ozawa1994,ginibre_scattering_2006,Hani2013, Bejenaru2015, Guo2013, Guo2014a, Guo2014,Kato2017, Guo2018}.\\

      The aim of this paper is twofold. First, we give a complete answer to the question of local well-posedness in dimension $d\g 4$, i.e. the energy-critical and super-critical dimensions. Second, we prove that these local solutions  are global in time and scatter, provided that the Schr\"odinger part is small enough.
To be more precise, consider the case $d\g 4$, and $(s, \ell)$ satisfying
\begin{equation}\label{eqn:cond on s l} \ell \g \frac{d}{2}-2, \qquad \max\Big\{\ell-1, \frac{\ell}{2} + \frac{d-2}{4}\Big\} \les s \les \ell + 2, \quad (s,\ell) \not = \Big(\frac{d}{2}, \frac{d}{2}-2\Big), \Big(\frac{d}{2},\frac{d}{2}+1\Big).\end{equation}

Our first main result is
    \begin{theorem}\label{thm:lwp-ill} The Zakharov system \eqref{eq:Zakharov} with initial condition \eqref{eq:ic} is locally well-posed with a real-analytic flow map, if and only if
      $(s,\ell)\in \RR^2$ satisfies \eqref{eqn:cond on s l}.
    \end{theorem}
    To be more precise, we consider mild solutions to an equivalent first order system \eqref{eq:Zakharov 1st order}, as usual. For this
 we show local well-posedness results, Theorem \ref{thm:gwp nonendpoint}, which applies to the non-endpoint case, and  Theorem \ref{thm:gwp endpoint}, for the endpoint case. Finally, we provide two examples in Subsection \ref{subsec:proof-lwp-ill}, which show that if the flow map exists for $(s,\ell)$ in the exterior of the region defined by \eqref{eqn:cond on s l}, it does not have bounded directional derivatives of second order at the origin. Partial ill-posedness results have been obtained earlier in \cite{ginibre_cauchy_1997,holmer_local_2007,Bejenaru2009,Domingues2019}. In the specific point $(s,\ell)=(2,3)$ in $d=4$ a stronger form of ill-posedness was proved in \cite[Section 7]{Bejenaru2015}, namely that there is no distributional solution at this regularity.

\begin{figure}[ht!]\label{fig:reg}
  \begin{center}
    \includegraphics[trim=200pt 700pt 600pt 750pt, clip=true, width=0.5\textwidth]{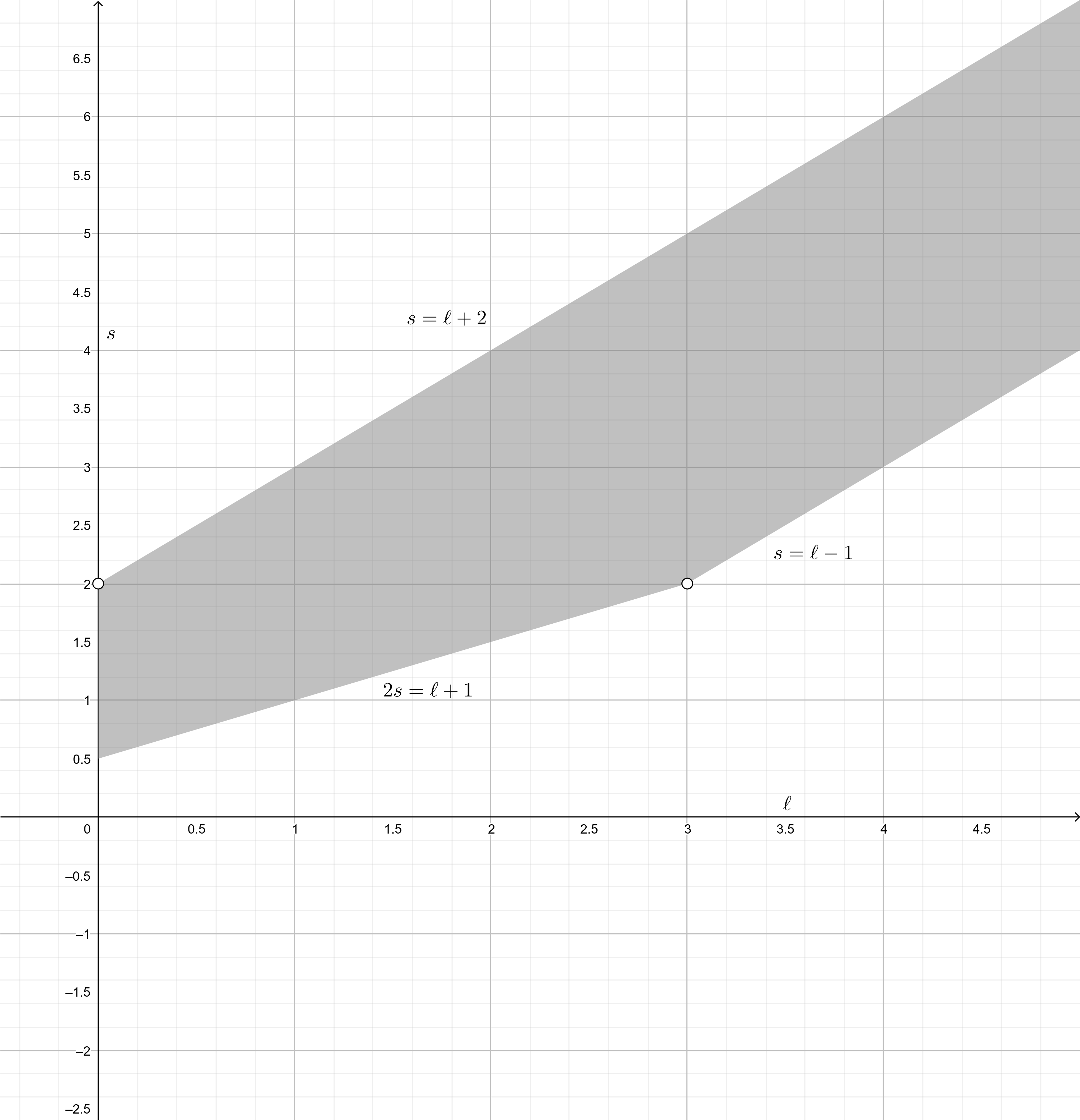}
  \end{center}
  \caption{In dimension $d=4$: Local well-posedness and small data global well-posedness within grey region, ill-posedness otherwise.}
\end{figure}

Our second main result is
\begin{theorem}\label{thm:small data gwp}
Let $d\g 4$ and $(s,\ell)$ satisfy \eqref{eqn:cond on s l}. For any data $(g_0, g_1) \in H^\ell(\RR^d) \times H^\ell(\RR^d)$, there exists $\epsilon>0$ such that for any $f \in H^s(\RR^d)$ satisfying $ \| f \|_{H^{\frac{d-3}{2}}}\les \epsilon$, we have a global solution $ u \in C(\RR, H^s(\RR^d))$, $(v, |\nabla|^{-1}\p_t v) \in C(\RR, H^\ell(\RR^d) \times H^\ell(\RR^d))$ to \eqref{eq:Zakharov} and \eqref{eq:ic}, which is unique under the condition
  \[ u \in L^2_{loc,t} (\R, W^{\frac{d-3}{2}, \frac{2d}{d-2}}_x(\R^d)) ,\]
and depends real-analytically on the initial data. This solution scatters as $t\to \pm \infty$.
\end{theorem}

Theorem  \ref{thm:small data gwp} is a consequence of Theorems \ref{thm:gwp nonendpoint}, \ref{thm:gwp endpoint} and \ref{thm:persistence}, which again apply to the first order system \eqref{eq:Zakharov 1st order} in the mild formulation, see Subsection \ref{subsec:proof-small-gwp}. In fact, we prove something stronger, and show that the smallness condition in Theorem \ref{thm:small data gwp} can be replaced with the weaker condition
    $$\| f \|_{H^\frac{d-3}{2}}^{\frac{1}{2}} \| e^{it\Delta} f\|_{L^2_t W^{\frac{d-3}{2},\frac{2d}{d-2}}_x}^{\frac{1}{2}} \les \epsilon.$$

We remark that Theorems \ref{thm:gwp nonendpoint}, \ref{thm:gwp endpoint} (setting $g_*=0$) also imply that the smallness condition on $f$ does not depend on $(g_0,g_1)$ provided that $\|(g_0, g_1)\|_{H^\frac{d-4}{2}} \ll 1$ is also sufficiently small. For readers primarily interested in this important and much easier case, we provide a simplified approach and results in Section \ref{sec:simple}.

In general, $\epsilon>0$ in Theorem \ref{thm:small data gwp} must depend on the wave initial data $(g_0,g_1)$, and it is not even uniform with respect to its norm, at least when $(s,\ell)$ is on a segment of the lowest regularity ($\ell=d/2-2$ and $(d-3)/2\le s<d/2-1$): Take any non-negative $f_0\in C^\infty_0(\RR^d)\setminus\{0\}$. Multiplying it with a large number $a\gg 1$, we can make the NLS energy negative $E_S(af_0)<0$. Imposing $g_0=-|af_0|^2$ and $g_1=0$ makes the Zakharov energy the same: $E_Z(af_0,g_0,g_1)=E_S(af_0)$. When the energy is negative, scattering is impossible, because the global dispersion would send the negative nonlinear part to zero as $t\to\infty$. Finally, to make the Schr\"odinger data small, we can use the scaling-invariance of the NLS: Let $f(x)=\lambda af_0(\lambda x)$ with $\lambda\to\infty$. Since this is the $\dot H^{d/2-1}$-invariant scaling, all $\dot H^s$ norms with $s<d/2-1$ tend to zero as the data concentrate, including the $L^2$ norm ($s=0$). For the wave component, the scaling leaves $\dot H^{d/2-2}$ invariant, which is the lowest (critical) regularity. In other words, we can make the Schr\"odinger data as small in $H^s$ as we like for $s<d/2-1$, while keeping the wave norm in $\dot H^{d/2-2}$.

Further, in the energy-critical case ($d=4$), we observe that there exist non-scattering solutions as soon as $\|g_0\|_{L^2}>\|W^2\|_{L^2}$, where $W(x)=(|x|^2/(d(d-2))+1)^{-1}$ is the ground state of the NLS.
To see this, start with  $f(x)=a W \chi(x/R)$ with a smooth cut-off function $\chi$ (which is needed since $W$ barely fails to be in $L^2(\RR^4)$). Choosing $a>1$, and then $R>1$ large enough depending on $a$, we obtain $E_S(f)<E_S(W)$ and $\||f|^2\|_{L^2}>\|W^2\|_{L^2}$, so that we can apply the grow-up result (with $g_0=-|f|^2$ and $g_1=0$ as above) in the radial case obtained in \cite{Guo2018}. The large data case in the energy-critical dimension $d=4$ is addressed in a follow-up paper \cite{Candy2021}.

The key contributions of Theorems \ref{thm:lwp-ill} and \ref{thm:small data gwp} are firstly that we give a complete characterisation of the region of well-posedness in arbitrary space dimension $d\g 4$, and secondly that we obtain global well-posedness and scattering for wave data of arbitrary size, only requiring the Schr\"odinger data to be small enough. In particular, in the energy-critical dimension $d=4$ this extends \cite{Bejenaru2015} to the subregion where $(s,\ell)= (1,0)$ or $s\geq 4\ell+1$ or $s> 2\ell+\frac{11}{8}$ and the scattering to wave data of arbitrary size.
Note that \cite{Bejenaru2015} covers the energy space $(s,\ell)=(1,0)$ but by a compactness argument, from which it is not immediately clear whether the solution map is analytic.
Further, if $d=4$, the large data threshold result in \cite{Guo2018} is restricted to radial data. In higher dimensions, this is an extension of the local well-posedness results in \cite{ginibre_cauchy_1997}, which apply in the subregion where $\ell\leq s\leq \ell+1$ and $2s>\ell+\frac{d-2}{2}$, and the global well-posedness and scattering result in \cite{Kato2017}, which applies if $(s,\ell)=(\frac{d-3}{2},\frac{d-4}{2})$ and both the wave and the Schr\"odinger data are small.

The recent well-posedness results cited above rely on a partial normal form transformation. This strategy introduces certain boundary terms which are non-dispersive and difficult to deal with in the low regularity setup. In this paper, we introduce a new perturbative approach which is based on Strichartz and maximal $L^2_{t,x}$ norms with additional temporal derivatives allowing us to exploit the different dispersive properties of the wave and the Schr\"odinger equation.  Further, the global well-posedness result allows for wave data of arbitrary size, which is achieved by treating the free wave evolution as a potential term in the Schr\"odinger equation.

One of the main challenges in proving the global well-posedness results in Theorem \ref{thm:lwp-ill} and Theorem \ref{thm:small data gwp}  in the range where $s > \ell +1$ lies in the fact that it seems impossible to control the endpoint Strichartz norm, i.e. to prove that $\lr{\nabla}^s u\in L^2_t L^{\frac{2d}{d-2}}_x$. To some extent, this is explained by considering
        $$ (i\p_t + \Delta) u = \phi_\lambda \psi_\mu $$ as a toy model for \eqref{eq:Zakharov},
where $\phi_\lambda = e^{it|\nabla|} f_\lambda$ is a free wave, $\psi_\mu = e^{it\Delta} g_\mu$ is a free solution to the Schr\"odinger equation, the wave data $f_\lambda$ has spatial frequencies $|\xi|\approx \lambda$, and the Schr\"odinger data $g_\mu$ has spatial frequencies $|\xi| \approx \mu$ with  $\mu \ll \lambda$. Note that this is essentially the first Picard iterate for \eqref{eq:Zakharov}. A computation shows that the product $\phi_\lambda \psi_\mu$ has spacetime Fourier support in the set $\{ |\tau| \ll \lambda^2,  |\xi| \approx \lambda\}$ and hence (modulo a free Schr\"odinger wave) we can write
    $$ u \sim (i\p_t + \Delta)^{-1}( \phi_\lambda \psi_\mu) \sim \lambda^{-2} (\phi_\lambda \psi_\mu). $$
In particular, we expect that (in the case $d=4$ for ease of notation)
    $$ \| \lr{\nabla}^s u \|_{L^2_t L^4_x(\RR^{1+4})} \approx \lambda^{s-2} \| \phi_\lambda \psi_\mu\|_{L^2_t L^4_x(\RR^{1+4})}.$$
If we assume the wave endpoint regularity, in $d=4$ we can only place $\phi_\lambda \in L^\infty_t L^2_x$. Thus applying H\"older's inequality together with the sharp Sobolev embedding and the endpoint Strichartz estimate for the free Schr\"odinger equation we see that
    $$ \| \phi_\lambda \psi_\mu \|_{L^2_t L^4_x(\RR^{1+4})}\lesa \| \phi_\lambda \|_{L^\infty_t L^4_x(\RR^{1+4})} \| \psi_\mu \|_{L^2_t L^\infty_x(\RR^{1+4})} \lesa \lambda \| f_\lambda \|_{L^2(\RR^4)} \mu \| g_\mu \|_{L^2(\RR^4)}. $$
Note that the above chain of inequalities is essentially forced if we may only assume the regularity $\phi_{\lambda} \in L^\infty_t L^2_x$. Consequently, we obtain
    $$ \| \lr{\nabla}^s u \|_{L^2_t L^4_x} \lesa \Big(\frac{\lambda}{\mu}\Big)^{s-1}\| f_\lambda \|_{L^2} \| g_\mu \|_{H^s}. $$
Again, as we can only place $f_\lambda \in L^2_x$, this imposes the restriction $s\les 1$. It is very difficult to see a way to improve the above computation, and in fact this high-low interaction is essentially what led to the restriction $s<1$ in \cite{Bejenaru2015, Kato2017}. Note however that this obstruction only leads to $\lr{\nabla}^s u \not \in L^2_t L^4_x(\RR^{1+4})$, and is \emph{not} an obstruction to well-posedness. In other words, provided only that $s\les 2$ we still have $u \in L^\infty_t H^s_x$ since similar to the above computation
    $$ \| u \|_{L^\infty_t H^s_x(\RR^{1+4})} \approx \lambda^{s-2} \| \phi_\lambda u_\mu \|_{L^\infty_t L^2_x} \lesa \Big(\frac{\lambda}{\mu}\Big)^{s-2}\| f_\lambda \|_{L^2_x} \| g_\mu \|_{H^s}. $$
In summary, the above example strongly suggests that it is not possible to construct  solutions to the Zakharov system by iterating in the endpoint Strichartz norms $L^2_t W^{s, 4}(\RR^{1+4})$, or even any space which contains the endpoint Strichartz space. Thus an alternative space is required, and this is what we construct in this paper.

A partial solution to the above problem of obtaining well-posedness in the regularity region $s\ge\ell+1$ was given in \cite{Bejenaru2015}. The approach taken there was to replace the endpoint Strichartz space $L^2_t W^{s,4}_x$ with the intermediate Strichartz spaces $L^q_t W^{s,r}_x$ for appropriate (non-endpoint, i.e. $q>2$) Schr\"odinger admissible $(q,r)$. However, the argument given in \cite{Bejenaru2015}  requires additional regularity for the wave component $v$ as it exploits Strichartz estimates for the wave equation to compensate for the loss in decay in the intermediate Schr\"odinger Strichartz spaces, and thus misses a neighbourhood of the corner $(s,l)=(\frac{d}{2},\frac{d}{2}-2)$.

The key observation that gives well-posedness in the full region \eqref{eqn:cond on s l} is that the output of the above high-low interaction has small temporal frequencies. Consequently, the endpoint Strichartz space only loses regularity at small temporal frequencies. This observation can be exploited by using norms of the form
        \begin{equation}\label{eqn:temporal deriv strich}
            \| ( \lr{\nabla} + |\p_t|)^a u \|_{L^2_t W^{s-2a, 4}_x(\RR^{1+4})}.
        \end{equation}
Note that if $u = e^{it\Delta} f$ is a free solution to the Schr\"odinger evolution, then $u$ has temporal Fourier support in $\{ |\tau| \approx |\xi|^2\}$ and hence
        $$ \| ( \lr{\nabla} + |\p_t|)^a u \|_{L^2_t W^{s-2a, 4}_x(\RR^{1+4})} \approx \| u \|_{L^2_t W^{s,4}_x}. $$
Thus the norm \eqref{eqn:temporal deriv strich} is equivalent to the standard endpoint Strichartz space for free Schr\"odinger waves. On the other hand, if $u$ has Fourier support in $\{ |\tau| \lesa |\xi| \}$, i.e. $u$ has only  small temporal frequencies, then
    $$\| ( \lr{\nabla} + |\p_t|)^a u \|_{L^2_t W^{s-2a, 4}_x(\RR^{1+4})} \approx \| u \|_{L^2_t W^{s-a, 4}_x}. $$
In other words, we only have $\lr{\nabla}^{s-a} u \in L^2_t L^4_x(\RR^{1+4})$ and thus we allow for a loss of regularity in the small temporal frequency region of the Strichartz norm. Moreover, again considering the above high-low interaction, we can control the output $(i\p_t + \Delta)^{-1} (\phi_\lambda \psi_\mu)$ in the temporal derivative Strichartz  space \eqref{eqn:temporal deriv strich} provided that $a \g s-1$. In particular choosing $a\sim 1$ gives the full range $s<2$. Thus roughly speaking, the norm \eqref{eqn:temporal deriv strich} matches the standard endpoint Strichartz space for the Schr\"odinger like portion of the evolution of $u$ (i.e. when $|\tau| \approx |\xi|^2$), but allows for a loss of regularity in the small temporal frequency regions $|\tau| \ll |\xi|^2$ of $u$ which are strongly influenced by nonlinear wave-Schr\"odinger interactions. We refer to estimate \eqref{eqn:stri with loss} and Remark \ref{rmk:str-control} below for further related comments.

\subsection{Outline of the paper}\label{subsect:outline}
In Section \ref{sec:not}, notation is introduced, the crucial function spaces are defined, and their key properties are discussed. Further, a product estimate for fractional time-derivatives is proved. Bilinear estimates for the Schr\"odinger and the wave nonlinearities are proved in Section \ref{sec:bil-est-schr} and \ref{sec:bil-est-wave}, respectively. In Section \ref{sec:simple} we provide a shortcut to simplified local and small data global well-posedness and scattering results which do not use the refined results of the following Sections. Local versions of the bilinear estimates in the endpoint case are proved in Section \ref{sec:local-bil}. In Section \ref{sec:wp} the technical well-posedness results are established, most notably Theorems \ref{thm:gwp nonendpoint} and Theorem \ref{thm:gwp endpoint}. Persistence of regularity is established in Section \ref{sec:pers}. Finally, the proofs of Theorem \ref{thm:lwp-ill} and Theorem \ref{thm:small data gwp} are completed in Section \ref{sec:proofs-main}.

\section{Notation and Preliminaries}\label{sec:not}

The Zakharov system has an equivalent first order formulation which is slightly more convenient to work with. Suppose that $(u, v)$ is a solution to \eqref{eq:Zakharov} and let $V = v - i |\nabla|^{-1} \p_t v$. Then $(u,V)$ solves the first order problem
    \begin{equation}\label{eq:Zakharov 1st order}
        \begin{split}
           i \p_t u + \Delta u &= \Re(V) u \\
            i\p_t V + |\nabla| V &= - |\nabla| |u|^2.
        \end{split}
    \end{equation}
Conversely, given a solution $(u, V)$ to \eqref{eq:Zakharov 1st order}, the pair $(u, \Re(V))$ solves the original Zakharov equation \eqref{eq:Zakharov}.

\subsection{Fourier multipliers}\label{subsec:fm}
Let $\varphi \in C^\infty_0(\RR)$ such that $\varphi \g 0$,  $\supp \varphi \subset \{\frac{1}{2} < r < 2\}$ and
$$ 1 = \sum_{ \lambda \in 2^\ZZ} \varphi\Big( \frac{r}{\lambda} \Big) \,\text{ for }\, r>0.$$
Let $\NN=\{0,1,2,\ldots\}$. For $\lambda\in 2^\NN$, define the spatial Fourier multipliers
$$ P_\lambda = \varphi\Big( \frac{|\nabla|}{\lambda}\Big)\text{ if }\lambda>1, \quad P_1 = \sum_{\lambda \in 2^\ZZ, \lambda\les 1} \varphi\Big( \frac{|\nabla|}{\lambda}\Big)$$
Thus $P_\lambda$ is a (inhomogeneous) Fourier multiplier localising  the spatial Fourier support to the set $\{ \frac{\lambda}{2} < |\xi| < 2 \lambda \}$ if $\lambda>1$ and $\{|\xi| < 2 \}$ if $\lambda=1$.
Further, for $\lambda\in 2^\ZZ$, we define
\[\qquad P^{(t)}_\lambda = \varphi\Big( \frac{ |\p_t|}{\lambda}\Big), \qquad C_\lambda = \varphi\Big(\frac{|i\p_t + \Delta|}{\lambda}\Big).\]
$P^{(t)}_\lambda$ localises the temporal Fourier support to the set $\{ \frac{\lambda}{2} < |\tau| < 2 \lambda\}$, and $C_\lambda$ localises the space-time Fourier support to distances $\approx \lambda$ from the paraboloid.

To restrict the Fourier support to larger sets, we use the notation
	\[P_{\les \lambda}=\sum_{\mu \in 2^{\ZZ},\mu \les \lambda} \varphi\Big( \frac{|\nabla|}{\mu}\Big), \; P^{(t)}_{\les \lambda}=\sum_{\mu \in 2^{\ZZ},\mu \les \lambda} \varphi\Big( \frac{|\p_t|}{\mu}\Big), \;C_{\les \lambda}=\sum_{\mu \in 2^{\ZZ},\mu \les \lambda} \varphi\Big( \frac{|i\p_t + \Delta|}{\mu}\Big),\]
and define $C_{>\mu} = I - C_{\les \mu}$. For ease of notation, for $\lambda \in 2^\NN$ we often use the shorthand $P_\lambda f = f_\lambda$. In particular, note that $u_1 = P_1 u$ has Fourier support in $\{|\xi| < 2\}$, and we have the identity
        $$ f = \sum_{\lambda \in 2^\NN} f_\lambda, \, \text{for any }\, f \in L^2(\RR^d).$$
For brevity, let us denote the frequently used decomposition into high and low modulation by
\EQ{
 \PN{\lambda } u := C_{\les (\frac{\lambda}{2^8})^2} P_\la u , \pq \PF{\lambda} u: = C_{ > ( \frac{\lambda}{2^8})^2} P_\la u,}
so that $u_\lambda=\PN{\lambda}u+\PF{\lambda}u$. Similarly, we take
$$ \PN{}:=\sum_{\lambda \in 2^\NN} P^N_\lambda, \qquad P^F = \sum_{\lambda \in 2^\NN} P^F_\lambda, \qquad P^F_{\les \lambda} = \sum_{\mu \in 2^\NN, \mu \les \lambda} P^F_\lambda, \qquad \text{etc.}$$
Note that $ u = P^N u + P^F u$, and these multipliers all obey the Schr\"odinger scaling, for instance
\EQ{
 (\PN{\la} u)(t/\la^2,x/\la)=\PN{2}(u(4t/\la^2,2x/\la)),}
where $\PN{2}$ is a space-time convolution with a Schwartz function, so that we can easily deduce that $\PN{\la}$ and $\PF{\la}$ are bounded on any $L^p_tL^q_x$ uniformly in $\la\in 2^\N$, and that $\PN{}$ and $\PF{}$ are bounded on any $L^2_tB^s_{q,2}$.

\subsection{Function spaces}\label{subsec:fs}
In the sequel, by default we consider tempered distributions. We define the inhomogeneous Besov spaces $B^s_{q, r}$ and Sobolev spaces $W^{s,p}$ via the norms
        $$ \| f \|_{B^s_{q,r}} = \Big( \sum_{\lambda \in 2^\NN} \lambda^{sr} \| f_\lambda \|_{L^q}^r \Big)^\frac{1}{r}, \qquad  \| f \|_{W^{s,p}} = \| \lr{\nabla}^s f \|_{L^p}. $$
We use the notation $2^*= \frac{2d}{d-2}$ and $2_* = (2^*)' = \frac{2d}{d+2}$ to denote the endpoint Strichartz exponents for the Schr\"odinger equation. Thus for $d \g 3$ we have
$$ \| e^{it\Delta} f \|_{L^\infty_t L^2_x \cap L^2_t L^{2^*}_x} + \Big\| \int_0^t e^{i(t-s)\Delta} F(s) ds \Big\|_{L^\infty_t L^2_x \cap L^2_t L^{2^*}_x} \lesa \| f\|_{L^2_x} + \| F \|_{L^2_t L^{2_*}_x}$$
by the (double) endpoint Strichartz estimate \cite{Keel1998}. To control the frequency localised Schr\"odinger component of the Zakharov evolution, we take parameters $s, a, b \in \RR$, $\lambda \in 2^\NN$ and define
$$ \| u \|_{S^{s,a, b}_\lambda} = \lambda^s \| u \|_{L^\infty_t L^2_x} + \lambda^{s-2a} \| (\lambda + |\p_t|)^a u \|_{L^2_t L^{2^*}_x} + \lambda^{s-1 + b} \Big\| \Big(\frac{ \lambda + |\p_t|}{\lambda^2 + |\p_t|}\Big)^a (i\p_t + \Delta) u \Big\|_{L^2_{t,x}} .$$

The parameters $a, b\in \RR$ are required to prove the bilinear estimates in the full admissible region \eqref{eqn:cond on s l}. Roughly speaking $a$ measures a loss of regularity in the small temporal frequency regime $|\tau| \ll \lr{\xi}$, for instance (when $b=0$) if $\supp \widetilde{u} \subset \{\tau \lesa \lr{\xi} \approx \lambda\}$ we have
    $$ \lambda^{s-2a} \| (\lambda + |\p_t|)^a u \|_{L^2_t L^{2^*}_x} + \lambda^{s-1 + b} \Big\| \Big(\frac{ \lambda + |\p_t|}{\lambda^2 + |\p_t|}\Big)^a (i\p_t + \Delta) u \Big\|_{L^2_{t,x}} \approx \lambda^{-a} \Big( \lambda^s \| u \|_{L^2_t L^{2^*}_x} + \lambda^{s-1} \| (i\p_t + \Delta) u \|_{L^2_{t,x}}\Big).$$
Thus, when the temporal frequencies are small, the non-$L^\infty_t H^s_x$ component of the norm $S^{s,a,b}_\lambda$ loses $\lambda^{-a}$ derivatives when compared to the standard scaling for the Schr\"odinger equation. On the other hand the $b$ parameter simply gives a gain in regularity in the high-modulation regime, for instance we have $\| P^F u \|_{L^\infty_t H^{s+b}_x} \lesa \| u \|_{S^{s,0,b}}$.

The choice of $a$ and $b$ will depend on $(s, \ell)$, there is some flexibility here, but one option is to choose
        \begin{equation}\label{eqn:choice of a,b}
         a = a^* := \begin{cases} \frac{3}{4}(s-\ell)-\frac{1}{2} &\text{ if }  s -\ell \g 1,  \\
                                        0 &\text{ if } s-\ell<1,  \end{cases}\qquad b=b^* := \begin{cases}0 &\text{ if } s-\ell>0, \\ \frac{1}{2}(\ell -s) + \frac{1}{2} &\text{ if }  s-\ell \les 0.   \end{cases}
        \end{equation}
Thus in the region $ \ell + 1 \les s \les \ell + 2$, when the Schr\"odinger component of the evolution is more regular, we require $a>0$ positive (depending on the size of $s-\ell$) and can take $b=0$. On the other hand,  in the ``balanced region'' $\ell< s< \ell + 1$ we can simply take $a=b=0$. In the final region $\ell -1 \les s \les \ell$, when the wave is more regular, we can take $a=0$ and require $b>0$ positive.

\begin{remark}
It is worth noting that due to the factor $(\lambda^2 + |\p_t|)^{-a} (\lambda + |\p_t|)^a$, the norm $\| \cdot \|_{S^{s,a,b}_\lambda}$ only controls the endpoint Strichartz estimate without loss when $a=0$. In particular, if $0\les a \les 1$, we only have
        \begin{equation}\label{eqn:stri with loss} \lambda^{s-a} \|  u_\lambda \|_{L^2_t L^{2^*}_x} \lesa \lambda^{s-2a} \| (\lambda + |\p_t|)^a u_\lambda \|_{L^2_t L^{2^*}_x} \lesa \| u_\lambda \|_{S^{s, a, 0}_\lambda}. \end{equation}
In view of the choice \eqref{eqn:choice of a,b}, this means that in the region $s-\ell \g 1$ we no longer have control over the endpoint Strichartz space $L^2_t W^{s, 2^*}_x$. On the other hand, in the small modulation regime, we retain control of the endpoint Strichartz space. More precisely, provided that $0\les a \les 1$, an application of Bernstein's inequality gives the characterisation
     \begin{equation}\label{eqn:norm chara}
        \| u_\lambda \|_{S^{s,a,b}_\lambda} \approx \lambda^s \Big(\| u_\lambda \|_{L^\infty_t L^2_x} + \| P^N_\lambda u \|_{L^2_t L^{2^*}_x}\Big) +  \lambda^{s-1 + b} \Big\| \Big(\frac{ \lambda + |\p_t|}{\lambda^2 + |\p_t|}\Big)^a (i\p_t + \Delta) u_\lambda \Big\|_{L^2_{t,x}}.
      \end{equation}
    \end{remark}

      To control the Schr\"odinger nonlinearity we take
      $$ \| F \|_{N^{s,a, b}_\lambda} =  \lambda^{s-2} \|P^{(t)}_{\les (\frac{\lambda}{2^8})^2} F \|_{L^\infty_t L^2_x}+\lambda^s  \| C_{\les (\frac{\lambda}{2^8})^2} F \|_{L^2_t L^{2_*}_x} + \lambda^{s-1 + b} \Big\| \Big(\frac{ \lambda + |\p_t|}{\lambda^2 + |\p_t|}\Big)^a F \Big\|_{L^2_{t,x}}.$$
      \begin{remark}
        In the special case $0\les a<\frac12$ we have
        \begin{equation}\label{eqn:N norm chara}\| F_\lambda \|_{N^{s,a, b}_\lambda}\approx \lambda^s  \| C_{\les (\frac{\lambda}{2^8})^2} F_\lambda \|_{L^2_t L^{2_*}_x} + \lambda^{s-1 + b} \Big\| \Big(\frac{ \lambda + |\p_t|}{\lambda^2 + |\p_t|}\Big)^a F_\lambda \Big\|_{L^2_{t,x}}.
        \end{equation}
        To see this, let $\frac{1}{r} = \frac{1}{2} - a$ and apply Bernstein's inequality together with the Sobolev embedding to obtain
    \begin{align*}
     \lambda^{s-2} \| P^{(t)}_{\les (\frac{\lambda}{2^8})^2} F_\lambda \|_{L^\infty_t L^2_x} &\lesa \lambda^{ s-2+\frac{2}{r}} \| P^{(t)}_{\les (\frac{\lambda}{2^8})^2}   F_\lambda \|_{L^2_x L^r_t} \\
      &\lesa \lambda^{ s-1-2a} \| (\lambda + |\p_t|)^{a} P^{(t)}_{\les (\frac{\lambda}{2^8})^2}   F_\lambda \|_{L^2_{t,x}}
    \end{align*}
       which implies the claim, since $b \g 0$.
\end{remark}

We also require a suitable space in which to control the evolution of the wave component. To this end, for $\ell, \alpha, \beta \in \RR$, we let
    $$ \| V \|_{W^{\ell, \alpha, \beta }_\lambda} = \lambda^\ell \| V \|_{L^\infty_t L^2_x} + \lambda^{\ell -\alpha} \| (\lambda + |\p_t|)^\alpha P^{(t)}_{\les (\frac{\lambda}{2^8})^2} V \|_{L^\infty_t L^2_x} + \lambda^{\beta-1} \| (i\p_t + |\nabla|) V \|_{L^2_{t,x}}.$$
Thus for small temporal frequencies we essentially take $(\frac{\lambda + |\p_t|}{\lambda})^\alpha V \in L^\infty_t H^\ell_x$, while for large temporal frequencies (in the Schr\"odinger like regime) the wave component $V$ has roughly $\beta$ derivatives. Eventually we will take $\alpha=a$ and $\beta=s-\frac{1}{2}$. Consequently, in the high temporal frequency regime, the wave component $V$ essentially inherits the regularity of the Schr\"odinger evolution $u$.
To bound the right-hand side of the half-wave equation at frequency $\lambda$, we define
\[\|G\|_{R^{\ell,\alpha,\beta}_\lambda}=\lambda^{\ell -2} \| G\|_{L^\infty_t L^2_x} + \lambda^{\ell-\alpha} \| ( |\p_t| + \lambda)^\alpha P^{(t)}_{\les (\frac{\lambda}{2^8})^2} G \|_{L^1_t L^2_x} + \lambda^{\beta-1} \|  G \|_{L^2_{t,x}}.\]

\begin{lemma}[Nested embeddings]\label{lem:embed prop} Let $s, a, a', b, b' \in \RR$ with $a' \les a$ and $b' \les b$. Then         $$ \| u_\lambda \|_{S^{s,a,b}_\lambda} \lesa \| u_\lambda \|_{S^{s,a', b}_\lambda}, \qquad \| u_\lambda \|_{S^{s,a,b'}_\lambda} \les \| u_\lambda \|_{S^{s,a,b}_\lambda}. $$ Similarly, if $\ell, \alpha, \alpha', \beta, \beta' \in \RR$ with $\alpha'\les \alpha$ and $\beta' \les \beta$ we have         $$ \| V_\lambda \|_{W^{\ell, \alpha', \beta}_\lambda} \lesa\| V_\lambda \|_{W^{\ell, \alpha, \beta}_\lambda}, \qquad \| V_\lambda \|_{W^{\ell, \alpha, \beta'}_\lambda} \les \| V_\lambda \|_{W^{\ell, \alpha, \beta}_\lambda}. $$ \end{lemma}
\begin{proof}
The first claim follows from the characterisation \eqref{eqn:norm chara}.
The remaining inequalities are clear from the definitions.
\end{proof}

To control the evolution of the full solution, we sum the dyadic terms in $\ell^2$, and define the norms
        $$ \| u \|_{S^{s,a,b}} = \Big( \sum_{\lambda \in 2^\NN} \| u_\lambda \|_{S^{s,a, b}_\lambda}^2 \Big)^\frac{1}{2}, \qquad \| F \|_{N^{s,a, b}} = \Big( \sum_{\lambda \in 2^\NN} \| F_\lambda \|_{N^{s, a, b}_\lambda}^2 \Big)^\frac{1}{2} $$
and
        $$\| V \|_{W^{\ell, \alpha,\beta}} = \Big( \sum_{\lambda \in 2^\NN} \| V_\lambda \|_{W^{\ell, \alpha,\beta}_\lambda}^2 \Big)^\frac{1}{2},\qquad \| G \|_{R^{\ell,\alpha,\beta}} = \Big( \sum_{\lambda \in 2^\NN} \| G_\lambda \|_{R^{\ell,\alpha,\beta}_\lambda}^2 \Big)^\frac{1}{2}+\|G_{\les 2^{16}}\|_{L^1_tL^2_x}.$$
Then, we define the corresponding spaces as the  collection of all tempered distributions with finite norm.

Let $I\subset \RR$ be an open interval, i.e. a connected open subset of the real line $\RR$. We localise the norms and spaces to time intervals $I \subset \RR$ via restriction norms. For instance, we define the restriction norm
$$ \| u \|_{S^{s,a,b}(I)} = \inf_{ u'\in S^{s,a,b} \text{ and } u'|_I = u} \| u'\|_{S^{s,a,b}},$$
provided that such an extension $u'\in S^{s,a, b}$ exists.
The norms  $\| \cdot \|_{N^{s,a,b}(I)}$, $\|\cdot \|_{W^{\ell,\alpha,\beta} (I)}$, and $\| \cdot \|_{R^{\ell,\alpha,\beta}(I)}$ and the corresponding spaces are defined similarly.

\subsection{Duhamel formulae and energy inequalities.}\label{subsec:duhamel}

The  solution operator for the inhomogeneous Schr\"odinger equation is denoted by
        $$\mc{I}_0[F](t) = -i \int_0^t e^{i(t-s) \Delta} F(s) ds .$$
For a general potential $V \in L^\infty_t L^2_x$, we let
        $$ \mc{I}_V[F](t) = - i \int_0^t \mc{U}_V(t,s) F(s) ds$$
where $\mc{U}_V(t,s) f$ denotes the homogeneous solution operator for the Cauchy problem
        $$  ( i \p_t + \Delta - \Re{V}) u = 0, \qquad u(s) = f. $$
        We show later that the operators $\mc{U}_V$ and $\mc{I}_V$ are well-defined on suitable function spaces, provided only that $V \approx  e^{it\Delta} f \in L^\infty_t H^{\frac{d-4}{2}}_x$, i.e. $V$ is close to a $L^\infty_t H^{\frac{d-4}{2}}_x$ solution to the wave equation.

        Similarly, we define the solution operator for the inhomogeneous half-wave equation by
        \[
\mc{J}_0[F](t) =-i \int_0^t e^{i(t-s) |\nabla|} F(s) ds.
\]

We record here two straightforward energy inequalities which we exploit in the sequel.

\begin{lemma}\label{lem:energy ineq}
  Let $s \in \RR $, $0\les a, b \les 1$.  For any $\lambda \in 2^\NN$ we have
\[\|e^{it \Delta} f_\lambda\|_{S^{s,a,b}_\lambda}\lesa \lambda^s \|f_\lambda\|_{L^2_x} \]
  and
          $$ \| \mc{I}_0[F_\lambda] \|_{S^{s,a,b}_\lambda} \lesa   \| F_\lambda \|_{N^{s, a, b}_\lambda}. $$
Moreover, if $0\in I \subset \RR$ is an open interval and $F \in N^{s,a,b}(I)$, then $\mc{I}_0[F] \in C(I, H^s)$.
\end{lemma}
\begin{proof}
  The estimate for the free solutions follows from the fact that the temporal frequency is of size $\lambda^2$ and the endpoint Strichartz estimate.

In order to prove the estimate for the Duhamel term, in view of the characterisation \eqref{eqn:norm chara} it suffices to bound the high-modulation contribution $ \lambda^s \| \mc{I}_0[ P^F_\lambda F] \|_{L^\infty_t L^2_x}$ due to the (double) endpoint Strichartz estimate. To this end, we first claim that for any $\mu>0$ and $G \in L^\infty_t L^2_x$ we have
    \begin{equation}\label{eqn:lem energy ineq:Linf bound} \| \mc{I}_0[ C_{> \mu} G ] \|_{L^\infty_t L^2_x} \lesa \mu^{-1} \| C_{> \mu} G \|_{L^\infty_t L^2_x}. \end{equation}
Assuming \eqref{eqn:lem energy ineq:Linf bound} for the moment, we conclude that
        \begin{equation}\label{eqn:lem energy ineq:low temp freq} \| \mc{I}_0[C_{> ( \frac{\lambda}{2^8})^2} F_\lambda] \|_{L^\infty_t L^2_x} \lesa \lambda^{-2} \| C_{> ( \frac{\lambda}{2^8})^2} F_\lambda \|_{L^\infty_t L^2_x}. \end{equation}
To improve this, we again use \eqref{eqn:lem energy ineq:Linf bound} and observe that
    \begin{align*}
        \| \mc{I}_0[ P^{(t)}_{> (\frac{\lambda}{2^8})^2} C_{> ( \frac{\lambda}{2^8})^2} F_\lambda] \|_{L^\infty_t L^2_x}
                        &\les  \sum_{\nu \gtrsim \lambda^2}
        \| \mc{I}_0[ P^{(t)}_{\nu} C_{\approx \nu} C_{> ( \frac{\lambda}{2^8})^2} F_\lambda] \|_{L^\infty_t L^2_x} \\
                        &\lesa  \sum_{ \nu \gtrsim \lambda^2 } \nu^{-1} \| P^{(t)}_{\nu} C_{> ( \frac{\lambda}{2^8})^2} F_\lambda \|_{L^\infty_t L^2_x} \\
                        &\lesa  \sum_{ \nu \gtrsim \lambda^2 } \nu^{ -\frac{1}{2}} \| P^{(t)}_{\nu} C_{> ( \frac{\lambda}{2^8})^2} F_\lambda \|_{L^2_{t,x}}  \lesa \lambda^{-s-b} \| F_\lambda \|_{N^{s, a,b}_\lambda}.
    \end{align*}
Hence the claimed inequality follows.

To complete the proof of the norm bounds, it only remains to verify the claimed bound \eqref{eqn:lem energy ineq:Linf bound}. Define $H(t) = (\p_t^{-1} P^{(t)}_{>\mu}[ e^{- it \Delta} G])(t)$. A computation gives the $L^\infty_t L^2_x$ bound
  $$
        \| H \|_{L^\infty_t L^2_x} \lesa \mu^{-1} \| e^{-it\Delta} G \|_{L^\infty_t L^2_x} = \mu^{-1} \| G \|_{L^\infty_t L^2_x}
   $$
and, since $C_{>\mu} G = e^{it\Delta} P^{(t)}_{>\mu} [e^{-it \Delta} G]$, the identity
   $$
    \p_t H(t) = P^{(t)}_{>\mu}[ e^{-it\Delta}G] (t) = e^{ - i t \Delta} C_{>\mu} G.
   $$
Therefore the bound \eqref{eqn:lem energy ineq:Linf bound} follows by writing $ \mc{I}_0[ C_{> \mu} G](t)= \mc{I}_0[e^{  i t \Delta} \p_t H](t) = -ie^{ i t \Delta} ( H(t) - H(0) )$.

We now turn to the proof of continuity. In view of the definition of
the time restricted space $N^{s,a,b}(I)$, it suffices to
consider the case $I=\RR$. Moreover, the norm bound proved implies
that it is enough to prove that if $\lambda \in 2^\NN$ and $F_\lambda
\in N^{0,a,b}$ then  $\mc{I}_0[F_\lambda] \in C(I,
L^2)$. If $\|F_\lambda\|_{N^{s,a,b}_\lambda}<\infty$ for $a,b \g 0$, then
$F_\lambda\in L^1_{t,loc}L^2_x$ and the continuity
follows from the dominated convergence theorem.
\end{proof}

The energy inequality has the following useful consequence.

\begin{lemma}\label{lem:scattering}
Let $s, a, b \in \RR $, $b \g 0$. If $F\in N^{s,a,b}$ then
    $$ \lim_{\substack{t, t' \to \infty}} \Big\| \int_t^{t'} e^{-is\Delta} F(s) ds \Big\|_{H^s} = 0. $$
\end{lemma}
\begin{proof}
After writing $\int_t^{t'} e^{ - i s\Delta} F(s) ds = e^{-it'\Delta} \mc{I}_0[F](t') - e^{-it\Delta}\mc{I}_0[F](t)$, the energy inequality in Lemma \ref{lem:energy ineq} implies that it suffices to prove that for every $\lambda \in 2^\NN$ we have
    $$ \lim_{t, t' \to \infty} \Big\| \int_t^{t'} e^{-is\Delta} F_\lambda(s) ds \Big\|_{L^2_x} = 0. $$
We decompose into low and high modulation contributions $F_\lambda = P^N_\lambda F + P^F_\lambda F$. For the former term, we observe that the endpoint Strichartz estimate gives
    $$ \Big\| \int_t^{t'} e^{-is\Delta} P^N_\lambda F(s) ds \Big\|_{L^2_x} \lesa \| P^N_\lambda F \|_{L^2_t L^{2_*}_x( (t, t')\times \RR^d)}$$
which  vanishes as $t, t' \to \infty$ since $P^N_\lambda F \in L^2_t L^{2_*}_x$. For the remaining high modulation contribution $P^F_\lambda F$, we let $G(t) = \p_t^{-1} P^{(t)}_{\gtrsim \lambda^2} ( e^{-it\Delta} P^F_\lambda F)$. Then $e^{-it\Delta} P^F_\lambda F = \p_t G$ and therefore an application of Sobolev embedding gives, uniformly for $M \g 1$,
    \begin{align*}
      \Big\| \int_t^{t'} e^{-is\Delta} P^F_\lambda F(s) ds \Big\|_{L^2_x} &= \| G(t') - G(t) \|_{L^2_x} \\
                            &\les \| (P^{(t)}_{\les M} G)(t') \|_{L^2_x} + \| (P^{(t)}_{\les M} G)(t) \|_{L^2_x}+\| (P^{(t)}_{> M} G)(t') \|_{L^2_x}+\| (P^{(t)}_{> M} G)(t) \|_{L^2_x}\\
                            &\lesa_\lambda \| (P^{(t)}_{\les M} G)(t') \|_{L^2_x} + \| (P^{(t)}_{\les M} G)(t) \|_{L^2_x} + \| C_{> M} P^F_\lambda F \|_{L^2_{t,x}}.
    \end{align*}
Since $\widetilde{P^{(t)}_{\les M} G} \in L^1_{\tau} L^2_\xi$ and $P^F_\lambda F \in L^2_{t,x}$, for any $\epsilon>0$, by choosing $M$ sufficiently large, and letting $t, t' \to \infty$ the Riemann-Lebesgue lemma implies that
    $$ \limsup_{t, t' \to \infty} \Big\| \int_t^{t'} e^{-is\Delta} P^F_\lambda F(s) ds \Big\|_{L^2_x}  \les \epsilon. $$
As this holds for every $\epsilon>0$, result follows.
\end{proof}

We also require an energy type inequality for the wave equation.

\begin{lemma}\label{lem:energy ineq wave}
Let $0\les \alpha \les 1$ and $\beta, \ell \in \RR$.
Then, for all $\lambda \in 2^{\NN}$,
$$ \| e^{it|\nabla|}g_\lambda \|_{W^{\ell, \alpha, \beta }_\lambda} \lesa \lambda^\ell \|g_\lambda \|_{L^2}, $$
and for $\lambda>2^{16}$,
$$ \| \mc{J}_0[G_\lambda] \|_{W^{\ell, \alpha, \beta }_\lambda} \lesa  \|G_\lambda\|_{R^{\ell,\alpha,\beta}_\lambda}.$$
Moreover, if $0\in I \subset \RR$ is an open interval and $G \in R^{\ell,\alpha,\beta}(I)$, then $\mc{J}_0[G] \in C(I, H^\ell)$.
      \end{lemma}
      \begin{proof}
        The estimate for free solutions follows from the fact that their temporal frequencies are of size $\lambda$.

        For the Duhamel integral we have
\[
\lambda^\ell \|\mc{J}_0[P_{\ll
  \lambda^2}^{(t)}G_\lambda]\|_{L^\infty_t L^2_x}\lesa \lambda^\ell\|P_{\ll
  \lambda^2}^{(t)}G_\lambda \|_{L^1_t L^2_x}\lesa \lambda^{\ell-\alpha} \|(\lambda+|\partial_t|)^\alpha P_{\ll
  \lambda^2}^{(t)}G_\lambda \|_{L^1_t L^2_x}.
\]
Similarly to \eqref{eqn:lem energy ineq:Linf bound} above we also
obtain, for $\lambda>2^{16}$,
\[
\lambda^\ell \|\mc{J}_0[P^{(t)}_{\gtrsim
  \lambda^2}G_\lambda]\|_{L^\infty_t L^2_x}\lesa \lambda^{\ell-2}\|P^{(t)}_{\gtrsim
  \lambda^2}G_\lambda\|_{L^\infty_t L^2_x},
\]
and deduce
\begin{equation}\label{eqn:inf-bound-wave}\lambda^\ell \|\mc{J}_0[G_\lambda]\|_{L^\infty_t L^2_x}\lesa \lambda^{\ell-\alpha} \|(\lambda+|\partial_t|)^\alpha P_{\ll
  \lambda^2}^{(t)}G_\lambda \|_{L^1_t L^2_x}+\lambda^{\ell-2} \|P^{(t)}_{\gtrsim
  \lambda^2}G_\lambda\|_{L^\infty_t L^2_x}.
\end{equation}
Since the bound for the $L^2_{t,x}$ component of the norm $\| \cdot\|_{W^{\ell, \alpha, \beta}}$ follows directly from the definition, it only remains to bound
    \begin{equation}\label{eqn:inf-wave-decomp}
        \lambda^{\ell-\alpha} \| (\lambda+|\p_t|)^\alpha P^{(t)}_{\ll \lambda^2}\mc{J}_0[ G_\lambda ] \|_{L^\infty_t L^2_x} \lesa \lambda^\ell \| P^{(t)}_{\lesa \lambda} \mc{J}_0[ G_\lambda ]\|_{L^\infty_t L^2_x} + \lambda^{\ell-\alpha} \| (\lambda+|\p_t|)^\alpha  P^{(t)}_{\ll \lambda^2} P^{(t)}_{\gg \lambda}\mc{J}_0[ G_\lambda ] \|_{L^\infty_t L^2_x}.
    \end{equation}
The first term on the righthand side of \eqref{eqn:inf-wave-decomp} can be bounded directly from \eqref{eqn:inf-bound-wave}. We turn to the second contribution in \eqref{eqn:inf-wave-decomp} and write
\[
P^{(t)}_{\ll \lambda^2} P^{(t)}_{\gg\lambda}\mc{J}_0[G_\lambda]=P^{(t)}_{\ll \lambda^2} P^{(t)}_{\gg\lambda}\mc{J}_0[P^{(t)}_{\ll \lambda^2} G_\lambda],
\]
where the  identity is due to the fact that
$d_\lambda=P^{(t)}_{\ll \lambda^2} P^{(t)}_{\gg\lambda}\mc{J}_0[P^{(t)}_{\gtrsim \lambda^2} G_\lambda]$ solves $(i\partial_t +|\nabla|)d_\lambda=0$,
therefore
$
d_\lambda=e^{it |\nabla|}d_\lambda(0)
$
and since $d_\lambda$ has temporal frequencies $\gg \lambda $ it must
vanish identically.

Let $e_\lambda:=\mc{J}_0[P_{\ll \lambda^2}^{(t)}G_\lambda]$ and $f_\lambda:=\mc{J}_0[(\lambda+|\partial_t|)^\alpha P_{\ll \lambda^2}^{(t)}G_\lambda]$. Then, $(i\partial_t +|\nabla|)((\lambda+|\partial_t|)^\alpha e_\lambda-f_\lambda)=0$, therefore
\[
(\lambda+|\partial_t|)^\alpha e_\lambda-f_\lambda=e^{it|\nabla|}z_\lambda, \quad z_\lambda=\big((\lambda+|\partial_t|)^\alpha e_\lambda-f_\lambda\big)\big|_{t=0}.
\]
Again, since the temporal frequencies of $e^{it|\nabla|}z$ are $\approx \lambda$, we conclude that
\[
(\lambda+|\partial_t|)^\alpha P^{(t)}_{\gg\lambda}\mc{J}_0[P_{\ll \lambda^2}^{(t)}G_\lambda]= P_{\gg \lambda}^{(t)}\mc{J}_0[(\lambda+|\partial_t|)^\alpha P_{\ll \lambda^2}^{(t)}G_\lambda],
\]
hence
\[
\|(\lambda+|\partial_t|)^\alpha P_{\gg \lambda}^{(t)}P_{\ll \lambda^2}^{(t)}\mc{J}_0[G_\lambda]\|_{L^\infty_t L^2_x}\lesa \| \mc{J}_0[(\lambda+|\partial_t|)^\alpha P_{\ll \lambda^2}^{(t)}G_\lambda]\|_{L^\infty_t L^2_x}\lesa \|(\lambda+|\partial_t|)^\alpha P_{\ll \lambda^2}^{(t)}G_\lambda\|_{L^1_t L^2_x}.
\]
Concerning the continuity, we observe that if $\|G_\lambda\|_{R^{0,a,b}_\lambda}<\infty$ for $a,b \g 0$, then
$G_\lambda\in L^1_{t, loc}L^2_x$ and the continuity follows from the dominated convergence theorem as in Lemma \ref{lem:energy ineq}.
\end{proof}

\subsection{A product estimate for fractional derivatives}\label{subsec:prod-est}

The definition of the norms $\| \cdot \|_{S^{s, a, b}_\lambda}$ involves three distinct regions of temporal frequencies, the low modulation case $|\tau +|\xi|^2| \ll \lambda^2$, the medium modulation case $|\tau| \ll \lambda^2$, and the high modulation case $|\tau| \gg \lambda^2$. When estimating bilinear quantities, this leads to a large number of possible frequency interactions. To help alleviate the number of possible cases we have to consider, we prove the following bilinear estimate which we later exploit as a black box.

\begin{lemma}\label{lem:elementary product est}
Let $a \in \RR$, $\mu>0$, and $1\les \tilde{p}, \tilde{q}, \tilde{r}, p, q, r \les \infty$ with $\frac{1}{p} = \frac{1}{q} + \frac{1}{r}$ and $\frac{1}{\tilde{p}} = \frac{1}{\tilde{q}} + \frac{1}{\tilde{r}}$. Then
   \begin{align*} \| ( \mu +  |\p_t|)^a ( v u ) \|_{L^{\tilde{p}}_t L^p_x}
        \lesa \mu^{-|a|} \| ( \mu + |\p_t|)^{|a|} v \|_{L^{\tilde{r}}_t L^r_x}\| ( \mu + |\p_t|)^a u \|_{L^{\tilde{q}}_t L^q_x}.
    \end{align*}
\end{lemma}
\begin{proof}
The proof is essentially well-known, and thus we shall be somewhat brief. The main obstruction is that we allow the endpoint case $\tilde{r}=\infty$, and, as we are working with fractional derivatives in time, this causes the usual difficulties due to the failure of the Littlewood-Paley theory. In particular, to avoid summation issues, we closely follow the proof of the endpoint Kato-Ponce type inequality contained in \cite{Bourgain2014}.

To simplify notation, and in contrast to the rest of the paper, we temporarily adopt the convention that the temporal frequency multipliers $P^{(t)}_\nu$ give an inhomogeneous decomposition over $\nu \in 2^\NN$, thus
    $$ P^{(t)}_1  = \sum_{\lambda \in 2^\ZZ, \lambda \les 1}  \varphi\Big( \frac{|\p_t|}{\lambda}\Big), \qquad f = \sum_{\nu \in 2^\NN} P^{(t)}_\nu f $$
where $\varphi$ is as in Subsection \ref{subsec:fm}.

We first consider the case $a>0$ and prove the stronger estimate
       \begin{equation}\label{eqn:lem elem prod:a>0 strong}
            \|(1+|\p_t|)^a(vu) \|_{L^{\tilde{p}}_t L^p_x} \lesa \| v \|_{L^{\tilde{r}}_t L^r_x} \| ( 1 + |\p_t|)^a u \|_{L^{\tilde{q}}_t L^q_x} + \| ( 1 + |\p_t|)^a v \|_{L^{\tilde{r}}_t L^r_x} \| u \|_{L^{\tilde{q}}_t L^q_x}.
       \end{equation}
Clearly, after rescaling, this implies the required estimate in the case $a>0$. The proof of the estimate \eqref{eqn:lem elem prod:a>0 strong} is a straightforward adaption of the argument given in \cite{Bourgain2014}. In more detail, we decompose
    $$ vu = \sum_{\nu\in 2^\NN} P^{(t)}_\nu v P^{(t)}_{\les \nu} u + \sum_{\nu\in 2^\NN} P^{(t)}_{<\nu} v P^{(t)}_{\nu} u. $$
By symmetry, it is enough to consider the first term. To deal with the problem of summation over frequencies, we introduce a commutator term and write
     \begin{align} \sum_{\nu \in 2^\NN} (1+|\p_t|)^a\big(P^{(t)}_\nu v P^{(t)}_{\les \nu} u\big)
        &= \sum_{\nu \in 2^\NN} \Big[ (1+|\p_t|)^a\big(P^{(t)}_\nu v P^{(t)}_{\les \nu} u\big) - \big((1+|\p_t|)^a P^{(t)}_\nu v\big) P^{(t)}_{\les \nu} u\Big] \notag \\
        &\qquad \qquad \qquad + \sum_{\nu \in 2^\NN} \big( (1+|\p_t|)^a P^{(t)}_\nu v\big) P^{(t)}_{\les \nu} u\notag \\
        &= \sum_{\nu \in 2^\NN} \Big[ (1+|\p_t|)^a\big(P^{(t)}_\nu v P^{(t)}_{\les \nu} u\big) - \big((1+|\p_t|)^a P^{(t)}_\nu v\big) P^{(t)}_{\les \nu} u\Big] \notag \\
        &\qquad \qquad \qquad + \big((1+|\p_t|)^a v\big) u  + \sum_{\nu \in 2^\NN} \big( (1+|\p_t|)^a P^{(t)}_{\nu} v\big) P^{(t)}_{> \nu} u. \label{eqn:lem elem prod:a>0 decomp}
     \end{align}
The bound for the second term in \eqref{eqn:lem elem prod:a>0 decomp} follows directly from H\"older's inequality. To bound the third term in \eqref{eqn:lem elem prod:a>0 decomp}, we note that for any $M \in 2^\NN$ we have
    \begin{align*}
      \sum_{\nu\in 2^\NN} \big\| ( 1 + |\p_t|)^a P^{(t)}_\nu v P^{(t)}_{> \nu} u \big\|_{L^{\tilde{p}}_t L^p_x} &\lesa \sum_{\nu\les M} \nu^a \| v \|_{L^{\tilde{r}}_t L^r_x} \| u \|_{L^{\tilde{q}}_t L^q_x} + \sum_{\nu>M} \nu^{-a} \| ( 1 +|\p_t|)^a v \|_{L^{\tilde{r}}_t L^r_x} \| ( 1 + |\p_t|)^a  u \|_{L^{\tilde{q}}_t L^q_r} \\
      &\lesa M^a \| v \|_{L^{\tilde{r}}_t L^r_x} \| u \|_{L^{\tilde{q}}_t L^q_x} + M^{-a} \| ( 1 +|\p_t|)^a v \|_{L^{\tilde{r}}_t L^r_x} \| ( 1 + |\p_t|)^a  u \|_{L^{\tilde{q}}_t L^q_r}.
    \end{align*}
Optimising in $M$ then gives
    $$ \Big\| \sum_{\nu\in 2^\NN} ( 1 + |\p_t|)^a P^{(t)}_\nu v P^{(t)}_{> \nu} u \Big\|_{L^{\tilde{p}}_t L^p_x}
                    \lesa \Big( \| v \|_{L^{\tilde{r}}_t L^r_x} \| u \|_{L^{\tilde{q}}_t L^q_x} \| ( 1 +|\p_t|)^a v \|_{L^{\tilde{r}}_t L^r_x} \| ( 1 + |\p_t|)^a  u \|_{L^{\tilde{q}}_t L^q_r}\Big)^{\frac{1}{2}} $$
and hence \eqref{eqn:lem elem prod:a>0 strong} follows for the third term in \eqref{eqn:lem elem prod:a>0 decomp}. Finally, to bound the first term in \eqref{eqn:lem elem prod:a>0 decomp}, we first claim that for any $0<\theta < \frac{1}{a}$ we have the commutator estimates
     \begin{equation}\label{eqn:lem elem prod:commutator}
        \begin{split}
        \big\| (1+|\p_t|)^a\big(P^{(t)}_\nu v P^{(t)}_{\les \nu} u\big) - & \big((1+|\p_t|)^a P^{(t)}_\nu v\big) P^{(t)}_{\les \nu} u\big\|_{L^{\tilde{p}}_t L^p_x} \\
        &\lesa \nu^{- \theta a} \| (1+|\p_t|)^a v \|_{L^{\tilde{r}}_t L^r_x} \| (1+|\p_t|)^a u \|_{L^{\tilde{q}}_t L^q_x}^\theta \| u \|_{L^{\tilde{q}}_t L^q_x}^{1-\theta}
        \end{split}
     \end{equation}
and
     \begin{equation}\label{eqn:lem elem prod:commutator II}
        \big\| (1+|\p_t|)^a\big(P^{(t)}_\nu v P^{(t)}_{\les \nu} u\big) - \big((1+|\p_t|)^a P^{(t)}_\nu v\big) P^{(t)}_{\les \nu} u\big\|_{L^{\tilde{p}}_t L^p_x} \lesa \nu^{a} \| v \|_{L^{\tilde{r}}_t L^r_x} \|  u \|_{L^{\tilde{q}}_t L^q_x}.
     \end{equation}
Assuming these bounds for the moment, we then have for any $M \in 2^\NN$
    \begin{align*}
      \sum_{\nu\in 2^\NN} \big\| (1&+|\p_t|)^a\big(P^{(t)}_\nu v P^{(t)}_{\les \nu} u\big) - \big((1+|\p_t|)^a P^{(t)}_\nu v\big) P^{(t)}_{\les \nu} u \big\|_{L^{\tilde{p}}_t L^p_x}\\
       &\lesa \sum_{\nu\les M} \nu^a \| v \|_{L^{\tilde{r}}_t L^r_x} \|  u \|_{L^{\tilde{q}}_t L^q_x} + \sum_{\nu>M} \nu^{-\theta a} \| (1+|\p_t|)^a v \|_{L^{\tilde{r}}_t L^r_x} \| (1+|\p_t|)^a u \|_{L^{\tilde{q}}_t L^q_x}^\theta \| u \|_{L^{\tilde{q}}_t L^q_x}^{1-\theta} \\
      &\lesa M^a \| v \|_{L^{\tilde{r}}_t L^r_x} \|  u \|_{L^{\tilde{q}}_t L^q_x}  + M^{-\theta a} \| (1+|\p_t|)^a v \|_{L^{\tilde{r}}_t L^r_x} \| (1+|\p_t|)^a u \|_{L^{\tilde{q}}_t L^q_x}^\theta \| u \|_{L^{\tilde{q}}_t L^q_x}^{1-\theta}.
    \end{align*}
Optimising in $M$, we conclude that
    $$ \Big\| \sum_{\nu\in 2^\NN} P^{(t)}_\nu v P^{(t)}_{\les \nu } u \Big\|_{L^{\tilde{p}}_t L^p_x}
                \lesa \Big(\| (1+|\p_t|)^a v \|_{L^{\tilde{r}}_t L^r_x} \| u \|_{L^{\tilde{q}}_t L^q_x} \Big)^{\frac{1}{1+\theta}} \Big( \| v \|_{L^{\tilde{r}}_t L^r_x} \| (1+|\p_t|)^a u \|_{L^{\tilde{q}}_t L^q_x}\Big)^{1-\frac{1}{1+\theta}}$$
and hence \eqref{eqn:lem elem prod:a>0 strong} follows. It only remains to prove the standard commutator bounds \eqref{eqn:lem elem prod:commutator} and \eqref{eqn:lem elem prod:commutator II}. We begin by noting that for any $a\in \RR$, we have the related estimate
    \begin{equation}\label{eqn:lem elem prod:commutator III}
         \big\| (1+|\p_t|)^a\big(P^{(t)}_\nu v P^{(t)}_{\ll \nu} u\big) - \big((1+|\p_t|)^a P^{(t)}_\nu v\big) P^{(t)}_{\ll \nu} u\big\|_{L^{\tilde{p}}_t L^p_x}
                \lesa \nu^{a-1} \| P^{(t)}_\nu v \|_{L^{\tilde{r}}_t L^r_x} \| \p_t P^{(t)}_{\ll \nu} u \|_{L^{\tilde{q}}_t L^q_x}
    \end{equation}
which follows by writing
    \begin{align*}
        (1+|\p_t|)^a\big(P^{(t)}_\nu v P^{(t)}_{\ll \nu} u\big) - & \big((1+|\p_t|)^a P^{(t)}_\nu v\big) P^{(t)}_{\ll \nu} u  \\
                &= \nu^a \int_\RR \nu \, \psi_1(s\nu) \big( P^{(t)}_\nu v\big)(t-s) \Big( \big( P^{(t)}_{\ll \nu} u\big)(t-s) - \big(P^{(t)}_{\ll \nu} u\big)(t)\Big) ds \\
                &= - \nu^{a-1} \int_\RR \int_0^1 \nu \, \psi_2( s \nu) \big( P^{(t)}_\nu v\big)(t-s) \big( \p_t P^{(t)}_{\ll \nu} u\big)(t-ss') ds' ds
    \end{align*}
for some $\psi_1 \in \mc{S}(\RR)$ (i.e. some smooth rapidly decreasing kernel independent of $\nu$, $u$, and $v$), $\psi_2(s)=s\psi_1(s)$, and so applying H\"older's inequality and using translation invariance, we obtain \eqref{eqn:lem elem prod:commutator III}. To conclude the proof of \eqref{eqn:lem elem prod:commutator}, we note that if $a >0$, then \eqref{eqn:lem elem prod:commutator III} also holds with $P^{(t)}_{\ll \nu} u$ replaced with $P^{(t)}_{\les \nu} u$ (this is simply another application of H\"older and Bernstein), and hence \eqref{eqn:lem elem prod:commutator} follows from the interpolation type bound
        \begin{align*} \| \p_t P^{(t)}_{\les \nu} u \|_{L^{\tilde{q}}_t L^q_x} \les \sum_{\nu'\in 2^\NN, \nu'\les \nu} \nu' \| P^{(t)}_{\nu'} u \|_{L^{\tilde{q}}_t L^q_x}
                &\lesa \sum_{\nu' \in 2^\NN, \nu'\les \nu} (\nu')^{1-\theta a} \|(1 +|\p_t|)^a u \|_{L^{\tilde{q}}_t L^q_x}^\theta \| u \|_{L^{\tilde{q}}_t L^q_x}^{1-\theta} \\
                &\lesa \nu^{1-\theta a} \| (1 +|\p_t|)^a u \|_{L^{\tilde{q}}_t L^q_x}^\theta \| u \|_{L^{\tilde{q}}_t L^q_x}^{1-\theta}
        \end{align*}
which holds for any $0\les \theta < 1/a$. Finally, the second commutator bound \eqref{eqn:lem elem prod:commutator II} follows by simply discarding the commutator structure and applying H\"older and Bernstein's inequalities. This completes the proof of \eqref{eqn:lem elem prod:a>0 strong} and hence the required estimate holds in the case $a>0$.

It only remains to consider the case $a<0$, but this follows by arguing via duality. Namely, the estimate \eqref{eqn:lem elem prod:a>0 strong} gives
    \begin{align*}
       &\| (1+|\p_t|)^a ( v u) \|_{L^{\tilde{p}}_t L^p_x} \\
       ={}& \sup_{\| w \|_{L^{\tilde{p}'}_t L^{p'}_x}\les 1} \Big| \int_{\RR^{1+d}} \big( (1+|\p_t|)^a w \big) v u dx dt \Big| \\
       \les{}& \| (1+|\p_t|)^a u \|_{L^{{\tilde{q}}}_t L^q_x} \sup_{\| w \|_{L^{\tilde{p}'}_t L^{p'}_x}\les 1} \big\| (1+|\p_t|)^{|a|} \big( v (1+|\p_t|)^a w\big) \big\|_{L^{\tilde{q}'}_t L^{q'}_x} \\
       \lesa{}& \| (1+|\p_t|)^a u \|_{L^{{\tilde{q}}}_t L^q_x} \sup_{\| w \|_{L^{\tilde{p}'}_t L^{p'}_x}\les 1} \Big( \|  v\|_{L^{\tilde{r}}_t L^r_x} \| w \|_{L^{\tilde{p}'}_t L^{p'}_x} + \| (1+|\p_t|)^{|a|} v\|_{L^{\tilde{r}}_t L^r_x} \| ( 1+ |\p_t|)^a w \|_{L^{\tilde{p}'}_t L^{p'}_x}  \Big)\\
       \lesa{}& \| (1+|\p_t|)^a u \|_{L^{{\tilde{q}}}_t L^q_x} \| (1+|\p_t|)^{|a|} v\|_{L^{\tilde{r}}_t L^r_x}
    \end{align*}
as required.
\end{proof}

\subsection{Decomposability of norms}\label{subsec:decom}
Given open intervals $I_1, I_2 \subset \RR$ we would like to bound the norm $\| u \|_{S^{s,a,b}(I_1\cup I_2)}$ in terms of the norms $\| u \|_{S^{s,a,b}(I_1)}$ and $\|u\|_{S^{s,a,b}(I_2)}$ on the small intervals $I_1$ and $I_2$.
\begin{lemma}[Decomposability]\label{lem:decomposability} There exists a constant $C>0$ such that for any $s\in \RR$, $0\le a,b\le 1$, any open intervals $I_1, I_2 \subset \RR$ with $I_1 \cap I_2 \not = \varnothing$, and any $u \in S^{s,a,b}(I_1) \cap S^{s,a,b}(I_2)$ we have
  $$ \| u \|_{S^{s,a,b}(I_1 \cup I_2)} \les C \big( 1 + |I_1 \cap I_2|^{-a_\sharp}\big) \big( \| u \|_{S^{s,a,b}(I_1)} + \| u \|_{S^{s,a,b}(I_2)}\big), $$
  for $a_\sharp:=\max\{a,\frac12\}$.
\end{lemma}
\begin{proof}
Let $\rho \in C^\infty(\RR)$ with $\rho(t)=1$ for $t\les -1$, $\rho(t) = 0$ for $t \g 1$, and for every $t \in \RR$
    $$ \rho(t) + \rho(-t) = 1. $$
After a shift, we may assume that $(-\epsilon, \epsilon) \subset I_1 \cap I_2$ for some $\epsilon>0$, and that $I_1$ lies to the left of $I_2$ (i.e. $\inf I_1 \les \inf I_2$). Define $\rho_1(t) = \rho( \epsilon^{-1} t)$ and $\rho_2(t) = \rho( - \epsilon^{-1} t)$ and let $u^j$ be an extension of $u|_{I_j}$ to $\RR$ such that
$\|u\|_{S^{s,a,b}(I_j)}  \sim \|u^j\|_{S^{s,a,b}}$.
By construction we have $u = \rho_1 u^1 + \rho_2 u^2$ on $I_1\cup I_2$, and hence by definition of the restriction norm
\begin{align*}
 \|u\|_{S^{s,a,b}(I_1\cup I_2)}
    \le \|\rho_1u^1\|_{S^{s,a,b}}+\|\rho_2u^2\|_{S^{s,a,b}}
    &\lesa (1+\epsilon^{-a_\sharp})(\|u^1\|_{S^{s,a,b}}+\|u^2\|_{S^{s,a,b}})\\
    &\lesa (1+\epsilon^{-a_\sharp})(\|u\|_{S^{s,a,b}(I_1)}+\|u\|_{S^{s,a,b}(I_2)}),
\end{align*}
provided that $S^{s,a,b}$ enjoys a localisability estimate of the form
\[
 \|\rho_ju\|_{S^{s,a,b}} \lesa (1+\epsilon^{-a_\sharp})\|u\|_{S^{s,a,b}}. \]
Taking $\epsilon>0$ as large as possible (namely $\epsilon \approx |I_1 \cap I_2|$) leads to the desired estimate.

It remains to prove the above localisability, which follows from the product estimate Lemma \ref{lem:elementary product est}. Indeed, for every frequency $\lambda \in 2^\NN$, we have
\[
 \|(\lambda + |\p_t|)^a( \rho_1 u_\lambda ) \|_{L^2_t L^{2^*}_x}
         \lesa  \|\lambda^{-a}(\lambda+|\p_t|)^a \rho_1\|_{L^\infty_t}  \|(\lambda + |\p_t|)^a u_\lambda  \|_{L^2_t L^{2^*}_x}, \]
where the norm of $\rho_1$ is bounded uniformly in $\lambda$ by $\| (1 + |\p_t| )^a [ \rho(\epsilon^{-1} t)] \|_{L^\infty_t} \lesa 1   + \epsilon^{-a}$.
The $L^\infty_t L^2_x$ component is trivially localisable. For the remaining $L^2_{t,x}$ component of $S^{s,a,b}$, we have
\[
 \Big\|\Big(\frac{\lambda + |\p_t|}{\lambda^2 + |\p_t|}\Big)^a (i\partial_t+\Delta)( \rho_1 u_\lambda ) \Big\|_{L^2_{t,x}}
 \le \Big\|\Big(\frac{\lambda + |\p_t|}{\lambda^2 + |\p_t|}\Big)^a [ \rho_1 (i\partial_t+\Delta)u_\lambda ] \Big\|_{L^2_{t,x}}
   + \Big\|\Big(\frac{\lambda + |\p_t|}{\lambda^2 + |\p_t|}\Big)^a [\dot \rho_1 u_\lambda]  \Big\|_{L^2_{t,x}}. \]
To bound the first term, we decompose $u$ into high and low temporal frequencies and observe that another application of Lemma \ref{lem:elementary product est} gives
    \begin{align*}
       \Big\|&\Big(\frac{\lambda + |\p_t|}{\lambda^2 + |\p_t|}\Big)^a [ \rho_1 (i\partial_t+\Delta)u_\lambda ] \Big\|_{L^2_{t,x}}    \\
                &\lesa \lambda^{-2a} \|(\lambda + |\p_t|)^a [ \rho_1 (i\partial_t+\Delta) P^{(t)}_{\ll \lambda^2} u_\lambda ] \Big\|_{L^2_{t,x}} + \| \rho_1 (i\partial_t+\Delta) P^{(t)}_{\gtrsim \lambda^2} u_\lambda \|_{L^2_{t,x}} \\
                &\lesa \| \lambda^{-a} (\lambda + |\p_t|)^a \rho_1\|_{L^\infty_t} \lambda^{-2a} \| (\lambda+|\p_t|)^a (i\p_t + \Delta) P^{(t)}_{\ll \lambda^2} u_\lambda\|_{L^2_{t,x}} + \| \rho_1 \|_{L^\infty_t} \| (i\p_t + \Delta) P^{(t)}_{\gtrsim \lambda^2} u_\lambda \|_{L^2_{t,x}} \\
                &\lesa ( 1+ \epsilon^{-a}) \lambda^{1-s-b} \| u_\lambda \|_{S^{s,a,b}_\lambda}.
    \end{align*}
On the other hand, for the second term, we have
    \begin{align*}
       \Big\|\Big(\frac{\lambda + |\p_t|}{\lambda^2 + |\p_t|}\Big)^a [\dot \rho_1 u_\lambda]  \Big\|_{L^2_{t,x}}&\lesa \| \dot \rho_1 u_\lambda \|_{L^2_{t,x}} \les \| \dot \rho_1 \|_{L^2_x} \| u_\lambda \|_{L^\infty_t L^2_x} \lesa \epsilon^{-\frac{1}{2}} \lambda^{-s} \| u_\lambda \|_{S^{s,a,b}_\lambda}
    \end{align*}
which implies the required bound, since $b\le 1$.
\end{proof}

\section{Bilinear Estimates for Schr\"odinger nonlinearity}\label{sec:bil-est-schr}

In this section we prove that we can bound the Schr\"odinger nonlinearity in the space $N^{s,a,b}$.

\begin{theorem}[Bilinear estimate for Schr\"odinger nonlinearity]\label{thm:main bilinear schro}
Let $d\g4$, $0 \les s \les \ell + 2$, $\beta\g 0$, and $0\les a, b \les 1$ such that
    $$\ell \g b + \tfrac{d-4}{2}, \qquad s-\ell \les a + 1 -b, \qquad s+\ell \g 2a, \qquad \beta  \g \max\{s-1, \tfrac{d-4}{2}+b\}.$$
and
    $$ (s, \ell) \not = \big(\tfrac{d-2}{2}+a, \tfrac{d-4}{2} +b\big),  \qquad (\beta, b) \not = \big(\tfrac{d-2}{2}, 1\big).$$
Then
    $$ \| \Re(V) u \|_{N^{s,a,b}} \lesa \| V \|_{W^{\ell,a,\beta}} \| u \|_{S^{s,a, 0}}. $$
\end{theorem}
\begin{proof}
In view of the definition of $N^{s,a,b}$ and $W^{\ell,a,\beta}$, a short computation shows that it suffices to prove the bounds
    \begin{align}\label{eqn:thm main bi schro:left L2}
        \bigg( \sum_{\lambda_0 \in 2^\NN}   \lambda^{2(s-1 -2a + b)}_0 \Big\| (\lambda_0 + |\p_t|)^a P_{\lambda_0}( v u) \Big\|_{L^2_{t,x}}^2 \bigg)^\frac{1}{2} &\lesa  \Big( \sum_\mu  \|  (\mu + |\p_t|)^a v_\mu \|_{L^\infty_t H^{\ell-a}_x}^2\Big)^\frac{1}{2} \| u \|_{S^{s,a,0}}
    \\
       \bigg( \sum_{\lambda_0 \in 2^\NN}  \lambda_0^{2s}  \|P_{\lambda_0}(vu) \|_{L^2_t L^{2_*}_x}^2 + \lambda_0^{2(s-1 + b)} \Big\| P_{\lambda_0}(vu) \Big\|_{L^2_{t,x}}^2 \bigg)^\frac{1}{2}  &\lesa  \| v \|_{L^2_t H^{\beta+1}_x} \| u \|_{S^{s,a,0}} \label{eqn:thm main bi schro:right L2} \\
        \Big(\sum_{\lambda_0 \in 2^\NN} \lambda_0^{2(s-2)} \| P_{\lambda_0}(vu) \|_{L^\infty_t L^2_x}^2 \Big)^\frac{1}{2} &\lesa \Big( \sum_\mu \| v_\mu \|_{L^\infty_t H^\ell}^2\Big)^\frac{1}{2} \| u \|_{L^\infty_t H^s_x}
        \label{eqn:thm main bi schro:Linfty}
    \end{align}
and, under the additional assumption that $\supp \widetilde{v} \subset \{ |\tau| \ll \lr{\xi}^2\}$, that we have
    \begin{equation}\label{eqn:thm main bi schro:stri}
        \Big( \sum_{\lambda_0} \lambda_0^{2s} \Big\|  P^N_{\lambda_0}( v u) \Big\|_{L^2_t L^{2_*}_x}^2 \Big)^\frac{1}{2}
                    \lesa \Big( \sum_{\mu} \| (\mu + |\p_t|)^a v_\mu \|_{L^\infty_t H^{\ell -a }_x}^2\Big)^\frac{1}{2} \| u \|_{S^{s,a,0}}.
    \end{equation}
More precisely, assuming that the bounds \eqref{eqn:thm main bi schro:left L2} -- \eqref{eqn:thm main bi schro:stri} hold, we decompose
        $$V= \sum_{\mu \in 2^\NN} V_\mu = \sum_{\mu \in 2^\NN} P^{(t)}_{\ll \mu^2} V_\mu + \sum_{\mu \in 2^\NN} P^{(t)}_{\gg \mu} P^{(t)}_{\gtrsim \mu^2} V_\mu + \sum_{\mu \in 2^\NN} P^{(t)}_{\lesa \mu} P^{(t)}_{\gtrsim \mu^2} V_\mu = V_1 +V_2 + V_3. $$
An application of \eqref{eqn:thm main bi schro:left L2}, \eqref{eqn:thm main bi schro:Linfty}, and \eqref{eqn:thm main bi schro:stri} (together with the invariance of the righthand side with respect to complex conjugation) gives
        \begin{align*}
          \| \Re(V_1) u \|_{N^{s,a,b}} &\lesa \Big( \sum_\mu \| (\mu + |\p_t|)^a P^{(t)}_{\ll \mu^2} V_\mu \|_{L^\infty_t H^{\ell -a}_x}^2\Big)^\frac{1}{2} \| u \|_{S^{s,a,0}} \lesa \| V \|_{W^{\ell, a, \beta}} \| u \|_{S^{s,a,0}}.
        \end{align*}
On the other hand, for the $V_2$ contribution, we note that since
        $$ \| V_2 \|_{L^2_t H^{\beta + 1}_x} \approx \Big( \sum_{\mu} \mu^{2(\beta+1)} \| P^{(t)}_{\gg \mu} V_\mu \|_{L^2_{t,x}}^2 \Big)^\frac{1}{2} \lesa \Big( \sum_{\mu} \mu^{2(\beta-1)} \| (i\p_t + |\nabla|) P^{(t)}_{\gg \mu} V_\mu \|_{L^2_{t,x}}^2 \Big)^\frac{1}{2} \lesa \| V \|_{W^{\ell, a, \beta}} $$
an application of \eqref{eqn:thm main bi schro:right L2} and \eqref{eqn:thm main bi schro:Linfty} implies that
    \begin{align*}
      \| \Re(V_2) u \|_{N^{s,a,b}} &\lesa \Big( \| V_2 \|_{L^\infty_t H^\ell_x} + \| V_2 \|_{L^2_t H^{\beta+1}_x}\Big) \| u \|_{S^{s,a,0}} \lesa \| V \|_{W^{\ell, a, \beta}} \| u \|_{S^{s,a,0}}
    \end{align*}
as required. Finally, the bound for the $V_3$ contribution follows from the fact that $\supp \widetilde{V}_3 \subset \{ |\tau| + |\xi| \lesa 1 \}$ together with \eqref{eqn:thm main bi schro:left L2}, \eqref{eqn:thm main bi schro:Linfty}, and the estimate \eqref{eqn:thm main bi schro:hi-hi temp} below.

We now turn to the proof of the bounds \eqref{eqn:thm main bi schro:left L2} -- \eqref{eqn:thm main bi schro:stri}. For the first estimate \eqref{eqn:thm main bi schro:left L2}, we begin by decomposing the product $vu$ into
     \begin{equation}\label{eqn:thm main bi schro:freq decomp}
        P_{\lambda_0}(vu) = \sum_{\lambda_1 \in 2^\NN} P_{\lambda_0} (vu_{\lambda_1}) =  \sum_{\lambda_1 \ll \lambda_0 } P_{\lambda_0} (vu_{\lambda_1}) + \sum_{\lambda_1 \gg \lambda_0 } P_{\lambda_0} (vu_{\lambda_1}) + \sum_{\lambda_1 \approx \lambda_0} P_{\lambda_0} (vu_{\lambda_1})
     \end{equation}
and consider the high-low interactions $\lambda_0 \gg \lambda_1$, low-high interactions $\lambda_0 \ll \lambda_1$, and the balanced interactions case $\lambda_0 \approx \lambda_1$.

\textbf{Case 1: $\lambda_0 \gg \lambda_1$.} Applying the product estimate Lemma \ref{lem:elementary product est}, together with Sobolev embedding gives
    \begin{align*}
      \lambda^{s-1 -2a + b}_0 \| &(\lambda_0 + |\p_t|)^a P_{\lambda_0}( v u_{\lambda_1}) \|_{L^2_{t,x}} \\
            &\lesa \lambda_0^{s-1 -2a + b}  \lambda_0^{-a} \| ( \lambda_0 + |\p_t|)^a v_{\approx \lambda_0} \|_{L^\infty_t L^2_x} \lambda_1^{\frac{d-2}{2}} \| (\lambda_0 + |\p_t|)^a u_{\lambda_1} \|_{L^2_t L^{2^*}_x} \\
            &\lesa \lambda_0^{s -\ell -1 -a + b} \lambda_1^{\frac{d-2}{2} + a -s}  \| ( \lambda_0 + |\p_t|)^a v_{\approx \lambda_0} \|_{L^\infty_t H^{\ell -a}_x} \lambda_1^{s-2a} \|(\lambda_1 + |\p_t|)^a u_{\lambda_1} \|_{L^2_t L^{2^*}_x}.
    \end{align*}
Therefore, provided that
    $$ s - \ell \les a + 1 - b, \qquad \ell \g \tfrac{d-4}{2} + b, \qquad (s, \ell) \not = \big( \tfrac{d-2}{2} + a, \tfrac{d-4}{2}+b\big), $$
we obtain
    \begin{align*}
      \bigg( \sum_{\lambda_0 \in 2^\NN}   \lambda^{2(s-1 -2a + b)}_0 &\Big\| \sum_{\lambda_1 \ll \lambda_0} (\lambda_0 + |\p_t|)^a P_{\lambda_0}( v u_{\lambda_1}) \Big\|_{L^2_{t,x}}^2 \bigg)^\frac{1}{2}\\
                &\lesa \bigg( \sum_{\lambda_0 \in 2^\NN} \Big( \sum_{\lambda_1 \ll \lambda_0} \lambda_0^{s -\ell -1 -a + b} \lambda_1^{\frac{d-2}{2} + a -s}  \| ( \lambda_0 + |\p_t|)^a v_{\approx \lambda_0} \|_{L^\infty_t H^{\ell -a}_x} \| u_{\lambda_1} \|_{S^{s,a,0}_{\lambda_1}}\Big)^2   \bigg)^\frac{1}{2} \\
                &\lesa \Big( \sum_{\mu \in 2^\NN} \| (\mu + |\p_t|)^a v_\mu \|_{L^\infty_t H^{\ell -a}}^2\Big)^\frac{1}{2} \sup_{\lambda_1} \| u_{\lambda_1}\|_{S^{s,a,0}_{\lambda_1}}
    \end{align*}
as required.

\textbf{Case 2: $\lambda_0 \ll  \lambda_1$.} We begin by observing that an application of the Sobolev embedding $W^{\frac{d-2}{2}, \frac{d}{d-1}}(\RR^d) \hookrightarrow L^2(\RR^d)$ implies that
        \begin{align*}
            \Big( \sum_{\lambda_0 \lesa \lambda_1} \lambda_0^{2(s-1-2a + b)} \| F_{\lambda_0} \|_{L^2_{t,x}}^2\Big)^\frac{1}{2}
                    &\lesa \| F_{\lesa \lambda_1} \|_{L^2_t H^{s-1-2a+b}_x}\\
                    &\lesa \| F_{\lesa \lambda_1} \|_{L^2_t W^{\frac{d-2}{2} + s-1-2a+b, \frac{d}{d-1}}_x} \lesa \lambda_1^{(s+\frac{d-4}{2}-2a+b)_+} \| F \|_{L^2_t L^{\frac{d}{d-1}}_x}.
        \end{align*}
On the other hand, again applying the product estimate Lemma \ref{lem:elementary product est} gives
    \begin{align*}
       \| (\lambda_1 + |\p_t|)^a ( v_{\approx \lambda_1} u_{\lambda_1}) \|_{L^2_t L^{\frac{d}{d-1}}_x}
                &\lesa  \lambda_1^{-a} \| (\lambda_1 + |\p_t|)^a  v_{\approx \lambda_1} \|_{L^\infty_t L^2_x} \| (\lambda_1 + |\p_t|)^a u_{\lambda_1} \|_{L^2_t L^{2^*}_x} \\
                &\lesa  \lambda_1^{ 2a -s - \ell} \| (\lambda_1 + |\p_t|)^a  v_{\approx \lambda_1} \|_{L^\infty_t H^{\ell-a}_x} \|u_{\lambda_1} \|_{S^{s,a,0}_{\lambda_1}}.
    \end{align*}
Hence, provided that
    $$ s+\ell \g 2a, \qquad \ell \g \tfrac{d-4}{2} + b$$
we see that
    \begin{align*}
      \bigg( \sum_{\lambda_0 \in 2^\NN} \lambda_0^{2(s-1-2a+b)} \Big\| &(\lambda_0 + |\p_t|)^a \sum_{\lambda_1 \gg \lambda_0} P_{\lambda_0}( v u_{ \lambda_1}) \Big\|_{L^2_{t,x}}^2 \bigg)^\frac{1}{2} \\
            &\lesa \sum_{\lambda_1 \in 2^\NN} \bigg( \sum_{\lambda_0 \lesa \lambda_1} \lambda_0^{2(s-1-2a+b)} \| (\lambda_1 + |\p_t|)^a  P_{\lambda_0}( v_{\approx \lambda_1} u_{\lambda_1}) \|_{L^2_{t,x}}^2 \bigg)^\frac{1}{2} \\
            &\lesa \sum_{\lambda_1 \in 2^\NN} \lambda_1^{( s + \frac{d-4}{2} -2a + b)_+} \| (\lambda_1 + |\p_t|)^a ( v_{\approx \lambda_1} u_{\lambda_1}) \|_{L^2_t L^{\frac{d}{d-1}}_x} \\
            &\lesa \sum_{\lambda_1 \in 2^\NN} \lambda_1^{ ( s + \frac{d-4}{2} -2a + b)_+ + 2a -s - \ell} \| (\lambda_1 + |\p_t|)^a v_{\approx \lambda_1} \|_{L^\infty_t H^{\ell -a}_x} \| u_{\lambda_1} \|_{S^{s,a,0}_{\lambda_1}} \\
            &\lesa \Big( \sum_{\mu \in 2^\NN} \| (\mu + |\p_t|)^a v_\mu \|_{L^\infty_t H^{\ell -a}_x}^2 \Big)^\frac{1}{2} \| u \|_{S^{s,a,0}}.
    \end{align*}

\textbf{Case 3: $\lambda_0 \approx \lambda_1$. } Similar to above, we have
    \begin{align*}
         \lambda_0^{s-1-2a + b} \| (\lambda_0 + |\p_t|)^a &P_{\lambda_0}( v u_{\lambda_1}) \|_{L^2_{t,x}} \\
                &\lesa \lambda_0^{s-1-2a + b} \lambda_0^{-a } \| (\lambda_0 + |\p_t|)^a v_{\lesa \lambda_0} \|_{L^\infty_t L^d_x} \| (\lambda_1 + |\p_t|)^a u_{\lambda_1} \|_{L^2_t L^{2^*}_x} \\
                &\lesa \lambda_0^{b-1} \| (\lr{\nabla} + |\p_t|)^a v_{\lesa \lambda_0} \|_{L^\infty_t H^{\frac{d-2}{2} -a}_x} \lambda_1^{s-2a} \| (\lambda_1 + |\p_t|)^a u_{\lambda_1} \|_{L^2_t L^{2^*}_x}
    \end{align*}
which is summable provided that
    $$ b \les 1, \qquad \ell \g \tfrac{d-4}{2} + b. $$
This completes the proof of \eqref{eqn:thm main bi schro:left L2}.

We now turn to the proof of the second estimate \eqref{eqn:thm main bi schro:right L2}. As previously, we apply the frequency decomposition \eqref{eqn:thm main bi schro:freq decomp} and consider each frequency interaction separately.

\textbf{Case 1: $\lambda_0 \gg \lambda_1$.} We start by noting that an application of Sobolev embedding gives
    $$ \sup_{\lambda_0 \in 2^\NN} \lambda_0^{s-\beta - 1} \Big( \| u_{\ll \lambda_0} \|_{L^\infty_t L^d_x} + \lambda_0^{b-1} \| u_{\ll \lambda_0} \|_{L^\infty_{t,x}}\Big) \lesa \| u \|_{L^\infty_t H^s_x} \lesa \| u \|_{S^{s,a,0}} $$
provided that
    $$ s \les \beta + 1, \qquad \beta \g \tfrac{d-4}{2} + b, \qquad ( \beta, b) \not = (\tfrac{d-2}{2}, 1). $$
Hence via H\"older's inequality we obtain
    \begin{align*}
      \bigg( \sum_{\lambda_0 \in 2^\N} &\lambda_0^{2s} \|  P_{\lambda_0}( v u_{\ll \lambda_0}) \|_{L^2_t L^{2_*}_x}^2 + \lambda_0^{2(s -1 + b)} \| P_{\lambda_0}(v u_{\ll \lambda_0} ) \|_{L^2_{t,x}}^2\bigg)^\frac{1}{2} \\
                &\lesa \Big( \sum_{\lambda_0 \in 2^\NN} \lambda_0^{2(\beta + 1)} \| v_{\approx \lambda_0} \|_{L^2_{t,x}}^2\Big)^\frac{1}{2}
                        \sup_{\lambda_0 \in 2^\NN} \lambda_0^{s-\beta -1 } \Big( \| u_{\ll \lambda_0} \|_{L^\infty_t L^d_x} + \lambda_0^{b-1} \| u_{\ll \lambda_0} \|_{L^\infty_{t,x}}\Big) \\
                &\lesa \| v \|_{L^2_t H^{\beta + 1}_x} \| u \|_{S^{s,a,0}}.
    \end{align*}

\textbf{Case 2: $\lambda_0 \ll \lambda_1$.} An application of Bernstein's inequality together with the square function characterisation of $L^p_x$ gives
    \begin{align*}
      \Big( \sum_{\lambda_0 \ll \lambda_1} \lambda_0^{2s} \| F_{\lambda_0} \|_{L^2_t L^{2_*}_x}^2 + \lambda_0^{2(s-1+b)} \| F_{\lambda_0} \|_{L^2_{t,x}}^2\Big)^\frac{1}{2}
            &\lesa \lambda_1^{s+b} \Big( \sum_{\lambda_0 \in 2^\NN} \| F_{\lambda_0} \|_{L^2_t L^{2_*}_x}^2\Big)^\frac{1}{2} \\
            &\lesa \lambda_1^{s+b} \Big\| \Big( \sum_{\lambda_0 \in 2^\NN} |F_{\lambda_0}|^2 \Big)^\frac{1}{2} \Big\|_{L^2_t L^{2_*}_x} \lesa \lambda_1^{s+b} \| F \|_{L^2_t L^{2_*}_x}.
    \end{align*}
Therefore applying Bernstein's inequality and H\"older's inequality we conclude that
    \begin{align*}
      \bigg( \sum_{\lambda_0 \in 2^\NN} \lambda_0^{2s} \Big\|& \sum_{\lambda_1 \gg \lambda_0} P_{\lambda_0}(v u_{\lambda_1}) \Big\|_{L^2_t L^{2_*}_x}^2 + \lambda_0^{2(s-1+b)} \Big\| \sum_{\lambda_1 \gg \lambda_0} P_{\lambda_0}( v u_{\lambda_1}) \Big\|_{L^2_{t,x}}^2\bigg)^\frac{1}{2} \\
            &\lesa \sum_{\lambda_1 \in 2^\NN}  \bigg( \sum_{\lambda_0 \ll \lambda_1} \lambda_0^{2s} \| P_{\lambda_0}(v_{\approx \lambda_1} u_{\lambda_1}) \|_{L^2_t L^{2_*}_x}^2 + \lambda_0^{2(s-1+b)} \| P_{\lambda_0}( v_{\approx \lambda_1} u_{\lambda_1}) \|_{L^2_{t,x}}^2\bigg)^\frac{1}{2} \\
            &\lesa \sum_{\lambda_1 \in 2^\NN} \lambda_1^{s+b} \| v_{\approx \lambda_1} u_{\lambda_1} \|_{L^2_t L^{2_*}_x} \\
            &\lesa \sum_{\lambda_1 \in 2^\NN} \lambda_1^{s+b + \frac{d-2}{2}} \| v_{\approx \lambda_1} \|_{L^2_{t,x}} \| u_{\lambda_1} \|_{L^\infty_t L^2_x} \lesa \| v \|_{L^2_t H^{\beta+1}_x} \| u \|_{S^{s,a,0}}
    \end{align*}
provided that
    $$ \beta \g \tfrac{d-4}{2} + b . $$

\textbf{Case 3: $\lambda_0 \approx \lambda_1$.} Let $\frac{1}{r} = \frac{1-b}{d}$. Similar to the above, an application of Sobolev embedding gives
    $$  \| v \|_{L^2_t L^d_x} + \| v \|_{L^2_t L^r_x} \lesa \| v \|_{L^2_t H^{\beta+1}_x} $$
provided that
    $$ \beta \g \tfrac{d-4}{2} + b, \qquad (\beta, b) \not = ( \tfrac{d-2}{2}, 1). $$
Consequently, via Bernstein's inequality we have
    \begin{align*}
      \bigg( \sum_{\lambda_0 \in 2^\NN} \lambda_0^{2s} \| v u_{\approx \lambda_0} \|_{L^2_t L^{2_*}_x}^2 + \lambda_0^{2(s-1+b)} \| v u_{\approx \lambda_0} \|_{L^2_{t,x}}^2 \bigg)^\frac{1}{2}
                &\lesa \Big( \| v  \|_{L^2_t L^d_x} + \| v \|_{L^2_t L^r_x}\Big) \Big( \sum_{\lambda_0 \in 2^\NN} \lambda_0^{2s} \| u_{\approx \lambda_0} \|_{L^\infty_t L^2_x}^2\big)^\frac{1}{2}\\
                &\lesa \| v \|_{L^2_t H^{\beta+1}_x} \| u \|_{S^{s,a,0}}.
    \end{align*}
This completes the proof of \eqref{eqn:thm main bi schro:right L2}.

 The $L^\infty_t L^2_x$ bound \eqref{eqn:thm main bi schro:Linfty} holds provided that $s\les \ell + 2$, $\ell \g \frac{d-4}{2}$, and $(s, \ell) \not = (\frac{d}{2}, \frac{d-4}{2})$. The proof is standard, and follows by adapting the proof of the product estimate $\| fg \|_{H^{s-2}} \lesa \| f \|_{H^\ell} \| g \|_{H^s}$.

 We now turn to the proof of the final estimate \eqref{eqn:thm main bi schro:stri}. As before, we decompose the inner sum into high-low interactions $\lambda_0 \gg \lambda_1$, low-high interactions $\lambda_0 \ll \lambda_1$, and the balanced interactions case $\lambda_0 \approx \lambda_1$, and consider each case separately.

\textbf{Case 1: $\lambda_0 \gg \lambda_1$.} The assumption on the Fourier support of $v$ implies the non-resonant identity
        $$  C_{ \les ( \frac{\lambda_0}{2^8})^2} P_{\lambda_0}( v u_{\lambda_1}) = C_{ \les ( \frac{\lambda_0}{2^8})^2} P_{\lambda_0}( v_{\approx \lambda_0} P^{(t)}_{\approx \lambda_0^2} C_{> (\frac{\lambda_1}{2^8})^2} u_{\lambda_1}).$$
Hence the disposability of the multiplier $P^N_{\lambda_0}$, and Bernstein's inequality, gives
        \begin{align*}
            \big\| P^N_{\lambda_0}( v u_{\lambda_1}) \big\|_{L^2_t L^{2_*}_x}
                        &\lesa \lambda_1^{\frac{d}{2}-1} \| v_{\approx \lambda_0} \|_{L^\infty_t L^2_x} \| P^{(t)}_{\approx \lambda_0^2} C_{> (\frac{\lambda_1}{2^8})^2} u_{\lambda_1} \|_{L^2_{t,x}} \\
                        &\lesa \lambda_1^{\frac{d}{2}-1} \lambda_0^{-2} \| v_{\approx \lambda_0} \|_{L^\infty_t L^2_x} \Big\| \Big(\frac{\lambda_1 + |\p_t|}{\lambda_1^2+|\p_t|}\Big)^a ( i \p_t + \Delta) u_{\lambda_1} \Big\|_{L^2_{t,x}}
        \end{align*}
Consequently, we conclude that
    $$ \lambda^s_0 \big\| P^N_{\lambda_0}( v u_{\lambda_1}) \big\|_{L^2_t L^{2_*}_x} \lesa \lambda^{s - \ell -2}_0 \lambda_1^{\frac{d}{2}-s} \| v_{\approx \lambda_0} \|_{L^\infty_t H^\ell_x} \| u \|_{S^{s,a,0}} $$
which is summable provided that
    $$ s \les \ell + 2, \qquad \ell \g \tfrac{d-4}{2}, \qquad (s, \ell) \not = \big( \tfrac{d}{2}, \tfrac{d-4}{2}\big). $$

\textbf{Case 2: $\lambda_0 \ll  \lambda_1$.}
We first observe that the Fourier support assumption on $v$ implies that
\[
C_{ \les ( \frac{\lambda_0}{2^8})^2} P_{\lambda_0}(vu_{\lambda_1})=C_{ \les ( \frac{\lambda_0}{2^8})^2} P_{\lambda_0}(v_{\approx \lambda_1}C_{\approx \lambda_1^2}u_{\lambda_1}).
\]
Bernstein's inequality and the temporal product estimate in Lemma \ref{lem:elementary product est} implies
\begin{align*}
  \lambda_0^s\|P^N_{\lambda_0}(vu_{\lambda_1})\|_{L^2_tL^{2_*}_x}& \lesa \lambda_0^{s+\frac{d-2}{2}-2a}\|(\lambda_1+|\partial_t|)^a (v_{\approx \lambda_1}C_{\approx \lambda_1^2}u_{\lambda_1})\|_{L^2_tL^{1}_x}\\
        & \lesa \lambda_0^{s+\frac{d-2}{2}-2a}\lambda_1^{-a}\|(\lambda_1+|\partial_t|)^a v_{\approx \lambda_1}\|_{L^\infty_tL^{2}_x}\|(\lambda_1+|\partial_t|)^a C_{\approx \lambda_1^2} u_{\lambda_1} \|_{L^{2}_{t,x}}\\
        & \lesa \lambda_0^{s+\frac{d-2}{2}-2a}\lambda_1^{2a-s-\ell-1}\|(\lambda_1+|\partial_t|)^a v_{\approx \lambda_1}\|_{L^\infty_tH^{\ell-a}_x}\| u_{\lambda_1}\|_{S^{s,a,0}_{\lambda_1}}
\end{align*}
which is certainly summable under the assumption that
\[
s+\ell\g 2a, \quad \ell\g \tfrac{d-4}{2}.
\]

\textbf{Case 3: $\lambda_0 \approx \lambda_1$. } We now consider the remaining high-high interactions. Via the product estimate in Lemma \ref{lem:elementary product est} we obtain
\begin{align}
  \lambda_0^s\|P^N_{\lambda_0}(vu_{\lambda_1})\|_{L^2_tL^{2_*}_x}
        &\lesa\lambda_1^{s-2a}\|(\lambda_1+|\partial_t|)^a(v_{\lesa \lambda_1}u_{\lambda_1})\|_{L^2_tL^{2_*}_x}\notag \\
        &\lesa\lambda_1^{s-2a}\lambda_1^{-a}\|(\lambda_1+|\partial_t|)^av_{\lesa \lambda_1}\|_{L^\infty_t L^{\frac{d}{2}}_x}\|(\lambda_1+|\partial_t|)^au_{\lambda_1}\|_{L^2_tL^{2^*}_x}\notag \\
&\lesa\|(\lr{\nabla}+|\partial_t|)^av\|_{L^\infty_t H^{\ell-a}_x}\|u_{\lambda_1}\|_{S^{s,a,0}_{\lambda_1}},
\label{eqn:thm main bi schro:hi-hi temp}
\end{align}
where we have used that $\ell\g \frac{d-4}{2}$ for the Sobolev embedding, and the summation is trivial in this case.
\end{proof}

We require a local version of the bilinear estimate, with the advantage that we can place $v$ in dispersive norms of the form $L^\infty_t L^2_x + L^2_t L^d_x$.

\begin{corollary}\label{cor:local schro bi}
Let $d\g 4$. Assume that $\beta \g \max\{ \frac{d-4}{2}, s-1\}$ and
\begin{equation}\label{eqn:cor local schro bi:cond} 0\les a \les 1, \qquad 0\les s \les \ell + 2, \qquad \ell \g \frac{d-4}{2}, \qquad s-\ell \les a + 1, \qquad s+\ell \g 2a,\end{equation}
with $(s, \ell) \not = (\frac{d-2}{2} +a, \frac{d-4}{2})$. There exists $C>0$ such that for any interval $0\in I \subset \RR$ we have
        $$ \| \mc{I}_0\big( \Re(V) u \big) \big\|_{S^{s,a,0}(I)} \les C \| V \|_{W^{\ell, a, \beta}(I) + L^2_t W^{s,d}_x(I\times \RR^d)} \| u \|_{S^{s,a,0}(I)}. $$
\end{corollary}
\begin{proof}
In view of Lemma \ref{lem:energy ineq} and Theorem \ref{thm:main bilinear schro}, it suffices to prove that for any $s\g 0$ and $0\les a \les 1$ we have
        \begin{equation}\label{eqn:cor local schro bi:initial bound}
            \| \mc{I}_0(\Re(V) u ) \|_{S^{s,a,0}(I)} \lesa \| V \|_{L^2_t W^{s,d}_x(I\times \RR^d)} \| u \|_{S^{s,a,0}(I)}.
        \end{equation}
An application of Bernstein's inequality together with \eqref{eqn:norm chara} gives
    \begin{align*}
      \| u_\lambda \|_{S^{s,a,0}_\lambda}  &\approx \lambda^s \Big( \| u_\lambda \|_{L^\infty_t L^2_x} + \| P^N_\lambda u \|_{L^2_t L^{2^*}_x}\Big) + \lambda^{s-1} \Big\| \Big( \frac{\lambda + |\p_t|}{\lambda^2 + |\p_t|}\Big)^a (i\p_t + \Delta) u_\lambda \Big\|_{L^2_{t,x}}\\
      &\lesa \lambda^s \Big( \| u_\lambda \|_{L^\infty_t L^2_x} + \| u_\lambda \|_{L^2_t L^{2^*}_x} + \| (i\p_t + \Delta) u_\lambda \|_{L^2_t L^{2_*}_x}\Big)
    \end{align*}
and hence the endpoint Strichartz estimate  implies that after extending $F$ from $I$ to $\RR$ by zero, that
    $$
          \| \mc{I}_0[F] \|_{S^{s,a,0}(I)}
                    \les \|  \mc{I}_0[ \ind_I F] \|_{S^{s,a,0}} \lesa \Big( \sum_{\lambda \in 2^\NN} \lambda^{2s} \| F_\lambda \|_{L^2_t L^{2_*}_x(I\times \RR^d)}^2 \Big)^\frac{1}{2} \lesa \| F \|_{L^2_t W^{s, 2_*}_x(I\times \RR^d)}.
    $$
The inequality \eqref{eqn:cor local schro bi:initial bound} then follows from the elementary product estimate
        $$ \| f g\|_{W^{s, 2_*}(\RR^d)}  \lesa \|f \|_{W^{s, d}(\RR^d)} \| g \|_{H^s(\RR^d)} $$
which holds for any $s\g 0$.
\end{proof}

\section{Bilinear estimates for the wave nonlinearity}\label{sec:bil-est-wave}
Here we give the bilinear estimates required to control solutions to
        $$ (i\p_t + |\nabla|) v = |\nabla|( \overline{\varphi} \psi ), \qquad v(0) = 0$$
with $\varphi, \psi \in S^{s,a, b}$. The main estimate we prove is the following.

\begin{theorem}[Bilinear estimate for wave nonlinearity]\label{thm:main bilinear wave}
Let $d\g 4$, $s , \ell, \beta \g 0$, and $0\les a, b \les 1$ satisfy
        $$ \beta \les \min\big\{s, 2s - \tfrac{d-2}{2} -a\big\}, \qquad  2a \les 2s - \ell - \tfrac{d-2}{2}, \qquad a-b \les s-\ell$$
and
        $$ (s, \ell) \not = \big( \tfrac{d}{2}, \tfrac{d+2}{2}\big), \big( \tfrac{d-2}{2}+a, \tfrac{d-2}{2} + b\big), \qquad (s, \beta) \not = (\tfrac{d-2}{2}+a, \tfrac{d-2}{2}+a). $$
If $\varphi, \psi \in S^{s,a, b}$,
then
        $$ \|\mathcal{J}_0(|\nabla|(\overline{\varphi} \psi))\|_{W^{\ell, a, \beta}} \lesa \| \varphi \|_{S^{s,a, b}} \| \psi\|_{S^{s,a, b}}.$$
\end{theorem}
\begin{proof} An application of the energy inequality in Lemma \ref{lem:energy ineq wave} implies that it suffices to prove the bounds
    \begin{align}
      \Big( \sum_{\mu\in 2^\NN} \mu^{2(\ell -a +1)} \| (\mu + |\p_t|)^a P_{\mu}P^{(t)}_{\ll \mu^2} (\overline{\varphi} \psi) \|_{L^1_t L^2_x}^2\Big)^\frac{1}{2}
            &\lesa \| \varphi \|_{S^{s,a,b}} \| \psi \|_{S^{s,a,b}}, \label{eqn:thm main bi wave:LinfL2 with weight} \\
      \Big( \sum_{\mu \in 2^\NN} \mu^{2(\ell -1)} \| P_{\mu} ( \overline{\varphi} \psi) \|_{L^\infty_t L^2_x}^2 \Big)^\frac{1}{2}
            &\lesa \| \varphi \|_{S^{s,a,b}} \| \psi \|_{S^{s,a,b}}, \label{eqn:thm main bi wave:LinfL2} \\
      \Big( \sum_{\mu \in 2^\NN} \mu^{2\beta} \| P_{\mu}(\overline{\varphi} \psi) \|_{L^2_{t,x}}^2\Big)^\frac{1}{2}
            &\lesa \| \varphi \|_{S^{s,a,b}} \| \psi \|_{S^{s,a,b}}, \label{eqn:thm main bi wave:L2}\\
      \| P_{\les 2^{16}} ( \overline{\varphi} \psi ) \|_{L^1_t L^2_x}
            &\lesa \| \varphi \|_{S^{s,a,b}} \| \psi \|_{S^{s,a,b}}. \label{eqn:thm main bi wave:low freq}
    \end{align}
We start with the proof of \eqref{eqn:thm main bi wave:LinfL2 with weight} and decompose the product $\overline{\varphi}\psi$ into the standard frequency trichotomy
        \begin{equation}\label{eqn:thm main bi wave:freq decomp}
            P_{\mu}(\overline{\varphi} \psi) = P_{\mu}(\overline{\varphi} \psi_{\ll \mu}) + \sum_{\lambda_1 \approx \lambda_2 \gtrsim \mu} P_{\mu}(\overline{\varphi}_{\lambda_1} \psi_{\lambda_2}) + P_{\mu}(\overline{\varphi}_{\ll \mu} \psi).
        \end{equation}
In view of the fact that the left hand side of \eqref{eqn:thm main bi wave:LinfL2 with weight} is invariant with respect to complex conjugation, it suffices to consider the first two terms in \eqref{eqn:thm main bi wave:freq decomp}, i.e. the high-low and high-high frequency interactions.

\textbf{Proof of \eqref{eqn:thm main bi wave:LinfL2 with weight} case 1: high-low interactions.} Note that in this case we must have $\mu \gg1$. A computation then gives the non-resonant identity
        $$ P^{(t)}_{\ll \mu^2}P_{\mu}(\overline{ \varphi} \psi_{\ll \mu}) = P^{(t)}_{\ll \mu^2}P_{\mu}(\overline{ C_{\ll \mu^2} \varphi}_{\approx \mu} P^{(t)}_{\approx \mu^2} \psi_{\ll \mu}) + P^{(t)}_{\ll \mu^2}P_{\mu}( \overline{C_{\gtrsim \mu^2} \varphi}_{\approx \mu} \psi_{\ll \mu}) = A_1 + A_2. $$
To bound the $A_1$ term, we observe that
    \begin{align*}
        \mu^{\ell + 1-a} \| (|\p_t| + \mu)^a P^{(t)}_{\ll \mu^2}&P_{\mu}(\overline{ C_{\ll \mu^2} \varphi}_{\approx \mu} P^{(t)}_{\approx \mu^2} \psi_{\ll \mu}) \|_{L^1_t L^2_x}\\
                    &\lesa \mu^{\ell + 1 +a} \| C_{\ll \mu^2} \varphi_{\approx \mu} \|_{L^2_t L^{2^*}_x} \| P^{(t)}_{\approx \mu^2} \psi_{\ll \mu} \|_{L^2_t L^d_x} \\
                    &\lesa  \mu^{\ell -s -1 + a}  \mu^{s-2a} \| (\mu + |\p_t|)^a \varphi_{\approx \mu} \|_{L^2_t L^{2^*}_x} \| (i\p_t + \Delta) P^{(t)}_{\approx \mu^2} \psi_{\ll \mu} \|_{L^2_t H^{\frac{d-2}{2}}_x} \\
                    &\lesa  \mu^{\ell -s -1 + a + ( \frac{d}{2} - s)_+}   \mu^{s-2a} \| (\mu + |\p_t|)^a \varphi_{\approx \mu} \|_{L^2_t L^{2^*}_x} \| \psi \|_{S^{s,a, 0}}.
    \end{align*}
Provided that
    $$  s-\ell \g a-b, \qquad 2s - \ell - \tfrac{d-2}{2} \g a $$
we can sum up over $\mu \gg 1$ to obtain \eqref{eqn:thm main bi wave:LinfL2 with weight} for the $A_1$ contribution. To bound $A_2$, we apply the temporal product estimate in  Lemma \ref{lem:elementary product est} which gives
    \begin{align*}
      \mu^{\ell + 1 - a} &\| (\mu + |\p_t|)^a P^{(t)}_{\ll \mu^2}P_{\mu}( \overline{C_{\gtrsim \mu^2} \varphi}_{\approx \mu} \psi_{\ll \mu}) \|_{L^1_t L^2_x}\\
                    &\lesa \mu^{\ell + 1 - 2a} \| ( \mu + |\p_t|)^a C_{\gtrsim \mu^2} \varphi_{\approx \mu} \|_{L^2_{t,x}} \| ( \mu + |\p_t|)^a \psi_{\ll \mu } \|_{L^2_t L^\infty_x} \\
                    &\lesa \mu^{\ell -1} \Big\| \Big(\frac{ \mu + |\p_t|}{\mu^2 + |\p_t|}\Big)^a (i\p_t + \Delta) \varphi_{\approx \mu} \Big\|_{L^2_{t,x}}
                        \sum_{\lambda \ll \mu}   \Big( \frac{\mu}{\lambda}\Big)^a \lambda^{\frac{d-2}{2}} \| (\lambda + |\p_t|)^a \psi_{\lambda} \|_{L^2_t L^{2^*}_x} \\
                    &\lesa \mu^{\ell - s -b +a } \sum_{\lambda \ll \mu} \lambda^{\frac{d-2}{2} +a -s} \Big\| \Big(\frac{ \mu + |\p_t|}{\mu^2 + |\p_t|}\Big)^a (i\p_t + \Delta) \varphi_{\approx \mu} \Big\|_{L^2_t H^{s+b-1}_x}  \| \psi \|_{S^{s,a,0}}.
    \end{align*}
This can be summed up over $\mu \gg 1$ to give \eqref{eqn:thm main bi wave:LinfL2 with weight} for the $A_2$ contribution provided that
        $$ s-\ell \g a -b, \qquad 2s -\ell - \tfrac{d-2}{2} \g 2a -b, \qquad (s, \ell) \not = (\tfrac{d-2}{2}+a, \tfrac{d-2}{2}+b). $$

\textbf{Proof of \eqref{eqn:thm main bi wave:LinfL2 with weight} case 2: high-high interactions.} An application of the product estimate in Lemma \ref{lem:elementary product est} together with Bernstein's inequality gives
  \begin{align*}
   \lambda_1^{\ell + 1 -a} \| (\lambda_1 + |\p_t|)^a ( \overline{\varphi}_{\lambda_1} \psi_{\lambda_2}) \|_{L^1_t L^2_x}
                &\lesa \lambda_1^{\ell + 1 - 2a} \lambda_1^{\frac{d-4}{2}} \| ( \lambda_1 + |\p_t|)^a \varphi_{\lambda_1} \|_{L^2_t L^{2^*}_x} \| ( \lambda_2 + |\p_t|)^a \psi_{\lambda_2} \|_{L^2_t L^{2^*}_x} \\
                &\lesa \lambda_1^{\ell + \frac{d-2}{2} + 2a - 2s} \| \varphi_{\lambda_1} \|_{S^{s,a,0}_{\lambda_1}} \| \psi_{\lambda_2} \|_{S^{s,a,0}_{\lambda_2}}.
  \end{align*}
On the other hand, since $\ell + 1 -a \g 0$, we have
    \begin{align*}
      \Big( \sum_{\mu \lesa \lambda_1} \mu^{2(\ell + 1-a)} \| (\mu + |\p_t|)^a P_{\mu} P^{(t)}_{\ll \mu^2} F \|_{L^1_t L^2_x}^2\Big)^\frac{1}{2}
                &\lesa \lambda_1^{\ell + 1 -a} \Big( \sum_{\mu \lesa \lambda_1} \| (\lambda_1 + |\p_t|)^a P_{\mu} F \|_{L^1_t L^2_x}^2\Big)^\frac{1}{2} \\
                &\lesa \lambda_1^{\ell + 1 - a } \| ( \lambda_1 + |\p_t|)^a F \|_{L^1_t L^2_x}.
    \end{align*}
Therefore summing up gives
    \begin{align*}
      \Big( \sum_{\mu \gg 1 } \mu^{2(\ell + 1 -a)} \Big\| &\sum_{\lambda_1\approx \lambda_2 \gtrsim \mu} (\mu + |\p_t|)^a P_\mu P^{(t)}_{\ll \mu^2}(\overline{\varphi}_{\lambda_1} \psi_{\lambda_2}) \Big\|_{L^1_t L^2_x}^2 \Big)^\frac{1}{2} \\
            &\lesa \sum_{\lambda_1 \approx \lambda_2} \lambda_1^{\ell + 1 -a} \| (\lambda_1 + |\p_t|)^a (\overline{\varphi}_{\lambda_1} \psi_{\lambda_2}) \|_{L^1_t L^2_x} \\
            &\lesa \sum_{\lambda_1 \approx \lambda_2} \lambda_1^{\ell + \frac{d-2}{2} + 2a - 2s} \| \varphi_{\lambda_1} \|_{S^{s,a,0}_{\lambda_1}} \| \psi_{\lambda_2} \|_{S^{s,a,0}_{\lambda_2}} \lesa \| \varphi \|_{S^{s,a,b}} \| \psi \|_{S^{s,a,b}}
    \end{align*}
where we used the assumption
    $$ 2s - \ell  - \tfrac{d-2}{2} \g 2a. $$
This completes the proof of \eqref{eqn:thm main bi wave:LinfL2 with weight}.

\textbf{Proof of \eqref{eqn:thm main bi wave:LinfL2}.} This is slightly easier than the previous estimate \eqref{eqn:thm main bi wave:LinfL2 with weight} as we no longer have to deal with the temporal weight $(\mu + |\p_t|)^a$. To bound the high-low interactions, we observe that
    \begin{align*}
        \mu^{\ell -1 } \| P_{\mu}(\overline{ \varphi} \psi_{\ll \mu }) \|_{L^\infty_t L^2_x} \lesa \mu^{\ell -1 } \sum_{\lambda \ll \mu} \lambda^{\frac{d}{2}} \| \varphi_{\approx \mu} \|_{L^\infty_t L^2_x} \| \psi_\lambda \|_{L^\infty_t L^2_x} &\lesa \mu^{\ell - s -1} \sum_{\lambda \ll \mu} \lambda^{\frac{d}{2}-s} \| \varphi_{\approx \mu} \|_{S^{s,a,0}} \| \psi \|_{S^{s,a,0}}
   \end{align*}
and hence provided that
    $$ s+1\g  \ell, \qquad 2s\g \ell + \frac{d-2}{2}, \qquad (s, \ell) \not = (\tfrac{d}{2}, \tfrac{d}{2} + 1),$$
we obtain
   \begin{align*}
     \Big( \sum_{\mu \gg 1} \mu^{2(\ell -1)}  \| P_{\mu}(\overline{ \varphi} \psi_{\ll \mu }) \|_{L^\infty_t L^2_x}^2 \Big)^\frac{1}{2}
                &\lesa \Big( \sum_{\mu \in 2^\NN} \mu^{2s} \| \varphi_{\approx \mu} \|_{L^\infty_t L^2_x}^2\Big)^\frac{1}{2} \| \psi \|_{S^{s,a,b}} \lesa \| \varphi \|_{S^{s,a,b}} \| \psi \|_{S^{s,a,b}}.
   \end{align*}
Similarly, to deal with the high-high interactions, we note that for any $\lambda_1 \approx \lambda_2$ since $\ell + \frac{d}{2} -1 \g 0$ an application of Bernstein's inequality gives
        \begin{align*}
             \Big( \sum_{\mu \lesa \lambda_1} \mu^{2(\ell - 1)} \| P_{\mu} ( \overline{\varphi}_{\lambda_1} \psi_{\lambda_2} ) \|_{L^\infty_t L^2_x}^2\Big)^\frac{1}{2}
                        &\lesa \Big( \sum_{\mu \lesa \lambda_1} \mu^{2(\ell +\frac{d}{2}- 1)} \| \varphi_{\lambda_1}\|_{L^\infty_t L^2_x}^2 \| \psi_{\lambda_2} \|_{L^\infty_t L^2_x}^2\Big)^\frac{1}{2} \\
                        &\lesa \lambda_1^{\ell + \frac{d-2}{2} - 2s} \| \varphi_{\lambda_1} \|_{S^{s,a,b}_{\lambda_1}} \| \psi_{\lambda_2} \|_{S^{s,a,b}_{\lambda_2}}
        \end{align*}
and therefore
        \begin{align*}
          \Big( \sum_{\mu \in 2^\NN} \mu^{2(\ell - 1)} \Big\| \sum_{\lambda_1 \approx \lambda_2 \gtrsim \mu} P_{\mu} ( \overline{\varphi}_{\lambda_1} \psi_{\lambda_2} ) \Big\|_{L^\infty_t L^2_x}^2\Big)^\frac{1}{2}
                    &\lesa \sum_{\lambda_1 \approx \lambda_2} \lambda_1^{\ell + \frac{d-2}{2}-2s}  \| \varphi_{\lambda_1} \|_{S^{s,a,b}_{\lambda_1}} \| \psi_{\lambda_2} \|_{S^{s,a,b}_{\lambda_2}} \lesa \| \varphi \|_{S^{s,a,b}} \| \psi \|_{S^{s,a,b}}
        \end{align*}
where we used the assumption
        $$  2s - \ell - \tfrac{d-2}{2} \g 0. $$
In view of the frequency decomposition \eqref{eqn:thm main bi wave:freq decomp}, together with the invariance of the left hand side of \eqref{eqn:thm main bi wave:LinfL2} under complex conjugation, this completes the proof of the $L^\infty_t L^2_x$ bound \eqref{eqn:thm main bi wave:LinfL2}.

\textbf{Proof of \eqref{eqn:thm main bi wave:L2}.} We now turn to the proof of the $L^2_{t,x}$ bound \eqref{eqn:thm main bi wave:L2}, and again decompose the product into the standard frequency trichotomy as in \eqref{eqn:thm main bi wave:freq decomp}. For the high-low interaction terms, we note that
    $$ \mu^{\beta} \| P_{\mu} (\overline{\varphi} \psi_{\ll \mu}) \|_{L^2_{t,x}} \lesa \mu^{\beta} \| \varphi_{\approx \mu} \|_{L^\infty_t L^2_x} \| \psi_{\ll \mu} \|_{L^2_t L^\infty_x} \lesa \mu^{ \beta} \sum_{1\les \lambda \ll \mu} \lambda^{\frac{d-2}{2} + a - s} \| \varphi_{\approx \mu} \|_{L^\infty_t L^2_x} \| (\lambda + |\p_t|)^a \psi_{\lambda} \|_{L^2_t L^{2^*}_x}$$
and hence, provided that
    $$ \beta \les s, \qquad 2s - \beta - \frac{d-2}{2} \g a, \qquad (s, \beta) \not = (\tfrac{d-2}{2} + a, \tfrac{d-2}{2}+a),$$
summing up gives
    \begin{align*}
      \Big( \sum_{\mu  \gg 1} \mu^{2\beta} \| P_{\mu}( \overline{\varphi} \psi_{\ll \mu}) \|_{L^2_{t,x}}^2\Big)^\frac{1}{2}
                    &\lesa \Big( \sum_{\mu \gg 1} \mu^{2s} \| \varphi_{\approx \mu}\|_{L^\infty_t L^2_x}^2\Big)^\frac{1}{2} \| \psi \|_{S^{s,a,b}} \lesa \| \varphi \|_{S^{s,a,b}} \| \psi \|_{S^{s,a,b}}.
    \end{align*}
Similarly, to bound the high-high interaction terms, we have for any $\lambda_1 \approx \lambda_2$
    $$ \lambda_1^{\beta} \| \overline{\varphi}_{\lambda_1} \psi_{\lambda_2} \|_{L^2_{t,x}} \lesa \lambda_1^\beta \| \varphi_{\lambda_1} \|_{ L^2_t L^{2^*}_x} \|  \psi_{\lambda_2} \|_{L^\infty_t L^d_x} \lesa \lambda_1^{\beta + a + \frac{d-2}{2} - 2s} \| \varphi_{\lambda_1} \|_{S^{s,a,0}_{\lambda_1}} \| \psi_{\lambda_2} \|_{S^{s,a,0}_{\lambda_2}}. $$
Therefore, noting that since $\beta \g 0$ we have
    $$ \Big( \sum_{\mu \lesa \lambda_1} \mu^{2\beta} \| P_{\mu} F \|_{L^2_{t,x}}^2\Big)^\frac{1}{2} \lesa \lambda_1^\beta \Big( \sum_{\mu \in 2^\NN} \| P_\mu F \|_{L^2_{t,x}}^2\Big)^{\frac{1}{2}} \lesa \lambda_1^{\beta} \|F \|_{L^2_{t,x}}$$
we conclude that
    \begin{align*}
      \Big( \sum_{\mu \in 2^\NN} \mu^{2\beta} \Big\| \sum_{\lambda_1 \approx \lambda_2 \gtrsim \mu} P_{\mu}(\overline{\varphi}_{\lambda_1} \psi_{\lambda_2}) \Big\|_{L^2_{t,x}}^2\Big)^\frac{1}{2}
        &\lesa \sum_{\lambda_1 \approx \lambda_2} \Big( \sum_{\mu \lesa \lambda_1} \mu^{2\beta} \| P_{\mu}(\overline{\varphi}_{\lambda_1} \psi_{\lambda_1}) \|_{L^2_{t,x}}^2\Big)^\frac{1}{2} \\
        &\lesa \sum_{\lambda_1 \approx \lambda_2} \lambda_1^{\beta + a + \frac{d-2}{2} - 2s} \| \varphi_{\lambda_1} \|_{S^{s,a,b}_{\lambda_1}} \| \psi_{\lambda_2} \|_{S^{s,a,b}_{\lambda_2}} \lesa \| \varphi \|_{S^{s,a,b}} \| \psi \|_{S^{s,a,b}}
    \end{align*}
provided that
    $$ 2s - \beta - \frac{d-2}{2} \g a. $$
This completes the proof of \eqref{eqn:thm main bi wave:L2}.

\textbf{Proof of \eqref{eqn:thm main bi wave:low freq}.} To prove the remaining estimate \eqref{eqn:thm main bi wave:low freq}, we can simply use Bernstein and H\"older inequalities and the endpoint Strichartz estimate with loss \eqref{eqn:stri with loss}
    \[
\|P_{\les 2^{16}}(|\nabla|(\overline{\varphi} \psi))\|_{L^1_tL^2_x}\lesa \|\overline{\varphi} \psi\|_{L^1_tL^{\frac{d}{d-2}}_x}\lesa \|\varphi\|_{L^2_tL^{2^*}_x}\|\psi\|_{L^2_tL^{2^*}_x}\lesa\| \varphi \|_{S^{a,a, b}} \| \psi\|_{S^{a,a, b}},
\]
since $s\geq a$.
\end{proof}

As in the Schr\"odinger case, we additionally provide a local version of the bilinear estimate which contains a dispersive norm.

\begin{corollary}\label{cor:loc bi wave}

Let $d\g 4$, $s , \ell, \beta \g 0$, and $0\les a \les 1$ satisfy
        $$ \beta < \min\big\{s, 2s - \tfrac{d-2}{2} -a\big\}, \qquad  2a < 2s - \ell - \tfrac{d-2}{2}, \qquad a < s-\ell.$$
There exists $0< \theta < 1$ and $C>0$ such that for any interval $0\in I \subset \RR$, if $\varphi, \psi \in S^{s, a, 0}(I)$,
then
        $$ \|\mc{J}_0[ \nabla (\overline{\varphi}\psi)]\|_{W^{\ell, a, \beta}(I)} \les  C \Big(\| \varphi \|_{S^{s,a, 0}(I)} \| \psi\|_{S^{s,a, 0}(I)}\Big)^{1-\theta} \Big( \| \varphi \|_{L^2_t L^{2^*}_x(I\times \RR^d)} \| \psi \|_{L^2_t L^{2^*}_x(I\times \RR^d)}\Big)^{\theta}.$$
\end{corollary}
\begin{proof}
Let $\lambda_1, \lambda_2 \in 2^\NN$. It suffices to show that there exists $\delta, N>0$ such that
        \begin{align}
            \| \mc{J}_0[|\nabla|(\overline{\psi_{\lambda_1}} \varphi_{\lambda_2})] \|_{W^{\ell, a, \beta}} &\lesa (\max\{\lambda_1, \lambda_2\})^{-\delta} \| \psi \|_{S^{s,a,0}} \| \varphi \|_{S^{s,a,0}}\label{eqn:cor loc bi wave:gain}
        \end{align}
together with an estimate with a derivative loss, but the Strichartz norm on the righthand side
        \begin{equation}\label{eqn:cor loc bi wave:stri}
            \begin{split}
            \| \mc{J}_0[|\nabla|&(\overline{\psi_{\lambda_1}} \varphi_{\lambda_2})] \|_{W^{\ell, a, \beta}(I)} \\
            &\lesa (\max\{\lambda_1, \lambda_2\})^{N} \Big(\| \psi \|_{L^2_t L^{2^*}_x(I\times \RR^d)} \| \varphi \|_{L^2_t L^{2^*}_x(I\times \RR^d)} \| \psi \|_{S^{s, a, 0}(I)} \| \varphi \|_{S^{s,a,0}(I)}\Big)^\frac{1}{2}.
            \end{split}
        \end{equation}
We start with the proof of \eqref{eqn:cor loc bi wave:gain}. Choose $s'<s$ such that
        $$ \beta < \min\big\{s', 2s' - \tfrac{d-2}{2} -a\big\}, \qquad  2a < 2s' - \ell - \tfrac{d-2}{2}, \qquad a < s'-\ell.$$
An application of Theorem \ref{thm:main bilinear wave} implies that
    $$ \| \mc{J}_0[ |\nabla|(\overline{\psi_{\lambda_1}} \varphi_{\lambda_2}) \|_{W^{\ell,a, \beta}} \lesa \| \psi_{\lambda_1} \|_{S^{s', a, 0}} \| \varphi_{\lambda_2} \|_{S^{s', a,0}} \lesa (\lambda_1 \lambda_2)^{s'-s} \| \psi \|_{S^{s,a,0}} \| \varphi \|_{S^{s,a,0}}$$
and hence \eqref{eqn:cor loc bi wave:gain} follows.

We now turn to the proof of \eqref{eqn:cor loc bi wave:stri}. An application of the standard energy inequality for the wave equation together with the convexity of $L^p_t$ and Bernstein's inequality implies that
    \begin{align*}
      \| \mc{J}_0[G_\lambda] \|_{W^{\ell, a, \beta}_\lambda}
      &\lesa \lambda^{\ell + a} \| \mc{J}_0[G_\lambda] \|_{L^\infty_t L^2_x} + \lambda^{\beta - 1} \|  G_\lambda \|_{L^2_{t,x}} \\
      &\lesa \lambda^{\ell + a + \beta + \frac{d}{4}} \Big( \| G_\lambda \|_{L^1_t L^2_x} + \| G_{\lambda} \|_{L^1_t L^2_x}^\frac{1}{2} \| G_\lambda \|_{L^\infty_t L^1_x}^\frac{1}{2}\Big)
    \end{align*}
Therefore there exists $N>0$ such that
    \begin{align*}
       &\| \mc{J}_0[|\nabla|(\overline{\psi_{\lambda_1}} \varphi_{\lambda_2})] \|_{W^{\ell, a, \beta}(I)}
       \les  \| \mc{J}_0[\ind_{I} |\nabla|(\overline{\psi_{\lambda_1}} \varphi_{\lambda_2})] \|_{W^{\ell, a, \beta}}  \\
       \lesa{}& (\max\{\lambda_1, \lambda_2\})^N \Big(\| \psi_{\lambda_1} \varphi_{\lambda_2} \|_{L^1_t L^2_x(I\times \RR^d)} +\| \psi_{\lambda_1} \varphi_{\lambda_2} \|_{L^1_t L^2_x(I\times \RR^d)}^\frac{1}{2} \| \psi_{\lambda_1} \varphi_{\lambda_2} \|_{L^\infty_t L^1_x(I\times \RR^d)}^\frac{1}{2} \Big)\\
            \lesa{}& (\max\{\lambda_1, \lambda_2\})^{N}  \Big(\| \psi \|_{L^2_t L^{2^*}_x(I\times \RR^d)} \| \varphi \|_{L^2_t L^{2^*}_x(I\times \RR^d)} \| \psi \|_{S^{s, a, 0}(I)} \| \varphi \|_{S^{s,a,0}(I)}\Big)^\frac{1}{2},
    \end{align*}
    and the proof is complete.
\end{proof}

\section{Simplified small data global and large data local theory}\label{sec:simple}

As a warm up to the proof of the main results contained in Theorem \ref{thm:lwp-ill} and Theorem \ref{thm:small data gwp}, we show how the bilinear estimates in the previous two sections can be used to prove a simplified small data global well-posedness and scattering result and a large data local well-posedness result in the non-endpoint case.

 Recall that $\mc{I}_{0}[F]$ denotes the solution to the inhomogeneous Schr\"odinger equation
        $$ (i\p_t + \Delta) \psi = F, \qquad \psi(0) = 0 $$
and similarly, $\mc{J}_0[G]$ denotes the solution to the inhomogeneous wave equation
        $$ (i\p_t + |\nabla|)\phi = G, \qquad \phi(0) = 0. $$
Given data $(f, g) \in H^s \times H^\ell$, define the functional
    $$ \Phi(f, g; \psi) := e^{it\Delta}f+  \mc{I}_{0}\Big[\Re(e^{it|\nabla|}g)\psi\Big]- \mc{I}_{0}\Big[ \Re\Big(\mc{J}_{0}(|\nabla||\psi|^2)\Big) \psi\Big].$$
Suppose that $(s, \ell)$ lies in the region \eqref{eqn:cond on s l}, and define the parameters $(a,b)$ as in \eqref{eqn:choice of a,b}. A computation shows that the energy inequality in Lemma \ref{lem:energy ineq} together with the bilinear estimates in Theorem \ref{thm:main bilinear schro} and Theorem \ref{thm:main bilinear wave} implies that we have the bound
    \begin{align*}
         \| \Phi(f,g; \psi) \|_{S^{s, a, b}}
            &\lesa \| f \|_{H^s} + \big\| \Re(e^{it|\nabla|}g)\psi \big\|_{N^{s,a,b}}  + \big\| \Re\big(\mc{J}_{0}(|\nabla||\psi|^2)\big)\psi \big\|_{N^{s,a,b}}\\
            &\lesa \| f \|_{H^s} + \| g \|_{W^{\ell, a, s-\frac{1}{2}}} \| \psi \|_{S^{s,a,b}}  + \big\| \mc{J}_{0}(|\nabla||\psi|^2)\big\|_{W^{\ell, a, s-\frac{1}{2}}} \| \psi \|_{S^{s, a, b}}  \\
            &\lesa \| f \|_{H^s} + \| g \|_{H^\ell} \|\psi \|_{S^{s,a,b}}  +  \| \psi \|_{S^{s, a, b}}^3.
    \end{align*}
Therefore $\Phi: S^{s,a,b} \to S^{s,a,b}$. Repeating this argument with differences, shows that provided $\| f \|_{H^s} + \| g \|_{H^\ell}$ is sufficiently small, there exists a fixed point $u \in \{ \psi \in S^{s, a, b} \mid \| \psi \|_{S^{s,a,b}}\lesa \| f \|_{H^s} \}$ to $\Phi$. Defining
        $$ V = e^{it|\nabla|} g - \mc{J}_0(|\nabla| |u|^2)$$
and again applying Theorem \ref{thm:main bilinear wave}, we then obtain a solution $(u, V) \in C(\RR, H^s\times H^\ell)$ to the Zakharov system \eqref{eq:Zakharov 1st order}. The scattering property follows from Lemma \ref{lem:scattering} and a similar analogue for the wave part.

Note that the above argument requires that we have the smallness condition $\| f\|_{H^s} + \| g \|_{H^\ell} \ll 1$. Our later arguments will significantly improve this to just requiring $g\in H^{\frac{d-4}{2}}$ and $\| f\|_{H^{\frac{d-3}{2}}} \ll_g 1$. In other words we only require smallness of $f$ in the \emph{endpoint} Sobolev space. In addition, we also obtain a stronger uniqueness claim, as well as persistence of regularity. \\

Let us now sketch a simplified, large data local well-posedness result in the non-endpoint case $s>\frac{d-3}{2}$. Suppose that $(s, \ell)$ satisfies \eqref{eqn:cond on s l} with $s>\frac{d-3}{2}$, and take $(a, b)$ as in \eqref{eqn:choice of a,b}. Define $\tilde{\ell} = \min\{ s - \frac{1}{2}, \ell\}$ and take the map $\Phi$ as above. The non-endpoint condition  $s>\frac{d-3}{2}$ is due to the use of Corollary \ref{cor:loc bi wave}, while the choice of $\tilde{\ell}$ is made to ensure that we can construct a fixed point for $\Phi$ in $S^{s, a,0}(I)$ via Corollary \ref{cor:local schro bi}. Once we have a fixed point $u \in S^{s,a,0}(I)$, we use an additional argument to upgrade this to $u\in S^{s,a,b}(I)$, which is needed to get the correct regularity for the wave component $V$.

Fix $(f,g)\in H^s \times H^\ell$. Choose an interval $0\in I \subset \RR$ such that
            $$ \| e^{it\Delta} f \|_{L^2_t L^{2^*}_x(I\times \RR^d)} + \| e^{it|\nabla|} g \|_{W^{\tilde{\ell}, a, s-\frac{1}{2}}(I) + L^2_t W^{s,d}_x(I\times \RR^d)} < \epsilon, $$
where $\epsilon > 0$ is fixed later (depending on $f$, $g$, and the absolute constants in the above bilinear estimates). Define the subset $\Omega \subset S^{s,a,0}(I)$ as
        $$ \Omega = \big\{ \psi \in S^{s,a,0}(I) \,\, \big|\,\, \|\psi\|_{L^2_t L^{2^*}_x(I\times \RR^d)}\lesa ( 1 + \| g \|_{L^2_x}) \epsilon, \,\, \| \psi \|_{S^{s,a,0}(I)} \lesa \| f \|_{H^s} \big\}. $$
An application of Corollary \ref{cor:local schro bi} and Corollary \ref{cor:loc bi wave} gives $\theta > 0$ such that for every $\psi \in \Omega$ we have
    \begin{align*}
      \| \Phi(f,g;\psi)\|_{S^{s,a,0}(I)} &\lesa \| f \|_{H^s} + \Big(\| e^{it|\nabla|} g \|_{W^{\tilde{\ell}, a, s-\frac{1}{2}}(I) + L^2_t W^{s, d}_x(I)} + \big\| \mc{J}_{0}(|\nabla||\psi|^2)\big\|_{W^{\tilde{\ell}, a, s-\frac{1}{2}}(I)} \Big) \| \psi \|_{S^{s, a, 0}(I)}\\
                        &\lesa   \| f \|_{H^s} + \epsilon  \| \psi \|_{S^{s, a, 0}(I)}  + \epsilon^{2\theta} ( 1 + \| g \|_{L^2_x})^{2\theta} \| \psi \|_{S^{s,a, 0}(I)}^{3 - 2 \theta}.
    \end{align*}
On the other hand, in view of the endpoint Strichartz estimate we have
    \begin{align*}
        \| \Phi(f, g; \psi) \|_{L^2_t L^{2^*}_x(I\times \RR^d)} &\lesa \| e^{it\Delta} f \|_{L^2_t L^{2^*}_x(I\times \RR^d)}  + \| \Re(e^{it|\nabla|}g) \psi \|_{L^2_t L^{2_*}_x(I\times \RR^d)} + \| \Re( \mc{J}_0[|\nabla| |\psi|^2]) \psi\|_{L^2_t L^{2_*}_x} \\
        &\lesa \epsilon + \| g\|_{L^2_x} \| \psi \|_{L^2_t L^{2^*}_x(I\times \RR^d)} + \| \mc{J}_0[|\nabla| |\psi|^2] \|_{L^\infty_t L^2_x(I\times \RR^d)} \| \psi \|_{L^2_t L^{2^*}_x(I\times \RR^d)}\\
        &\lesa \epsilon + \epsilon \| g \|_{L^2_x} + \epsilon^{1 + 2\theta} ( 1 + \| g \|_{L^2})^{1 + 2\theta} \| \psi \|_{S^{s,a,0}(I)}^{2 - 2 \theta}.
    \end{align*}
Consequently, choosing $\epsilon>0$ sufficiently small, we see that $\Phi: \Omega \to \Omega$. A similar argument shows that $\Phi$ is a contraction on $\Omega$ (with respect to the norm $\| \cdot \|_{S^{s,a,0}(I)}$), and hence there exists a fixed point $u\in \Omega \subset S^{s,a,0}(I)$ for $\Phi$.

We now upgrade the (far paraboloid) regularity to $u \in S^{s,a, b}(I)$. Note that this is immediate if $s>\ell$ since $b=0$ in this case. If $s\les \ell$, then an application of Theorem \ref{thm:main bilinear schro} together with Lemma \ref{lem:energy ineq} gives
   \begin{align*}
     \| u \|_{S^{s,a, b}(I)} &\lesa \| f \|_{H^s} + \| \Re(e^{it|\nabla|} g ) u \|_{N^{s,a,b}(I)} +  \| \Re(\mc{J}_0[|\nabla| |u|^2]) u \|_{N^{s,a,b}(I)}  \\
                             &\lesa \| f \|_{H^s} + \| g \|_{H^\ell} \| u \|_{S^{s,a,0}(I)} + \| \mc{J}_0[|\nabla| |u|^2] \|_{W^{s- \frac{1}{2}(1-b), a, s-\frac{1}{2}}(I)} \| u \|_{S^{s,a, 0}(I)}.
   \end{align*}
To check conditions of Theorem \ref{thm:main bilinear schro}, it is helpful to note that $a=0$ when $s\les \ell$. Theorem \ref{thm:main bilinear wave} implies
        $$ \| \mc{J}_0[|\nabla| |u|^2] \|_{W^{s- \frac{1}{2}(1-b), a, s-\frac{1}{2}}(I)} \lesa \| u \|_{S^{s,a,0}(I)}^2$$
and hence we conclude that
        $$ \| u \|_{S^{s, a, b}(I)} \lesa \| f \|_{H^s} + \| g \|_{H^\ell} \| u \|_{S^{s,a,0}(I)} + \| u \|_{S^{s,a,0}(I)}^3. $$
Consequently, if  $u\in S^{s,a,0}(I)$ is a solution to $\Phi(f,g;u) = u$, then we have the improved regularity $u\in S^{s,a,b}(I)$. As above, we now define
   $$ V = e^{it|\nabla|} g - \mc{J}_0(|\nabla| |u|^2).$$
Since $u\in S^{s,a,b}(I)$, an application of Theorem \ref{thm:main bilinear wave} then gives $V \in W^{\ell, a, s-\frac{1}{2}}(I)$. In particular, we have a local solution $(u, V) \in C(I, H^s\times H^\ell)$ to the Zakharov system \eqref{eq:Zakharov 1st order}.

\section{Local bilinear estimates in the endpoint case}\label{sec:local-bil}

In this section, we deal with the endpoint case $(s,\ell) = (\frac{d-3}{2}, \frac{d-4}{2})$ and establish bilinear estimates which include both dispersive norms and a slightly weaker frequency summation on the right hand side. Define the norms
        $$ \| u \|_{S^{s,0,0}_w} = \| u \|_{L^\infty_t H^s_x} + \| u \|_{L^2_t W^{s,2^*}_x} + \| (i\p_t + \Delta) u \|_{L^2_t H^{s-1}_x} $$
and 
        $$ \| v \|_{W^{\ell, 0, \beta}_w} = \| v \|_{L^\infty_t H^\ell_x} + \| (i\p_t + |\nabla|) v \|_{L^2_t H^{\beta-1}_x}. $$
The notation here is chosen to match that introduced earlier. In particular we have $\| u \|_{S^{s,0,0}_w} \lesa \| u \|_{S^{s,0,0}}$ and $\|v \|_{W^{\ell, 0, \beta}_w} \lesa \| v \|_{W^{\ell, 0, \beta}}$.
We start with an improvement of Theorem \ref{thm:main bilinear schro} in the endpoint case.

\begin{proposition}\label{prop:schro est for endpoint}
Let $d\g 4$ and $(s, \ell) =(\frac{d-3}{2}, \frac{d-4}{2})$. Let $I\subset \RR$ be an interval. Then
        \begin{equation}\label{eqn:prop schro est end:main}
        \| v u \|_{N^{s,0,0}(I)} \lesa \|v \|_{W^{\ell, 0, \ell}_w(I)} \| u \|_{L^2_t W^{s, 2^*}_x(I\times \RR^d)}^{\frac{1}{2}} \| u \|_{S^{s,0,0}_w(I)}^\frac{1}{2}.
        \end{equation}
\end{proposition}
\begin{proof} 
Suppose for the moment that we can prove that for any $\alpha \gg 1$ we have
    \begin{align}
      \Big\| \sum_{\lambda \g \alpha} P_\lambda( v u_{\frac{\lambda}{\alpha}}) \Big\|_{N^{s,0,0}(I)}
                &\lesa \alpha^\frac{1}{2} \| v \|_{L^\infty_t H^\ell_x(I\times \RR^d)} \| u \|_{L^2_t W^{s,2^*}_x(I\times \RR^d)}\label{eqn:prop schro est end:alpha loss}\\
     \Big\| \sum_{\lambda \g \alpha} P_\lambda( v u_{\frac{\lambda}{\alpha}}) \Big\|_{N^{s,0,0}}
                &\lesa \alpha^{-\frac{1}{2}} \| v \|_{W^{\ell, 0, \ell}_w} \| u \|_{S^{s,0,0}_w}.
                \label{eqn:prop schro est end:alpha gain}
    \end{align}
Since
     \begin{align*}
    \Big\| \sum_{\lambda} P_\lambda( v u_{\gtrsim \lambda}) \Big\|_{N^{s,0,0}(I)} &\lesa \Big( \sum_{\lambda} \lambda^{2s} \|  \ind_I v  u_{\gtrsim \lambda}  \|_{L^2_t L^{2_*}_x}^2 \Big)^\frac{1}{2}\\
            &\lesa \Big\| v \Big(\sum_{\lambda} \lambda^{2s} |u_{\gtrsim \lambda}|^2\Big)^\frac{1}{2} \Big\|_{L^2_t L^{2_*}_x(I\times \RR^d)} \\
            &\lesa \| v \|_{L^\infty_t L^{\frac{d}{2}}_x(I\times \RR^d)} \Big\| \Big(\sum_{\lambda} \lambda^{2s} |u_{\gtrsim \lambda}|^2\Big)^\frac{1}{2} \Big\|_{L^2_t L^{2^*}_x(I\times \RR^d)} \\
            &\lesa \| v \|_{L^\infty_t H^\ell_x(I\times \RR^d)} \| u \|_{L^2_t W^{s,2^*}_x(I\times \RR^d)}
    \end{align*}
an application of \eqref{eqn:prop schro est end:alpha loss} and \eqref{eqn:prop schro est end:alpha gain}, together with the definition of the restricted spaces $N^{s,0,0}(I)$, $S^{s,0,0}_w(I)$, and $W^{\ell, 0, \ell}_w(I)$, then implies that for any $M\gg 1$ we have
    \begin{align*}
        \| v u \|_{N^{s,0,0}(I)} &\les \Big\| \sum_{\lambda} P_\lambda( v u_{\gtrsim \lambda}) \Big\|_{N^{s,0,0}(I)} + \sum_{1 \ll \alpha \les M } \Big\| \sum_{\lambda \g \alpha} P_\lambda( v u_{\frac{\lambda}{\alpha}}) \Big\|_{N^{s,0,0}(I)} + \sum_{\alpha > M } \Big\| \sum_{\lambda \g \alpha} P_\lambda( v u_{\frac{\lambda}{\alpha}}) \Big\|_{N^{s,0,0}(I)} \\
            &\lesa M^\frac{1}{2} \| v \|_{W^{\ell, 0, \ell}_w(I)} \| u \|_{L^2_t W^{s,2^*}_x(I\times \RR^d)} + M^{-\frac{1}{2}}  \| v \|_{W^{\ell, 0, \ell}_w(I)} \| u \|_{S^{s,0,0}_w(I)}
    \end{align*}
Optimising in $M$ then gives \eqref{eqn:prop schro est end:main}. Thus it remains to prove the bounds \eqref{eqn:prop schro est end:alpha loss} and \eqref{eqn:prop schro est end:alpha gain}. For the former estimate, we observe that since $s=\ell + \frac{1}{2}$
    \begin{align*}
        \Big\| \sum_{\lambda \g \alpha} P_\lambda( v u_{\frac{\lambda}{\alpha}}) \Big\|_{N^{s,0,0}(I)} &\lesa
        \Big( \sum_{\lambda \g \alpha} \lambda^{2s} \|  v u_{\frac{\lambda}{\alpha}} \|_{L^2_t L^{2_*}_x(I\times \RR^d)}^2\Big)^\frac{1}{2} \\
                        &\lesa \Big\| \Big( \sum_{\lambda \g \alpha} \lambda^{2s}|v_{\approx \lambda}|^2 | u_{\frac{\lambda}{\alpha}}|^2\big)^\frac{1}{2} \Big\|_{L^2_t L^{2_*}_x(I\times \RR^d)} \\
                        &\lesa \alpha^\frac{1}{2} \Big\| \sup_{\mu} \mu^\ell |v_\mu| \Big( \sum_{\lambda} \lambda |u_\lambda|^2\Big)^\frac{1}{2} \Big\|_{L^2_t L^{2_*}_x(I\times \RR^d)} \\
                        &\lesa \alpha^\frac{1}{2} \| v \|_{L^\infty_t H^\ell_x} \| u \|_{L^2_t W^{s, 2^*}_x(I\times \RR^d)}
    \end{align*}
where the last line followed via H\"older's inequality and Sobolev embedding. The proof of \eqref{eqn:prop schro est end:alpha gain} is more involved, and exploits the fact that the high-low interactions are non-resonant. In particular, since $\lambda \g \alpha \gg 1$, the non-resonant identity
    $$P^N_\lambda(P^{(t)}_{\ll \lambda^2} v u_{\frac{\lambda}{\alpha}}) = P^N_\lambda(P^{(t)}_{\ll \lambda^2} v_{\approx \lambda} P^{(t)}_{\approx \lambda^2} u_{\frac{\lambda}{\alpha}})$$
implies that
    \begin{align}
      &\Big\|\sum_{\lambda \g \alpha} P_\lambda( v u_{\frac{\lambda}{\alpha}}) \Big\|_{N^{s,0,0}}^2
                \lesa  \sum_{\lambda \g \alpha} \lambda^{2s} \| P^N_\lambda( v u_{\frac{\lambda}{\alpha}}) \|_{L^2_t L^{2_*}_x}^2  + \sum_{\lambda \g \alpha} \lambda^{2(s-1)}\| P^F_\lambda( v u_{\frac{\lambda}{\alpha}}) \|_{L^2_{t,x}}^2 \notag \\
                \lesa{}&  \sum_{\lambda \g \alpha} \lambda^{2s} \| P^{(t)}_{\ll \lambda^2} v_{\approx \lambda} P^{(t)}_{\approx \lambda^2} u_{\frac{\lambda}{\alpha}}\|_{L^2_t L^{2_*}_x}^2  +  \sum_{\lambda \g \alpha} \lambda^{2s} \| P^{(t)}_{\gtrsim \lambda^2} v_{\approx \lambda} u_{\frac{\lambda}{\alpha}} \|_{L^2_t L^{2_*}_x}^2  +  \sum_{\lambda \g \alpha} \lambda^{2(s-1)}\| v_{\approx \lambda} u_{\frac{\lambda}{\alpha}} \|_{L^2_{t,x}}^2 . \label{eqn:prop schro est end:decomp II}
    \end{align}

To estimate the first term in \eqref{eqn:prop schro est end:decomp II}, we observe that since $ s= \ell + \frac{1}{2}$, we have
    \begin{align*}
      \Big( \sum_{\lambda \g \alpha} \lambda^{2s} \|  P^{(t)}_{\ll \lambda^2} v_{\approx \lambda}   P^{(t)}_{\approx \lambda^2} u_{\frac{\lambda}{\alpha}}  \|_{L^2_t L^{2_*}_x}^2 \Big)^\frac{1}{2}
                    &=  \Big\| \Big( \sum_{\lambda \g \alpha} \lambda^{2s} \| P^{(t)}_{\ll \lambda^2} v_{\approx \lambda}  P^{(t)}_{\approx \lambda^2} u_{\frac{\lambda}{\alpha}} \|_{L^{2_*}_x}^2\Big)^\frac{1}{2} \Big\|_{L^2_t}\\
                    &\lesa \| v \|_{L^\infty_t H^\ell_x} \Big( \sum_{\lambda \g \alpha} \lambda^{2(s-\ell)} \| P^{(t)}_{\approx \lambda} u_{\frac{\lambda}{\alpha}} \|_{L^2_t L^d_x}^2\Big)^\frac{1}{2}\\
                    &\lesa \| v \|_{L^\infty_t H^\ell_x} \Big( \sum_{\lambda \g \alpha} \lambda^{2(s-\ell-2)} \|  (i\p_t + \Delta) u_{\frac{\lambda}{\alpha}} \|_{L^2_t L^d_x}^2\Big)^\frac{1}{2} \\
                    & \lesa \alpha^{-\frac{3}{2}} \| v \|_{L^\infty_t H^\ell_x} \| (i\p_t + \Delta) u \|_{L^2_t H^{s-1}_x}.
    \end{align*}
To bound the second term in \eqref{eqn:prop schro est end:decomp II}, again using the fact that $\ell + \frac{1}{2} = s = \frac{d-3}{2}$, we have
    \begin{align*}
      \Big( \sum_{\lambda \g \alpha} \lambda^{2s} \| P^{(t)}_{\gtrsim \lambda^2} v_{\approx \lambda} u_{\frac{\lambda}{\alpha}} \|_{L^2_t L^{2_*}_x}^2 \Big)^\frac{1}{2}
        &\lesa  \Big( \sum_{\lambda \g \alpha} \lambda^{2(s+\frac{1}{2})} \| P^{(t)}_{\gtrsim \lambda^2} v_{\approx \lambda}\|_{L^2_{t,x}}^2 \Big)^\frac{1}{2}
        \sup_{\lambda \g \alpha} \lambda^{-\frac{1}{2}} \| u_{\frac{\lambda}{\alpha}} \|_{L^\infty_t L^{d}_x} \\
        &\lesa \alpha^{-\frac{1}{2}} \Big( \sum_{\lambda} \lambda^{2(s-\frac{3}{2})} \| (i\p_t + |\nabla|) P^{(t)}_{\gtrsim \lambda^2} v_{\approx \lambda} \|_{L^2_{t,x}}^2 \Big)^\frac{1}{2} \|u \|_{L^\infty_t H^s_x} \\
        &\lesa \alpha^{- \frac{1}{2}} \|  (i\p_t + |\nabla|) v \|_{L^2_t H^{\ell-1}_x} \|u \|_{L^\infty_t H^s_x}.
    \end{align*}
Finally, for the last term in \eqref{eqn:prop schro est end:decomp II}, since $s-\ell - 1 = -\frac{1}{2}$, an application of Bernstein's inequality gives
    \begin{align*}
    \Big( \sum_{\lambda \g \alpha} \lambda^{2(s-1)}\| v_{\approx \lambda} u_{\frac{\lambda}{\alpha}} \|_{L^2_{t,x}}^2 \Big)^\frac{1}{2}
            &\lesa \Big\| \Big( \sum_{\lambda \g \alpha} \lambda^{2(s-1)} \Big(\frac{\lambda}{\alpha}\Big)^{d-2} \| v_{\approx \lambda} \|_{L^2_x}^2 \| u_{\frac{\lambda}{\alpha}} \|_{L^{2^*}_x}^2 \Big)^\frac{1}{2} \Big\|_{L^2_t} \\
            &\lesa \alpha^{-\frac{1}{2}} \| v \|_{L^\infty_t H^\ell} \big\| \sup_{\lambda} \lambda^s \| u_\lambda \|_{L^{2^*}_x} \big\|_{L^2_t} \\
            &\lesa \alpha^{-\frac{1}{2}} \|v \|_{L^\infty_t H^\ell} \| u \|_{L^2_t W^{s, 2^*}_x}.
    \end{align*}
\end{proof}

We have a related estimate to deal with the wave nonlinearity.

\begin{proposition}\label{prop:wave est end}
Let $d\g 4$, $ (s, \ell) = (\frac{d-3}{2}, \frac{d-4}{2})$. Let $0\in I\subset \RR$ be an interval. Then,
    \begin{equation}\label{eqn:prop wave est end:main}
            \| \mc{J}_0( |\nabla|(\overline{\varphi}\psi) ) \|_{W^{\ell, 0, \ell}(I)} \lesa \Big(\| \varphi \|_{L^2_t W^{s, 2^*}_x(I\times \RR^d)} \| \psi \|_{L^2_t W^{s, 2^*}_x(I\times \RR^d)}\Big)^\frac{1}{2} \Big(\| \varphi \|_{S^{s,0,0}_w(I)} \| \psi \|_{S^{s,0,0}_w(I)} \Big)^\frac{1}{2}
    \end{equation}
\end{proposition}
\begin{proof} 
Suppose for the moment that for any $\alpha \gg 1$ we have the bounds
    \begin{align}
      \Big\| \sum_{\lambda \g \alpha} |\nabla| \mc{J}_0( \overline{\varphi}_\lambda \psi_{\frac{\lambda}{\alpha}} ) \Big\|_{W^{\ell, 0, \ell}(I)}
&\lesa \, \alpha^{\frac{1}{2}} \|\varphi \|_{L^2_t W^{s, 2^*}_x(I\times \RR^d)} \| \psi \|_{L^2_t W^{s, 2^*}_x(I\times \RR^d)} \notag \\
       &+ \alpha^{-\frac{1}{2}} \Big( \|\varphi \|_{L^2_t W^{s, 2^*}_x(I\times \RR^d)} \| \psi \|_{L^2_t W^{s, 2^*}_x(I\times \RR^d)} \|\varphi \|_{S^{s,0,0}_w(I)} \| \psi \|_{S^{s,0,0}_w(I)}\Big)^\frac{1}{2}, \label{eqn:prop wave est end:loss} \\
    \Big\| \sum_{\lambda \g \alpha} |\nabla| \mc{J}_0( \overline{\varphi}_\lambda \psi_{\frac{\lambda}{\alpha}} ) \Big\|_{W^{\ell, 0, \ell}(I)}
&\lesa \, \alpha^{-\frac{1}{2}} \|\varphi \|_{S^{s,0,0}_w(I)} \| \psi \|_{S^{s,0,0}_w(I)} ,\label{eqn:prop wave est end:gain} \\
      \Big\| \sum_{\lambda} |\nabla| \mc{J}_0( \overline{\varphi}_\lambda  \psi_{\approx \lambda}) \Big\|_{W^{\ell, 0, \ell}(I)}
&\lesa \, \Big( \|\varphi \|_{L^2_t W^{s, 2^*}_x(I\times \RR^d)} \| \psi \|_{L^2_t W^{s, 2^*}_x(I\times \RR^d)} \|\varphi \|_{S^{s,0,0}_w(I)} \| \psi \|_{S^{s,0,0}_w(I)}\Big)^\frac{1}{2}. \label{eqn:prop wave est end:mid}
    \end{align}
Let $M \gg 1$. As in Proposition \ref{prop:schro est for endpoint}, by decomposing
    $$ \overline{\varphi} \psi = \sum_{\alpha \gg 1} \sum_{\lambda \g \alpha} \overline{\varphi}_\lambda \psi_{\frac{\lambda}{\alpha}} + \sum_{\lambda} \overline{\varphi}_\lambda \psi_{\approx \lambda} + \sum_{\alpha \gg 1} \sum_{\lambda \g \alpha} \overline{\varphi}_{\frac{\lambda}{\alpha} } \psi_\lambda$$
and using symmetry, an application of \eqref{eqn:prop wave est end:loss} for $\alpha <M$, \eqref{eqn:prop wave est end:gain} for $\alpha \g M$, and
\eqref{eqn:prop wave est end:mid} for the remaining high-high interactions gives
    \begin{align*}
      \| V& \|_{W^{\ell, 0, \ell }(I)} \\
                &\lesa \| |\nabla| \mc{J}_0(\overline{\varphi} \psi) \|_{W^{\ell, 0, \ell}(I)} \\
                &\lesa \sum_{\alpha \gg 1} \Big(\Big\| \sum_{\lambda \g \alpha } |\nabla| \mc{J}_0( \overline{\varphi}_\lambda \psi_{\frac{\lambda}{\alpha}}) \Big\|_{W^{\ell, 0, \ell}(I)} + \Big\| \sum_{\lambda \g \alpha } |\nabla| \mc{J}_0( \overline{\varphi}_{\frac{\lambda}{\alpha}} \psi_\lambda) \Big\|_{W^{\ell, 0, \ell}(I)}\Big) + \Big\| \sum_{\lambda } |\nabla| \mc{J}_0( \overline{\varphi}_\lambda \psi_{\approx \lambda}) \Big\|_{W^{\ell, 0, \ell}(I)} \\
                &\lesa M^{\frac{1}{2}}  \|\varphi \|_{L^2_t W^{s, 2^*}_x(I\times \RR^d)} \| \psi \|_{L^2_t W^{s, 2^*}_x(I\times \RR^d)} + \Big( \|\varphi \|_{L^2_t W^{s, 2^*}_x(I\times \RR^d)} \| \psi \|_{L^2_t W^{s, 2^*}_x(I\times \RR^d)} \|\varphi \|_{S^{s,0,0}_w(I)} \| \psi \|_{S^{s,0,0}_w(I)}\Big)^\frac{1}{2} \\
       &\qquad + M^{-\frac{1}{2}} \| \varphi \|_{S^{s,0,0}_w(I)} \| \psi \|_{S^{s,0,0}_w(I)} +\Big( \|\varphi \|_{L^2_t W^{s, 2^*}_x(I\times \RR^d)} \| \psi \|_{L^2_t W^{s, 2^*}_x(I\times \RR^d)} \|\varphi \|_{S^{s,0,0}_w(I)} \| \psi \|_{S^{s,0,0}_w(I)}\Big)^\frac{1}{2}.
    \end{align*}
Optimising in $M$ then gives \eqref{eqn:prop wave est end:main}.

It remains to prove the bounds  \eqref{eqn:prop wave est end:loss}, \eqref{eqn:prop wave est end:gain}, and \eqref{eqn:prop wave est end:mid}. We begin by noting that an application of H\"older's inequality and the  Sobolev embedding together with the assumptions on $(s,\ell)$ imply that
    \begin{align*}
      \Big( \sum_{\lambda \g \alpha} \lambda^{2(\ell + 1)} \|  \overline{\varphi}_\lambda \psi_{\frac{\lambda}{\alpha}} \|_{L^1_t L^2_x(I\times \RR^d)}^2\Big)^\frac{1}{2}
        &\lesa \alpha^{\ell + 1 -s} \Big\| \Big( \sum_{\lambda} \lambda^{2s} |\varphi_\lambda|^2\Big)^\frac{1}{2} \Big( \sum_{\lambda} \lambda^{2(\ell+1-s)} |\psi_\lambda|^2\Big)^\frac{1}{2}\Big\|_{L^1_t L^2_x(I\times \RR^d)} \\
        &\lesa \alpha^{\ell + 1 -s} \| \varphi \|_{L^2_t W^{s,2^*}_x(I\times \RR^d)} \| \psi \|_{L^2_t W^{\ell +1 -s, d}_x(I\times \RR^d)}\\
         &\lesa  \alpha^{\frac{1}{2}} \| \varphi \|_{L^2_t W^{s,2^*}_x(I\times \RR^d)} \| \psi \|_{L^2_t W^{s, 2^*}_x(I\times \RR^d)}.
    \end{align*}
 On the other hand, again applying a combination of H\"older's inequality and Sobolev embedding, we have
    \begin{align}
        &\Big( \sum_{\lambda \g \alpha} \lambda^{2\ell} \| \ind_I \overline{\varphi}_\lambda \psi_{\frac{\lambda}{\alpha}} \|_{L^2_{t,x}}^2\Big)^\frac{1}{2}\notag \\
                &\lesa \alpha^{\ell - s} \Big\| \Big( \sum_{ \lambda} \lambda^{2s} |\varphi_\lambda|^2\Big)^\frac{1}{2} \sup_\lambda \lambda^{\ell -s} |\psi_\lambda| \Big\|_{L^2_{t,x}(I\times \RR^d)} \notag \\
                &\lesa \alpha^{\ell -s } \Big( \| \varphi \|_{L^2_t W^{s,2^*}_x(I \times \RR^d)} \| \psi \|_{L^\infty_t W^{\ell - s, d}_x(I\times \RR^d)}\Big)^\frac{1}{2}\Big( \| \varphi \|_{L^\infty_t H^s_x} \big\| \sup_\lambda \lambda^{\ell - s} \| \psi_\lambda \|_{L^\infty_x} \big\|_{L^2_t(I)} \Big)^\frac{1}{2} \notag \\
                &\lesa \alpha^{-\frac{1}{2}} \Big( \| \varphi \|_{L^2_t W^{s, 2^*}_x(I\times \RR^d)} \| \psi \|_{L^2_t W^{s, 2^*}_x(I\times \RR^d)} \| \varphi\|_{L^\infty_t H^s_x(I\times \RR^d)} \| \psi \|_{L^\infty_t H^s_x(I\times \RR^d)} \Big)^\frac{1}{2}. \label{eqn:prop wave est end:L2}
    \end{align}
 The bound \eqref{eqn:prop wave est end:loss} now follows from the standard energy inequality as
     \begin{align}
        \| \mc{J}_0(G) \|_{W^{\ell, 0, \ell}(I)} \les \| \mc{J}_0(\ind_I G) \|_{W^{\ell, 0, \ell}} &\lesa \Big( \sum_{\mu \in 2^\NN} \mu^{2\ell} \| \mc{J}_0(\ind_I G_\mu)\|_{L^\infty_t L^2_x}^2 + \mu^{2(\ell - 1)} \| \ind_I G_\mu \|_{L^2_{t,x}}^2\Big)^\frac{1}{2} \notag \\
         &\lesa \Big( \sum_{\mu \in 2^\NN} \mu^{2\ell} \|G_\mu\|_{L^1_t L^2_x(I\times \RR^d)}^2 + \mu^{2(\ell - 1)} \| \ind_I G_\mu \|_{L^2_{t,x}}^2\Big)^\frac{1}{2}. 
         \label{eqn:prop wave est end:energy}
     \end{align}

We now turn to the proof of \eqref{eqn:prop wave est end:gain}, this requires exploiting the fact that the high-low interactions are non-resonant.
More precisely, since $\lambda \g \alpha \gg 1$, the non-resonant identity
    $$ P^{(t)}_{\ll \lambda^2} ( \overline{\varphi}_\lambda \psi_{\frac{\lambda}{\alpha}}) = P^{(t)}_{\ll \lambda^2} ( C_{\gtrsim \lambda^2} \overline{\varphi}_\lambda \psi_{\frac{\lambda}{\alpha}}) + P^{(t)}_{\ll \lambda^2} ( C_{\ll \lambda^2} \overline{\varphi}_\lambda P^{(t)}_{\approx \lambda^2} \psi_{\frac{\lambda}{\alpha}})$$
together with
    \begin{align*}
      \Big(\sum_{\lambda \g \alpha } \lambda^{2(\ell + 1)} \big\| P^{(t)}_{\ll \lambda^2} ( C_{\gtrsim \lambda^2} \overline{\varphi}_\lambda \psi_{\frac{\lambda}{\alpha}}) \big\|_{L^1_t L^2_x}^2\Big)^\frac{1}{2}
                        &\lesa \alpha^{\ell - s} \Big(\sum_{\lambda \g \alpha } \lambda^{2(s+1)} \| C_{\gtrsim \lambda^2} \overline{\varphi}_\lambda \|_{L^2_{t,x}}^2 \Big)^\frac{1}{2} \sup_{\lambda} \lambda^{\ell -s} \| \psi_\lambda \|_{L^2_t L^\infty_x} \\
                        &\lesa \alpha^{-\frac{1}{2}} \| (i\p_t + \Delta) \overline{\varphi}\|_{L^2_t H^{s-1}_x}  \| \psi\|_{L^2_t W^{s, 2^*}_x}
    \end{align*}
and
    \begin{align*}
      \Big(\sum_{\lambda \g \alpha } \lambda^{2(\ell + 1)} \big\| P^{(t)}_{\ll \lambda^2} ( C_{\ll \lambda^2} \overline{\varphi}_\lambda P^{(t)}_{\approx \lambda^2}\psi_{\frac{\lambda}{\alpha}}) \big\|_{L^1_t L^2_x}^2\Big)^\frac{1}{2}
                &\lesa \sup_{\lambda} \lambda^s \| \varphi_\lambda \|_{L^2_t L^{2^*}_x} \Big( \sum_{\lambda \g \alpha} \lambda^{2(\ell+1-s)} \| P^{(t)}_{\approx \lambda^2} \psi_{\frac{\lambda}{\alpha}} \|_{L^2_t L^d_x}^2 \Big)^\frac{1}{2} \\
                &\lesa \alpha^{-\frac{3}{2}} \| \varphi \|_{L^2_t W^{s, 2^*}_x} \| (i\p_t + \Delta) \psi \|_{L^2_t H^{s-1}_x}
    \end{align*}
and the bounds \eqref{eqn:prop wave est end:energy} and \eqref{eqn:prop wave est end:L2} implies that 
    $$ \Big\| \sum_{\lambda \g \alpha} |\nabla| \mc{J}_0( P^{(t)}_{\ll \lambda^2} \overline{\varphi}_\lambda \psi_{\frac{\lambda}{\alpha}} ) \Big\|_{W^{\ell, 0, \ell}}
\lesa \, \alpha^{-\frac{1}{2}} \|\varphi \|_{S^{s,0,0}_w} \| \psi \|_{S^{s,0,0}_w}. $$ 
On the other hand, for the high temporal frequency term, we note that the same argument giving \eqref{eqn:lem energy ineq:Linf bound} together with Berstein's inequality gives 
    $$ \| \mc{J}_0( P^{(t)}_{\gtrsim \lambda^2} G_\lambda ) \|_{L^\infty_t L^2_x} \lesa \sum_{\nu \gtrsim \lambda^2}  \nu^{-1} \| P^{(t)}_{\nu} G_\lambda \|_{L^\infty_t L^2_x} \lesa \lambda^{-1} \| G_\lambda \|_{L^2_{t,x}}$$
and hence \eqref{eqn:prop wave est end:gain} now follows from another application of \eqref{eqn:prop wave est end:L2}. 

The final estimate is the high-high case \eqref{eqn:prop wave est end:mid}. An application of Sobolev embedding gives
    \begin{align*}
      \Big( \sum_{\lambda} \lambda^{2(\ell + 1)} \|  \ind_I \overline{\varphi}_\lambda \psi_{\approx \lambda} \|_{L^1_t L^2_x}^2 \Big)^\frac{1}{2}
                &\lesa \Big\| \Big( \sum_\lambda \lambda^{2s} |\varphi_\lambda|^2\Big)^\frac{1}{2} \Big( \sum_{\lambda} \lambda^{2(\ell + 1-s)} |\psi_\lambda|^2 \Big)^\frac{1}{2} \Big\|_{L^1_t L^2_x} \\
                &\lesa \| \varphi \|_{L^2_t W^{s, 2^*}_x(I\times \RR^d)} \| \psi \|_{L^2_t W^{s,2^*}_x(I\times \RR^d)}
    \end{align*}
and hence \eqref{eqn:prop wave est end:mid} follows from the energy estimate \eqref{eqn:prop wave est end:energy} together with the $L^2_{t,x}$ bound \eqref{eqn:prop wave est end:L2} in the special case $\alpha \approx 1$.
\end{proof}

\section{Well-posedness results}\label{sec:wp}

\subsection{Global well-posedness for the model problem}\label{subsec:gwp-model}
The first step in the proof of Theorem \ref{thm:small data gwp}, is to prove the following global result for the model problem
        $$ ( i \p_t + \Delta - \Re(V)) u = F, \qquad u(0) = f$$
where we assume that $ f\in H^s$ and $\| F \|_{N^{s,a,0}}<\infty$. In particular, this shows that the Duhamel operators $\mc{I}_V$ are well-defined as maps from $N^{s,a,0}$ to $S^{s,a,0}$, even for large wave potentials $V$.

\begin{theorem}\label{thm:global model}
Let $0 \les  s \les \ell + 2$ and $\ell \g \frac{d-4}{2}$ with $(s, \ell) \not = (\frac{d}{2}, \frac{d}{2}-2)$. Let $\beta = \max\{\frac{d-4}{2}, s-1\}$ and $a=a^*(s,\ell)$ where $a^*(s, \ell)$ is as in \eqref{eqn:choice of a,b}. There exists $\epsilon>0$ such that if $0\in I \subset \RR$ is an open interval, and
        $$f \in H^s(\RR^d), \qquad V_L = e^{it|\nabla|}g \in L^\infty_t H^\ell_x, \qquad  \|V - V_L \|_{W^{\ell, a,\beta}(I)} < \epsilon, \qquad F \in N^{s,a, 0}(I) $$
then there exists a unique solution $u \in C(I, H^s(\RR^d)) \cap L^2_t L^{2^*}_x(I\times \RR^d)$ to the Cauchy problem
        $$ (i\p_t + \Delta - \Re(V)) u =  F, \qquad u(0) = f. $$
Moreover, there exists a constant $C=C(V_L)>0$ (independent of $I$, $f$, $V$, and $F$), such that
    $$ \| u \|_{S^{s,a,0}(I)} \les C(V_L)\big( \| f\|_{H^s} + \| F \|_{N^{s,a,0}(I)} \big) $$
and, writing $I=(T_-, T_+)$ with $-\infty\les T_- < T_+ \les \infty$, there exists $f_\pm \in H^s$ such that
    $$ \lim_{t\to T_\pm} \| e^{-it\Delta} u(t) - f_\pm \|_{H^s_x} = 0. $$
\end{theorem}

\begin{remark}[Free wave potentials] The potential $V$ in Theorem \ref{thm:global model} should be thought of as a small perturbation of the free wave $V_L=e^{it|\nabla|}g$. In particular, in the special case where the potentially is simply a free wave, i.e. $V=V_L$, the  smallness condition is trivially satisfied. Consequently, for any $f \in H^s$, $g \in H^\ell$, $F\in N^{s,a,0}$, Theorem \ref{thm:global model} gives a global solution $u \in S^{s,a,0}$ to the Schr\"odinger equation
        \begin{equation}\label{eqn:linear schro free wave potential} (i\p_t + \Delta - \Re(V_L) ) u = F, \qquad u(0) = f. \end{equation}
Thus no smallness condition is required on the potential $V_L$ or the data $f$. Moreover, for any open interval $I\subset \RR$ and $g\in H^\ell$, the Duhamel integral is a continuous map $\mc{I}_{V_L} : N^{s,a,0}(I) \to S^{s,a,0}(I)$, and we have the bound
        $$ \| \mc{I}_{V_L}[F] \|_{S^{s,a,0}(I)} \lesa \| F \|_{N^{s,a, 0}(I)}. $$
\end{remark}

\begin{remark}[Strichartz control]\label{rmk:str-control}
When $a>0$, the solution space $S^{s,a,0}$ does not control the Strichartz space $L^2_t W^{s, 2^*}_x$. On the other hand, when $0\les s < \ell + 1$, we have $a^*(s,\ell)=0$. Therefore, an application of  \eqref{eqn:stri with loss} and Theorem \ref{thm:global model} implies that solutions to the Schr\"odinger equation \eqref{eqn:linear schro free wave potential} satisfy the (global) Strichartz estimate
    \begin{align*}
         \Big( \sum_{\lambda \in 2^\NN} \lambda^{2s} \| u_\lambda \|_{L^2_t L^{2^*}_x(\RR^{1+d})}^2\Big)^\frac{1}{2} \lesa \| u \|_{S^{s,0,0}} \lesa \| f \|_{H^s} + \| F \|_{N^{s,0,0}}.
    \end{align*}
In particular, for any $0\les s < \ell + 1$, $\ell \g \frac{d-4}{2}$, and $(f,g)\in H^s\times H^\ell$ we have
    $$ \| u \|_{L^2_t W^{s,2^*}_x} \lesa \| f \|_{H^s} + \big\| \big(i\p_t + \Delta - \Re(V_L)\big) u \big\|_{L^2_t W^{s,2_*}_x}. $$
\end{remark}

The first step in the proof of Theorem \ref{thm:global model} is to prove a local version with the additional assumption that the potential $V$ is small in some dispersive type norm.

\begin{proposition}\label{prop: model with small wave}
Let $0 \les  s \les \ell + 2$ and $\ell \g \frac{d}{2}-2$ with $(s, \ell) \not = (\frac{d}{2}, \frac{d}{2}-2)$. Let $\beta = \max\{\frac{d-4}{2}, s-1\}$ and define $a=a^*(s,\ell)$ as in \eqref{eqn:choice of a,b}. There exists $\epsilon>0$ and $C>0$  such that if $0\in I \subset \RR$ is an open interval, and
        $$ f \in H^s(\RR^d), \qquad F \in N^{s,a, 0}(I), \qquad V\in W^{\ell, a, \beta}(I), \qquad \|V \|_{W^{\ell, a,\beta}(I) + L^2_t W^{s,d}_x(I\times \RR^d)} < \epsilon,$$
then the Cauchy problem
        $$ ( i\p_t + \Delta) u = \Re(V) u + F, \qquad u(0) =f $$
has a unique solution $u \in C(I, H^s(\RR^d))\cap L^2_t L^{2^*}_x(I\times \RR^d)$ and we have the bound
        $$ \| u \|_{S^{s,a,0}(I)} \les C \big( \| f\|_{H^s}  + \|F \|_{N^{s,a,0}(I)}\big). $$
Moreover, writing $I=(T_-, T_+)$ with $-\infty\les T_- < T_+ \les \infty$, there exists $f_\pm \in H^s_x$ such that
        $$ \lim_{t\to T_\pm} \| e^{-it\Delta} u(t) - f_\pm \|_{H^s_x} = 0. $$
\end{proposition}
\begin{proof}
This is a direct application of Lemma \ref{lem:energy ineq}, Lemma \ref{lem:scattering}, and Corollary \ref{cor:local schro bi}. Define the sequence $u_j \in S^{s,a,0}(I)$ for $j\g 1$ by solving
        $$ (i\p_t + \Delta)u_j = \Re(V) u_{j-1} + F, \qquad u_j(0)=f$$
and let $u_0 = 0$. An application of Corollary \ref{cor:local schro bi} together with the smallness assumption on $V$ implies that
        $$ \| u_j \|_{S^{s, a, 0}(I)} \lesa \| f \|_{H^s} + \epsilon \| u_{j-1} \|_{S^{s,a,0}(I)} + \| F \|_{N^{s,a,0}(I)}$$
and
        $$ \| u_j - u_{j-1} \|_{S^{s,a,0}(I)} \lesa  \epsilon \| u_{j-1} - u_{j-2} \|_{S^{s,a,0}(I)}$$
Thus provided $\epsilon>0$ is sufficient small (depending only on the constant in Corollary \ref{cor:local schro bi}), the sequence $u_j$ is a Cauchy sequence and hence converges to a (unique) solution $u\in S^{s,a,0}(I)$. Uniqueness in the larger space $L^\infty_t L^2_x \cap L^2_t L^{2^*}_x$ follows by standard arguments from the Strichartz estimate
        \begin{align*}
            \| \mc{I}_0[\Re(V) u] \|_{L^\infty_t L^2_x \cap L^2_t L^{2^*}_x(I\times \RR^d)}
                    &\lesa \| \Re(V) u \|_{L^2_t L^{2_*}_x(I\times \RR^d)} \\
                    &\lesa \| V \|_{(L^\infty_t L^2_x +  L^2_t L^d_x)(I\times \RR^d)} \| u \|_{L^\infty_t L^2_x \cap L^2_t L^{2^*}_x(I\times \RR^d)}\\
                     &\lesa \| V \|_{W^{\ell, a,\beta}(I) + L^2_t W^{s,d}_x(I\times \RR^d)} \| u \|_{L^\infty_t L^2_x \cap L^2_t L^{2^*}_x(I\times \RR^d)} .
        \end{align*}
Finally, to prove the existence of the limits $\lim_{t \to T_{\pm}} e^{-it\Delta} u(t)$, it suffices to show that $e^{-it\Delta} u$ is a Cauchy sequence as $t \to T_+$. To this end, we first observe that by Corollary \ref{cor:local schro bi} we have $G = \Re(V) u + F \in N^{s,a,0}(I)$. Let $G'\in N^{s,a,0}$ be any extension of $G$ from $I$ to $\RR$. Then for any $t, t' \in I$
    \begin{align*}
        \| e^{-it\Delta} u(t) - e^{-it'\Delta} u(t') \|_{H^s} = \| e^{-it\Delta}\mc{I}_0[G](t) - e^{-it'\Delta} \mc{I}_0[G](t') \|_{H^s}
        &= \| e^{-it\Delta}\mc{I}_0[G'](t) - e^{-it'\Delta} \mc{I}_0[G'](t') \|_{H^s} \\
        &= \Big\| \int_{t'}^t e^{-is\Delta} G'(s) ds \Big\|_{H^s}
    \end{align*}
and therefore, an application of Lemma \ref{lem:energy ineq} and Lemma \ref{lem:scattering} implies that $e^{-it\Delta} u(t)$ is a Cauchy sequence as required.
\end{proof}

To apply the previous proposition, we need to decompose $\RR$ into intervals on which $V_L$ is small. This exploits the dispersive properties of the free wave $V_L=e^{it|\nabla|}g$. More precisely, we have the  following minor variation of \cite[Lemma 4.1]{Bejenaru2015}.

\begin{lemma}[{\cite[Lemma 4.1]{Bejenaru2015}}]\label{lem:interval decom}
Let $\ell, s, a\g0$, $\epsilon>0$, and $V_L=e^{it|\nabla|} g \in L^\infty_t H^\ell_x$. Then there exists a finite collection of open intervals $(I_j)_{j=1,\dots, N}$ such that $\RR = \cup_{j=1}^N I_j$, $\min |I_j \cap I_{j+1}|>0$, and
        $$ \sup_{j=1,\dots, N} \| V_L \|_{W^{\ell, a, \beta}(I_j) + L^2_t W^{s,d}_x(I_j \times \RR^d)} < \epsilon. $$
\end{lemma}
\begin{proof}
Decompose $g=g_1 + g_2$ where $g_2 \in C^\infty_0(\RR^d)$ and $\| g_1 \|_{H^\ell} < \epsilon$. Since $g_2$ is smooth and compactly supported, the dispersive estimate for the free wave equation gives $e^{it|\nabla|} g_2 \in L^2_t W^{s,d}_x(\RR^{1+d})$ and hence we can find a collection of open intervals $(I_j)_{j=1,\dots, N}$ such that $\RR = \cup_{j=1}^N I_j$, $\min |I_j \cap I_{j+1}|>0$, and
        $$ \sup_{j=1,\dots, N} \| e^{it|\nabla|} g_2 \|_{L^2_t W^{s,d}_x(I_j \times \RR^d)} < \epsilon. $$
On the other hand, the definition of the norm $W^{\ell,a, \beta}$ implies that
        $$ \| e^{it|\nabla|} g_1 \|_{W^{\ell,a, \beta}(I_j)} \les \| e^{it|\nabla|} g_1 \|_{W^{\ell,a, \beta}} \lesa \| g_1 \|_{H^\ell} \lesa \epsilon. $$
Therefore, for every $j=1, \dots, N$, we have
        $$ \| V_L \|_{W^{\ell, a, \beta}(I_j) + L^2_t W^{s,d}_x(I_j \times \RR^d)} \les \| e^{it|\nabla|} g_1 \|_{W^{\ell, a, \beta}(I_j)} + \| e^{it|\nabla|} g_2 \|_{L^2_t W^{s,d}_x(I_j \times \RR^d)} \lesa \epsilon. $$
\end{proof}

The proof of Theorem \ref{thm:global model} now follows by repeatedly applying Proposition \ref{prop: model with small wave} together with the decomposability property in Lemma \ref{lem:decomposability}.

\begin{proof}[Proof of Theorem \ref{thm:global model}]
Let $\epsilon>0$ and suppose that
        $$ \| V - V_L \|_{W^{\ell, a, \beta}} < \epsilon. $$
An application of Lemma \ref{lem:interval decom} gives finite number of open intervals $(I_j)_{j=1, \dots, N}$ and points $t_j \in I_{j-1} \cap I_j$ such that $ I = \cup_{j=1}^N I_j$, $\min |I_j \cap I_{j+1} |>0$, and
       $$ \sup_{j=1,\dots, N} \| V \|_{W^{\ell, a, \beta}(I_j) + L^2_t W^{s,d}_x(I_j \times \RR^d)} \les \| V - V_L \|_{W^{\ell, a, \beta}(I)} +  \sup_{j=1,\dots, N} \| V_L \|_{W^{\ell, a, \beta}(I_j) + L^2_t W^{s,d}_x(I_j \times \RR^d)} < 2\epsilon. $$
Assuming $\epsilon>0$ is sufficiently small, Proposition \ref{prop: model with small wave} gives a (unique) solution $u \in C(I_j, H^s)\cap L^2_t L^{2^*}_x(I_j\times \RR^d)$ on the interval $0\in I_j$ to the Cauchy problem
    \begin{equation}\label{eqn:thm gwp model:extended eqn}
        (i\p_t + \Delta) u= \Re(V) u +F, \qquad u(0) =f
    \end{equation}
such that
    $$ \| u \|_{S^{s,a,0}(I_j)} \lesa \| f \|_{H^s} +  \| F \|_{N^{s,a,0}(I_j)} \lesa \| f \|_{H^s} +  \| F \|_{N^{s,a,0}(I)}. $$
Taking new data $u(t_j)$ and $u(t_{j-1})$, and again applying Proposition \ref{prop: model with small wave}, we get a unique solution
    $$u \in C( I_{j-1} \cup I_j \cup I_{j+1}, H^s)\cap L^2_t L^{2^*}_x(I_{j-1} \cup I_j \cup I_{j+1}\times \RR^d)$$
with
    $$ \sup_{k=j-1, j, j+1} \| u \|_{S^{s,a,0}(I_k)} \lesa \| f \|_{H^s} + \| F \|_{N^{s,a,0}(I)}. $$
Continuing in this manner, after at most $N$ steps, we obtain a unique solution $ u \in C(I, H^s) \cap L^2_t L^{2^*}_x(I \times \RR^d)$ such that
    $$ \|u \|_{S^{s,a,0}(I)} \lesa_N \sup_{j=1, \dots, N} \| u \|_{S^{s,a,0}(I_j)} \lesa_N \| f\|_{H^s} + \| F \|_{N^{s,a,0}(I)} $$
where the first inequality is a consequence of Lemma \ref{lem:decomposability}. Finally, to show that the claimed limits as $t \to \sup I$ and $t \to \inf I$ exist, we simply repeat the argument at the end of the proof of Proposition \ref{prop: model with small wave}.
\end{proof}

\subsection{Local and small data global results for the Zakharov system}\label{subsec:local-smallglobal}

We first consider the non-endpoint case $s>\frac{d-3}{2}$.

\begin{theorem}[LWP and small data GWP: non-endpoint case]\label{thm:gwp nonendpoint}
  Let $d\g 4$ and suppose that $(s,\ell)$ satisfies the conditions \eqref{eqn:cond on s l} and $s>\frac{d-3}{2}$. Let $a=a^*(s,\ell)$ and $b=b^*(s, \ell)$ as in \eqref{eqn:choice of a,b}. For some $0<\theta<1$ and any $g_* \in H^\ell(\R^d)$ there exists $\epsilon>0$, such that if $f_*\in H^s(\R^d)$ satisfies
  \begin{equation}\label{eq:cond-fg}
        \| f_* \|_{H^s}^{1-\theta} \| e^{it\Delta}f_* \|_{L^2_t L^{2^*}_x(I\times \RR^d)}^{\theta} < \epsilon,  \text{ for an interval }0\in I \subset \RR,
      \end{equation}
      then for all $(f,g)$ in \[ D_\epsilon(f_*,g_*):=\big\{(f,g)\in H^s\times H^\ell:\;
\|f-f_*\|_{H^s}<\epsilon, \quad \|g-g_*\|_{H^\ell}<\epsilon\big\},
      \]
      there exists a unique solution $(u, V) \in S^{s, a, b}(I) \times W^{\ell, a, s -\frac{1}{2}}(I)$ to \eqref{eq:Zakharov 1st order}.
The flow map $$ H^s(\RR^d)\times H^\ell(\RR^d) \supset D \ni (f,g)\mapsto (u,V)\in S^{s,a, b}(I)\times W^{\ell, a, s-\frac{1}{2}}(I)$$
is real-analytic, where $D=D_\epsilon(f_*,g_*)$ is the open bi-disc defined above. Moreover, if $I= \RR$, then there exists $(f_\pm, g_\pm) \in H^s \times H^\ell$ such that
$$ \lim_{t\to \pm \infty} \Big( \big\| u(t) - e^{it\Delta} f_\pm \big\|_{H^s} + \| V(t) - e^{it|\nabla|} g_\pm \|_{H^\ell} \Big) = 0. $$
\end{theorem}

\begin{proof}
  Fix $(s, \ell)$ satisfying the conditions \eqref{eqn:cond on s l} and $s>\frac{d-3}{2}$, and define $a= a^*(s, \ell)$ and $b=b^*(s, \ell)$ as in \eqref{eqn:choice of a,b}. Let $\tilde{\ell} = \min\{ \ell, s-\frac{1}{2}\}$ and define $V_L = e^{it|\nabla| } g_*$ to the free wave evolution of $g_* \in H^\ell$, and similarly $u_L=e^{it \Delta}f_* $ for $f_*\in H^s$ in case of the free Schr\"odinger evolution.

  Let us recall that $\mc{I}_{V_L}[F]$ denotes the solution to the inhomogeneous Schr\"odinger equation
        $$ \big(i\p_t + \Delta - \Re(V_L)\big) \psi = F, \qquad \psi(0) = 0 $$
and similarly, $\mc{J}_0[G]$ denotes the solution to the inhomogeneous wave equation
        $$ (i\p_t + |\nabla|)\phi = G, \qquad \phi(0) = 0. $$
We claim there exists $0< \theta < 1$ such that
        \begin{align}
          \big\| \mc{I}_{V_L}[ \Re(\phi) \psi ] \big\|_{S^{s,a,0}(I)} &\lesa_{g_*}  \| \phi \|_{W^{\tilde{\ell}, a, s-\frac{1}{2}}(I)} \| \psi \|_{S^{s,a,0}(I)}
           \label{eqn:thm gwp nonend:schro est}\\
          \big\| \mc{I}_{V_L}[ \Re(\phi) \psi] \big\|_{L^2_t L^{2^*}_x(I\times \RR^d)} &\lesa_{g_*} \| \phi \|_{W^{\tilde{\ell}, 0, 0}(I)} \| \psi \|_{L^2_t L^{2^*}_x(I\times \RR^d)}
          \label{eqn:thm gwp nonend:schro est L2}\\
         \big\| \mc{J}_0\big[ |\nabla|(\overline{\psi} \varphi)\big] \big\|_{W^{\tilde{\ell}, a,  s-\frac{1}{2}}(I)}
          &\lesa \Big( \| \psi \|_{L^2_t L^{2^*}_x(I\times \RR^d)} \| \varphi \|_{L^2_t L^{2^*}_x(I\times \RR^d)}\Big)^\theta \Big( \| \psi \|_{S^{s, a, 0}(I)} \| \varphi \|_{S^{s, a,0}(I)}\Big)^{1-\theta}.
          \label{eqn:thm gwp nonend:wave est}
         \end{align}
The estimate \eqref{eqn:thm gwp nonend:schro est} follows from Theorem \ref{thm:global model} and Theorem \ref{thm:main bilinear schro}. To prove \eqref{eqn:thm gwp nonend:schro est L2}, we apply Remark \ref{rmk:str-control} and observe that via the Littlewood-Paley square function estimate and Bernstein's inequality
    \begin{align*}
      \big\| \mc{I}_{V_L}[F] \big\|_{L^2_t L^{2^*}_x(I\times \RR^d)} &\lesa \Big( \sum_{\lambda \in 2^\NN} \big\| P_\lambda \mc{I}_{V_L}[F] \big\|_{L^2_t L^{2^*}_x(I\times \RR^d)}^2 \Big)^\frac{1}{2} \\
      &\lesa_{g_*} \| F \|_{N^{0,0,0}(I)} \lesa_{g_*} \| F \|_{L^2_t L^{2_*}_x(I\times \RR^d)},
    \end{align*}
 see also \eqref{eqn:N norm chara}.
Therefore
\begin{align*}\big\| \mc{I}_{V_L}[\Re( \phi) \psi] \big\|_{L^2_t L^{2^*}_x(I\times \RR^d)} &\lesa_{g_*} \| \phi \psi \|_{L^2_t L^{2_*}_x} \lesa_{g_*} \| \phi \|_{L^\infty_t L^\frac{d}{2}_x(I\times \RR^d)} \| \psi \|_{L^2_t L^{2^*}_x(I\times \RR^d)} \\
  &\lesa_{g_*} \| \phi \|_{W^{\tilde{\ell}, 0, 0}(I)} \| \psi \|_{L^2_t L^{2^*}_x(I\times \RR^d)}
\end{align*}
and so \eqref{eqn:thm gwp nonend:schro est L2} follows. The final estimate \eqref{eqn:thm gwp nonend:wave est} is a direct application of Corollary \ref{cor:loc bi wave}.

Set $\rho=V-V_L$ and $g^*=g-g_*$. Then the pair $(u, \rho)$ solves
    \begin{align*}
        (i\p_t + \Delta - \Re(V_L)) u&= \Re(\rho)u, \qquad u(0) =f, \\
       (i\p_t +|\nabla|) \rho &= - |\nabla| |u|^2, \qquad \rho(0) = g^*.
    \end{align*}
After noting that $\rho=e^{it|\nabla|}g^*-\mc{J}_0(|\nabla| |u|^2)$ it suffices to find a fixed point $u\in S^{s, a, 0}(I)$ for the map
  \begin{equation*}
    \Phi(f,g; u):=e^{it\Delta}f + \mc{I}_{V_L}\Big( \Re(V_L) e^{it\Delta} f\Big) +\mc{I}_{V_L}\Big(\Re(e^{it|\nabla|}g^*)u\Big)- \mc{I}_{V_L}\Big(\mc{J}_{0}(|\nabla||u|^2) u\Big).
  \end{equation*}
Let $\Lambda = \| e^{it\Delta} f_*\|_{L^2_t L^{2^*}_x(I)} \| f_*\|_{H^s}^{-1}$ ($\Lambda=1$ if $f_*=0$) and define the norm
        $$ \| u \|_{Z} = \inf_{u = u_1 + u_2} \Big( \Lambda^\theta \| u_1 \|_{S^{s,a,0}(I)} + \Lambda^{\theta-1} \| u_1 \|_{L^2_t L^{2^*}_x (I)} + \| u_2 \|_{S^{s, a, 0}(I)} \Big).  $$
Note that $(1+\Lambda^{-\theta})^{-1} \| u \|_{S^{s,a,0}(I)} \les \| u \|_{Z} \les \| u \|_{S^{s,a,0}(I)}$ and hence $\| \cdot \|_{Z}$ is an equivalent norm on $S^{s,a,0}(I)$. Moreover, since
$ a^{1-\theta} b^{\theta} \les \Lambda^{\theta} a + \Lambda^{\theta-1} b$ and $\| u \|_{L^2_t L^{2^*}_x(I)} \lesa \| u \|_{S^{s,a,0}}$, a short computation using the bounds \eqref{eqn:thm gwp nonend:schro est}, \eqref{eqn:thm gwp nonend:schro est L2}, and \eqref{eqn:thm gwp nonend:wave est} implies that
        $$ \| \mc{I}_{V_L}[\Re(\phi) u ] \|_{Z} \lesa_{g_*} \| \phi \|_{W^{\tilde{\ell}, a,  s-\frac{1}{2}}(I)} \| u \|_{Z}, \qquad \| \mc{J}_0[|\nabla| (\overline{u} v)] \|_{W^{\tilde{\ell}, a, s-\frac{1}{2}}(I)} \lesa \| u \|_{Z} \| v \|_{Z}. $$
Moreover, in view of the endpoint Strichartz estimate and the definition of $\Lambda$ we have
       \begin{align*}
            \| e^{it\Delta} f\|_{Z} &\lesa \Lambda^{\theta} \| f_*\|_{H^s} + \Lambda^{\theta -1} \| e^{it\Delta } f_* \|_{L^2_t L^{2^*}_x(I)} + \| f- f_*\|_{H^s} \\
            &= 2\| f_*\|_{H^s}^{1-\theta} \| e^{it\Delta } f_* \|_{L^2_t L^{2^*}_x(I)}^\theta + \| f- f_*\|_{H^s} \lesa \epsilon + \| f- f_*\|_{H^s}.
       \end{align*}
Consequently, for any $(f, g) \in D$ and $u, v\in S^{s, a, 0}(I)$ we see that
        $$ \| \Phi(f,g; u) \|_{Z} \lesa_{g_*} \epsilon + \epsilon \| u \|_{Z} + \| u \|_{Z}^3$$
and
        $$ \| \Phi(f,g; u) -  \Phi(f,g; v)  \|_{Z} \lesa_{g_*} \epsilon \| u - v\|_{Z} + \big( \| u \|_{Z} + \| v \|_{Z}\big)^2 \| u - v\|_{Z}. $$
Let $C=C(g_*)$ denote the largest of the above implicit constants and take $K=\{ u\in S^{s,a,0} \mid \|u\|_{Z} \les 2 C \epsilon \} $. Then provided $\epsilon = \epsilon(g_*)>0$ is chosen sufficiently small, we get unique fixed point $u\in K \subset S^{s,a,0}$.

In addition, as a consequence of the above estimates, for $(f,g)\in D$ and $u \in K$, we have that for any $v \in S^{s,a,0}(I)$, the linear map
$Tv = v - D_v \Phi(f,g;u)$ is a small perturbation of the identity (with respect to the norm $\|\cdot \|_{Z}$), and hence $T$ is a linear homeomorphism onto $S^{s,a,0}(I)$. Furthermore, the map $\Phi$ is real-analytic (as a composition of linear, bi- and trilinear maps over $\R$). If $u[f,g]$ denotes the solution with initial data $(f,g)$, the implicit function theorem \cite[Theorem~15.3]{Deimling85} implies that the flow map $D\ni (f,g)\mapsto u[f,g]\in S^{s,a,0}(I)$ is real-analytic.
Define $V=e^{it|\nabla|}g-\mc{J}_0(|\nabla| |u|^2)$. Estimate \eqref{eqn:thm gwp nonend:wave est} implies that $V\in W^{\tilde{\ell}, a, s-\frac{1}{2}}(I)$ and $(u,V)$ is a solution of \eqref{eq:Zakharov 1st order}. Also,
$D\ni (f,g)\mapsto V[f,g]=e^{it|\nabla|}g-\mc{J}_0(|\nabla| |u[f,g]|^2)\in  W^{\tilde{\ell}, a, s-\frac{1}{2}}(I) $ is a composition of real-analytic maps and therefore real-analytic. In the case $s\geq \ell+\frac12$ we have $\ell=\tilde\ell$ and $b=0$, so this is the claim.

In the remaining case $s<\ell+\frac12$, we have $a=0$. Define $\kappa=\ell$ if $s>\ell$ and $\kappa=s-\frac{1}{2}(1-b)$ if $s\leq \ell$. An application of Theorem \ref{thm:main bilinear schro} gives
        $$ \| \mc{I}_0[ \Re(\phi) \psi ] \|_{S^{s,0,b}(I)} \lesa \| \phi \|_{W^{\kappa, 0, s-\frac{1}{2}}(I)} \| \psi \|_{S^{s,0,0}(I)}$$
while Theorem \ref{thm:main bilinear wave}  implies that
$$ \| \mc{J}_0[|\nabla|(\overline{\psi} \varphi)]\|_{W^{\kappa,0,s-\frac{1}{2}}(I)} \lesa \| \psi \|_{S^{s,0,0}(I)} \| \varphi \|_{S^{s,0,0}(I)}. $$
For $(f,g )\in D$ and the solution $u \in K$ we have
\begin{equation}\label{eq:alt-fp}
u=e^{it\Delta}f+\mc{I}_0(\Re(e^{it|\nabla|}g)u)-\mc{I}_0(\mc{J}_{0}(|\nabla||u|^2)u).
\end{equation}
Thus, we  conclude that
\begin{align*}
  \|u \|_{S^{s,0,b}(I)} &\lesa \|f\|_{H^s}+\|g\|_{H^\ell}\|u\|_{S^{s,0,0}(I)}+\|u\|_{S^{s,0,0}(I)}^3\\
  &\lesa (1 + \| g \|_{H^\ell} +  \|f \|_{H^s}^2) \| f\|_{H^s}.
  \end{align*}
  Equation \eqref{eq:alt-fp} also shows that $D\ni (f,g)\mapsto u[f,g]\in S^{s,0,b}(I)$ is a composition of real-analytic maps, hence real-analytic. Theorem \ref{thm:main bilinear wave} again implies that
  $$ \| \mc{J}_0[|\nabla|(\overline{\psi} \varphi)]\|_{W^{\ell,0,s-\frac{1}{2}}(I)} \lesa \| \psi \|_{S^{s,0,b}(I)} \| \varphi \|_{S^{s,0,b}(I)}. $$
  We conclude
  $$ \| V \|_{W^{\ell, 0, s-\frac{1}{2}}(I)} \lesa \|g\|_{H^\ell}+\big(1 + \|g\|_{H^\ell} + \| f \|_{H^s}^2 \big)^2 \| f\|_{H^s}^2 $$
  and, as above, $D\ni (f,g)\mapsto V[f,g]\in W^{\ell, 0, s-\frac{1}{2}}(I)$ is real-analytic.

 Finally, we remark that if $I=\RR$, then the solution scatters. This follows from an analogous argument to that used in the proof of Theorem \ref{thm:global model} (i.e. show that $u(t)$ forms a Cauchy sequence as $t\to \infty$). It only remains to prove uniqueness in $S^{s,a, b}(I) \times W^{\ell, a, s-\frac{1}{2}}(I)$, but this again a consequence of the estimates proved above.
\end{proof}

We now consider the endpoint case $(s,\ell) = (\frac{d-3}{2}, \frac{d-4}{2})$.

\begin{theorem}[LWP and small data GWP: endpoint case]\label{thm:gwp endpoint}
Let $d\g 4$ and fix $(s,\ell)=(\frac{d-3}{2}, \frac{d-4}{2})$. For any $g_* \in H^\ell$ there exists $\epsilon>0$, such that if $f_* \in H^s$ and $0\in I \subset \RR$ is an interval with
        \begin{equation}\label{eqn:thm gwp endpoint:smallness}
             \| e^{it\Delta}f_* \|_{L^2_t W^{s, 2^*}_x(I\times \RR^d)}^\frac{1}{2} \| f_* \|_{H^s}^\frac{1}{2} < \epsilon,
        \end{equation}
        then for all $(f,g)$ in 
        \[
 D_\epsilon(f_*,g_*):=\{(f,g)\in H^s\times H^\ell: \|f-f_*\|_{H^s}<\epsilon, \quad \|g-g_*\|_{H^\ell}<\epsilon\},
        \]
        there exists a unique solution $u \in C(I, H^s)\cap L^2_t W^{s, 2^*}_x(I\times \RR^d)$, $V\in C(I, H^\ell)$ to \eqref{eq:Zakharov 1st order}. Moreover, $(u,V) \in S^{s,0,0}(I)\times W^{\ell,0,s-\frac12}(I)$ and the flow map $$ H^s(\RR^d)\times H^\ell(\RR^d) \supset D \ni (f,g)\mapsto (u,V)\in S^{s,0, 0}(I)\times W^{\ell, 0, s-\frac{1}{2}}(I)$$
is real-analytic, where $D=D_\epsilon(f_*,g_*)$ is the open bi-disc defined above. If $I= \RR$, then there exists $(f_\pm, g_\pm) \in H^s \times H^\ell$ such that
        $$ \lim_{t\to \pm \infty} \Big( \big\| u(t) - e^{it\Delta} f_\pm \big\|_{H^s} + \| V(t) - e^{it|\nabla|} g_\pm \|_{H^\ell} \Big) = 0. $$
\end{theorem}
\begin{proof}
Let $g_* \in H^\ell$ and $\epsilon > 0$ to be fixed later depending only on $g_*$ and the implicit constants in Theorem \ref{thm:global model} and Propositions \ref{prop:schro est for endpoint} and \ref{prop:wave est end}. Let $f_* \in H^s$ and $0\in I \subset \RR$ satisfy the smallness condition \eqref{eqn:thm gwp endpoint:smallness}. As in the proof of Theorem \ref{thm:gwp nonendpoint}, we let $V_L = e^{it|\nabla|} g_*$ and $g^*=g-g_*$ and aim to find a fixed point $ u = \Phi(f, g; u)$ where
    \begin{equation*}
        \Phi(f,g; u):=e^{it\Delta}f+ \mc{I}_{V_L}\Big( \Re(V_L) e^{it\Delta} f\Big)+\mc{I}_{V_L}\Big(\Re(e^{it|\nabla|}g^*)u\Big)- \mc{I}_{V_L}\Big(\mc{J}_{0}(|\nabla||u|^2) u\Big).
    \end{equation*}
Let $\Lambda = \| e^{it\Delta} f_*\|_{L^2_t W^{s, 2^*}_x(I)} \|f_*\|_{H^s}^{-1}$ (with $\Lambda =1$ if $f_*=0$) and define the norm
        $$ \|u \|_Z = \inf_{u=u_1 + u_2} \Big( \Lambda^\frac{1}{2} \|u_1\|_{S^{s,0,0}(I)} + \Lambda^{-\frac{1}{2}} \| u_1\|_{L^2_t W^{s, 2^*}_x(I)} + \| u_2\|_{S^{s,0,0}(I)}\Big).$$
Note that $\Lambda^{\frac{1}{2}} \| u \|_{S^{s,0,0}(I)} \lesa \| u \|_{Z} \les \| u \|_{S^{s,0,0}(I)}$ and thus $\|\cdot \|_{Z}$ forms an equivalent norm on $S^{s,0,0}$. An application of Theorem \ref{thm:global model} implies that
    $$\| \mc{I}_{V_L}[ F] \|_{S^{s,0,0}(I)} \lesa_{g_*} \| F \|_{N^{s,0,0}(I)}$$
and therefore Proposition \ref{prop:schro est for endpoint} and Proposition \ref{prop:wave est end} give
     \begin{align*}
     \| \mc{I}_{V_L}[ \Re(\phi) u ] \|_{S^{s,0,0}(I)} &\lesa_{g^*} \| \phi \|_{W^{\ell, 0, s-\frac{1}{2}}(I)} \| u \|_{L^2_tW^{s, 2^*}_x(I)}^\frac{1}{2} \| u \|_{S^{s,0,0}(I)}^{\frac{1}{2}}, \\
     \| \mc{J}_0[|\nabla|(\overline{u} v) ] \|_{W^{\ell, 0, s-\frac{1}{2}}(I)} &\lesa \| u \|_{L^2_tW^{s, 2^*}_x(I)}^\frac{1}{2} \| u \|_{S^{s,0,0}(I)}^{\frac{1}{2}}\| v \|_{L^2_tW^{s, 2^*}_x(I)}^\frac{1}{2} \| v \|_{S^{s,0,0}(I)}^{\frac{1}{2}}.
     \end{align*}
Since $\| u \|_{L^2_tW^{s, 2^*}_x(I)} \lesa \| u \|_{S^{s,0,0}(I)}$ and $(ab)^\frac{1}{2} \les \Lambda^{\frac{1}{2}} a + \Lambda^{-\frac{1}{2}} b$, we conclude that
    $$ \| \mc{I}_{V_L}[ \Re(\phi) u ] \|_{Z}\les \| \mc{I}_{V_L}[ \Re(\phi) u ] \|_{S^{s,0,0}(I)} \lesa \| \phi \|_{W^{\ell, 0, s-\frac{1}{2}}(I)} \| u \|_{Z}$$
and
    $$  \| \mc{J}_0[|\nabla|(\overline{u} v) ] \|_{W^{\ell, 0, s-\frac{1}{2}}(I)} \lesa \|u\|_Z \| v \|_{Z}. $$
Moreover, in view of our choice of $\Lambda>0$ and the smallness condition \eqref{eqn:thm gwp endpoint:smallness} we see that
    \begin{align*}
        \| e^{it\Delta} f\|_Z &\lesa \Lambda^{\frac{1}{2}} \| f_* \|_{H^s} + \Lambda^{-\frac{1}{2}} \| e^{it\Delta} f_*\|_{L^2_tW^{s, 2^*}_x(I)} + \| f - f_*\|_{H^s} \\
                              &\lesa \| f_*\|_{H^s}^{\frac{1}{2}} \| e^{it\Delta} f_* \|_{L^2_tW^{s, 2^*}_x(I)}^\frac{1}{2}  + \| f - f_*\|_{H^s} \lesa \epsilon +  \| f - f_*\|_{H^s}.
    \end{align*}
Combining the above bounds then gives
    $$ \| \Phi(f, g;u)\|_{Z} \lesa_{g_*} \epsilon +  \| f - f_*\|_{H^s} + \| g-g_*\|_{H^\ell} \| u \|_Z + \|u\|_{Z}^3$$
and
    $$ \| \Phi(f, g;u)-\Phi(f, g;v)\|_{Z} \lesa_{g_*}  \| g-g_*\|_{H^\ell} \| u-v \|_Z + (\|u\|_{Z}+\|v\|_Z)^2 \| u - v\|_Z.$$
A routine contraction argument then implies that, provided $\epsilon = \epsilon(g_*)>0$ is sufficiently small, for any $(f,g) \in D$ there is a unique fixed point  $u\in S^{s,0,0}(I)$. Setting $V = V_L - \mc{J}_0[|\nabla| |u|^2]$, we get a solution $(u, V) \in S^{s,0,0}(I) \times W^{\ell, 0, s-\frac{1}{2}}(I)$ due to Proposition \ref{prop:wave est end}. Also, the flow map is real-analytic, we omit the details.

To prove that the solution scatters, we note that writing $I=(T_0, T_1)$, then as in the proof of Theorem \ref{thm:global model}, a computation shows that for any sequence of times $t_j \nearrow T_1$, the sequence $ (e^{- i t_j \Delta} u(t_j), e^{-it_j|\nabla|} V(t_j))$ forms a Cauchy sequence in $H^s \times H^\ell$. In particular the limits
    $$\lim_{t\nearrow T_1} \big(e^{- i t \Delta} u(t), e^{-it|\nabla|} V(t)\big) \qquad \text{and} \qquad \lim_{t \searrow T_0} \big(e^{- i t \Delta} u(t), e^{-it|\nabla|} V(t)\big)$$
exist in $H^s \times H^\ell$. Therefore, if $I=\RR$, the solution scatters to free solutions as $t \to \pm \infty$.

To check the uniqueness claim, we note that the above bounds together with a continuity argument give uniqueness in $S^{s,0,0}(I) \times W^{\ell, 0, \ell}(I)$. In particular, in view of Propositions \ref{prop:schro est for endpoint} and \ref{prop:wave est end}, to prove uniqueness it suffices to show that if $ (u, V) \in C(I, H^s\times H^\ell)$ is a solution to \eqref{eq:Zakharov 1st order} with 
		$$ \| u\|_{L^\infty_t H^s_x(I\times \RR^d)} + \| u \|_{L^2_t W^{s,2^*}_x(I\times \RR^d)} + \| V \|_{L^\infty_t H^\ell(I\times \RR^d)} < \infty $$
then $u \in S^{s,0,0}_w(I)$ and $V \in W^{\ell, 0, \ell}_w(I)$ where the `weak' solution spaces $S^{s,0, 0}_w$ and $W^{\ell, 0, \ell}_w$ are defined in the beginning of Section \ref{sec:local-bil}. To this end, we note that a standard product estimate gives
    $$ \| (i\p_t + \Delta) u \|_{L^2_t H^{s-1}_x(I\times \RR^d)}  =  \| \Re(V) u \|_{L^2_t H^{s-1}_x(I\times \RR^d)} \lesa \| V \|_{L^\infty_t H^\ell(I\times \RR^d)} \| u \|_{L^2_t W^{s, 2^*}_x(I\times \RR^d)}$$
and 
    $$ \| (i\p_t + |\nabla|) V \|_{L^2_t H^{\ell -1}_x(I\times \RR^d)} \lesa \| |\nabla| |u|^2 \|_{L^2_t H^{s-\frac{1}{2}}_x(I\times \RR^d)} \lesa \| u\|_{L^\infty_t H^s_x(I\times \RR^d)} \| u \|_{L^2_t W^{s, 2^*}_x(I\times \RR^d)}. $$
Consequently, extending $u$ from the interval $I = (T_0, T_1)$ to $\RR$ using free Schr\"odinger waves,
    $$u' = \ind_{(-\infty, T_0)}(t) e^{i(t-T_0)\Delta} u(T_0) + \ind_I(t) u(t) + \ind_{(T_1, \infty)}(t) e^{i(t-T_1)\Delta} u(T_1)$$
(potentially shrinking $I$ ensures $u(T_0), u(T_1)\in H^s$ are well-defined) the endpoint Strichartz estimate implies
\begin{align*}
	\| u'\|_{S^{s,0,0}_w(I)} &\les \| u' \|_{L^\infty_t H^s_x} + \| u' \|_{L^2_t W^{s,2^*}_x} + \| (i\p_t + \Delta) u' \|_{L^2_t H^{s-1}_x} \\
					&\lesa	\| u \|_{L^\infty_t H^s_x(I\times \RR^d)} + \| u \|_{L^2_t W^{s,2^*}_x (I\times \RR^d)} + \| (i\p_t + \Delta) u \|_{L^2_t H^{s-1}_x (I\times \RR^d)} 
\end{align*}
and hence $u\in S^{s,0,0}_w(I)$. Similarly, we extend $V$ from $I$ by free waves 
   $$V' = \ind_{(-\infty, T_0)}(t) e^{i(t-T_0)|\nabla|} V(T_0) + \ind_I(t) V(t) + \ind_{(T_1, \infty)}(t) e^{i(t-T_1)|\nabla|} V(T_1)$$
and observe that 
\begin{align*}
	\| V \|_{W^{\ell, 0, \ell}_w(I)} &\les \| V'\|_{L^\infty_t H^\ell_x} + \| (i\p_t + |\nabla|)V' \|_{L^2_t H^{\ell-1}_x} \\
	&\lesa 	\| V\|_{L^\infty_t H^\ell_x(I\times \RR^d)} + \| (i\p_t + |\nabla|)V \|_{L^2_t H^{\ell-1}_x (I\times \RR^d)}.
\end{align*}
Therefore $V \in W^{\ell, 0, \ell}_w(I)$.

\end{proof}

\section{Persistence of Regularity}\label{sec:pers}
In this section our goal is show that under suitable assumptions on a solution $(u, V)$ to \eqref{eq:Zakharov 1st order}, any additional regularity of the data $(u, V)(0)$ persists in time.

\begin{theorem}\label{thm:persistence}
Let $(s, \ell)$ satisfy \eqref{eqn:cond on s l} and fix $a=a^*(s,\ell)$, $b=b^*(s,\ell)$ as in \eqref{eqn:choice of a,b}. Suppose that $(u, V)$ is a solution to the Zakharov system \eqref{eq:Zakharov 1st order} on some interval $I\ni 0$ with
    $$ \| u\|_{L^\infty_t H^{\frac{d-3}{2}}_x(I\times \RR^d)} + \| u \|_{L^2_t W^{\frac{d-3}{2}, 2^*}_x(I\times \RR^d)} + \| V \|_{L^\infty_t H^{\frac{d-4}{2}}_x(I\times \RR^d)} < \infty.$$
    If $(u, V)(0) \in H^s\times H^\ell$, then $(u, V) \in S^{s,a,b}(I)\times W^{\ell, a, s-\frac{1}{2}}(I)$, and the flow map
    is real-analytic with respect to the $H^s\times H^\ell$ and $ S^{s,a,b}(I)\times W^{\ell, a, s-\frac{1}{2}}(I)$ topologies.
\end{theorem}

We break the proof of Theorem \ref{thm:persistence} into three main steps.
\begin{enumerate}
    \item (Improving Schr\"odinger regularity when $s\g \ell + \frac{1}{2}$.) If $(s, \ell)$ and $(\tilde{s}, \ell)$ satisfy \eqref{eqn:cond on s l} and $\ell +\frac{1}{2} \les s < \tilde{s}$, then
     $$
        (u, V) \in S^{s,a,0}(I)\times W^{\ell, a, s-\frac{1}{2}}(I) \text{ and } u(0)\in H^{\tilde{s}} \qquad \Longrightarrow \qquad (u, V) \in S^{\tilde{s}, \tilde{a}, 0}(I)\times W^{\ell, \tilde{a}, s-\frac{1}{2}}(I)
     $$
     where $a=a^*(s, \ell)$ and $\tilde{a}=a^*(\tilde{s}, \ell)$. \\

   \item (Improving wave regularity when $s \g \ell + \frac{1}{2}$.) If $(s, \ell)$ and $(s, \tilde{\ell})$ satisfy \eqref{eqn:cond on s l} and $\ell < \tilde{\ell}  \les s - \frac{1}{2}$, then
     $$
     (u, V) \in S^{s,a,0}(I)\times W^{\ell, a, s-\frac{1}{2}}(I) \text{ and } V(0)\in H^{\tilde{\ell}} \qquad \Longrightarrow \qquad (u, V) \in S^{s, \tilde{a}, 0}(I)\times W^{\tilde{\ell}, \tilde{a}, s-\frac{1}{2}}(I)
     $$
where now $a=a^*(s, \ell)$ and $\tilde{a}=a^*(s,\tilde{\ell})$. \\

    \item (Improving wave regularity when $\ell > s -1$.) If $(s, \ell)$ and $(s, \tilde{\ell})$ satisfy \eqref{eqn:cond on s l} and $s-1 < \ell < \tilde{\ell}$, then
     $$
        (u, V) \in S^{s,0,b}(I)\times W^{\ell, 0, s-\frac{1}{2}}(I) \text{ and } V(0)\in H^{\tilde{\ell}} \qquad \Longrightarrow \qquad (u, V) \in S^{s, 0, \tilde{b}}(I)\times W^{\tilde{\ell}, 0, s-\frac{1}{2}}(I)
     $$
     where $b=b^*(s, \ell)$ and $\tilde{b}=b^*(s,\tilde{\ell})$.\\
\end{enumerate}
Theorem \ref{thm:persistence} then follows by repeatedly applying the implications (i)-(iii)
and using the fact that the assumptions on $(u, V)$ in Theorem \ref{thm:persistence} imply that $(u, V)\in S^{\frac{d-3}{2}, 0, 0}(I)\times W^{\frac{d-4}{2}, 0, \frac{d-4}{2}}(I)$.\\

In the remainder of this section we give the proof of the implications (i), (ii), and (iii) in Subsections \ref{subsec:improving schro reg}, \ref{subsec:improving wave reg I}, and \ref{subsec:improving wave reg II} respectively. The proof of Theorem \ref{thm:persistence} is then given in Subsection \ref{subsec:proof of persis thm}.

\subsection{Improving Schr\"odinger regularity}\label{subsec:improving schro reg} Our goal here is to prove the implication (i). Let $(s, \ell)$ and $(\tilde{s}, \ell)$ satisfy \eqref{eqn:cond on s l} and $\ell +\frac{1}{2} \les s < \tilde{s}$. Let $\tilde{a} = a^*(\tilde{s}, \ell)$ and $a  = a^*(s, \ell)$. Clearly we may also assume that $\tilde{s}< s +\frac{1}{8}$, since the general case follows by repeatedly applying this special case. The key point is to prove that  there exists $\theta>0$ such that for any interval $\tilde{I}\subset \RR$
        \begin{equation}\label{eqn:schro reg gain main}
            \| \mc{J}_0( |\nabla| |u|^2 ) \|_{W^{\ell, \tilde{a}, \beta}(\tilde{I})} \lesa \| u \|_{L^2_t W^{\frac{d-3}{2}, 2^*}_x(\tilde{I}\times \RR^d)}^\theta \| u \|_{S^{s, a, 0}(\tilde{I})}^{2-\theta}
        \end{equation}
where $\beta = \max\{\frac{d-4}{2}, \tilde{s}-1\}$. Supposing \eqref{eqn:schro reg gain main} holds, decomposing $I = \cup_{j=1}^N I_j$ with $\min|I_j \cap I_{j+1}|>0$, we may assume that on each interval $I_j$ we have
       $$ \| u \|_{L^2_t W^{\frac{d-3}{2}, 2^*}_x(I_j\times \RR^d)}^\theta \| u \|_{S^{s, a, 0}(I_j)}^{2-\theta} \ll  \epsilon$$
where $\epsilon>0$ is as in Theorem \ref{thm:gwp nonendpoint}. Choose $t_j \in I_j \cap I_{j+1}$. Applying \eqref{eqn:schro reg gain main} and a time translated version of Theorem \ref{thm:gwp nonendpoint} then implies that $u \in S^{\tilde{s}, \tilde{a}, 0}(I_j)$ with real-analytic dependence on $(u(t_j),V(t_j))$ for every $j=1, \dots, N$. Summing up the finite number of intervals $I_j$ via Lemma \ref{lem:decomposability}, then gives $u \in S^{\tilde{s}, \tilde{a}, 0}(I)$ and real-analytic dependence on $(u(0),V(0))$. In particular, we have the implication (i) under the additional assumption that $s< \tilde{s} < s+\frac{1}{8}$. But this implies (i) after repeatedly applying the above argument.

We now turn to the proof of \eqref{eqn:schro reg gain main}. In view of the bound $\| V \|_{W^{\ell, \tilde{a}, \beta}} \lesa \| V \|_{W^{\ell+\tilde{a}-a, a, \beta}}$, it suffices to show that
    \begin{equation}\label{eqn:schro reg gain main2}
        \| \mc{J}_0( |\nabla| |u|^2 ) \|_{W^{\ell+\tilde{a}-a, a, \beta}(\tilde{I})} \lesa \| u \|_{L^2_t W^{\frac{d-3}{2}, 2^*}_x(\tilde{I}\times \RR^d)}^\theta \| u \|_{S^{s, a, 0}(\tilde{I})}^{2-\theta}.
    \end{equation}
If $s>\frac{d-3}{2}$, then a computation shows that
    $$ \beta < \min\{s, 2s - \tfrac{d-2}{2} - a\}, \qquad 2a < 2s - (\ell + \tilde{a} -a)  - \tfrac{d-2}{2}, \qquad a< s-(\ell+\tilde{a}-a) $$
and hence \eqref{eqn:schro reg gain main2} follows from Corollary \ref{cor:loc bi wave}. On the other hand, in the endpoint case $s=\frac{d-3}{2}$, we have $a=\tilde{a}=0$ and $\ell = \frac{d-4}{2}$, and hence \eqref{eqn:schro reg gain main2}  follows from Proposition \ref{prop:wave est end}.

\subsection{Improving wave regularity I}\label{subsec:improving wave reg I} Our goal here is to prove the implication (ii). Let $(s, \ell)$ and $(s, \tilde{\ell})$ satisfy \eqref{eqn:cond on s l} and $\ell < \tilde{\ell}  \les s - \frac{1}{2}$. Without loss of generality, we may make the additional assumption that $\tilde{\ell} < \ell + \frac{1}{2}$, as the general case again follows by repeating this special case. Let $a=a^*(s, \ell)$ and $\tilde{a} = a^*(s, \tilde{\ell})$. A computation shows that
    $$ 2a < 2s - \tilde{\ell} - \frac{d-2}{2}, \qquad a < s  - \tilde{\ell}. $$
In particular, since $\tilde{a} \les a$, an application of Corollary \ref{cor:loc bi wave} implies that there exists $\theta>0$ such that
    \begin{equation}\label{eqn:wave reg gain I main}
        \| \mc{J}_0(|\nabla| |u|^2) \|_{W^{\tilde{\ell}, \tilde{a}, s-\frac{1}{2}}(I)} \les \| \mc{J}_0(|\nabla| |u|^2) \|_{W^{\tilde{\ell}, a, s-\frac{1}{2}}(I)} \lesa \| u \|_{L^2_t W^{\frac{d-3}{2}, 2^*}_x(I\times \RR^d)}^\theta \| u \|_{S^{s,a,0}(I)}^{2-\theta}
    \end{equation}
and hence $V = e^{it|\nabla|} V(0) + \mc{J}_0(|\nabla| |u|^2) \in W^{\tilde{\ell}, \tilde{a}, s-\frac{1}{2}}(I)$. It only remains to improve the Schr\"odinger regularity to $u \in S^{s, \tilde{a}, 0}(I)$ but this follows by arguing as in (i). Namely, we can decompose the interval $I = \cup_{j=1}^N I_j$ into a finite number of intervals $I_j$ satisfying $\min|I_j \cap I_{j+1}|>0$ and
        $$\| u \|_{L^2_t W^{\frac{d-3}{2}, 2^*}_x(I_j\times \RR^d)}^\theta \| u \|_{S^{s,a,0}(I)}^{2-\theta} \ll \epsilon $$
where $\epsilon>0$ is as in Theorem \ref{thm:gwp nonendpoint}. Choose $t_j \in I_j \cap I_{j+1}$. Applying the estimate \eqref{eqn:wave reg gain I main} together with Theorem \ref{thm:gwp nonendpoint}, we conclude that $u \in S^{s, \tilde{a}, 0}(I_j)$  with real-analytic dependence on $(u(t_j),V(t_j))$ for $j=1, \dots, N$ and hence $u \in S^{s,\tilde{a},0}(I)$  by Lemma \ref{lem:decomposability} and real-analytic dependence on $(u(0),V(0))$. Therefore the implication (ii) follows.

\subsection{Improving wave regularity II}\label{subsec:improving wave reg II} Our goal here is to prove the implication (iii). Let $(s, \ell)$ and $(s, \tilde{\ell})$ satisfy \eqref{eqn:cond on s l} and $s -1 < \ell < \tilde{\ell}$. Let $b = b^*(s, \ell)$ and $\tilde{b} = b^*(s, \tilde{\ell})$. Suppose that $(u, V) \in S^{s, 0, b}(I) \times W^{\ell, 0, s-\frac{1}{2}}(I)$, we would like to improve this to $(u, v) \in S^{s, 0, \tilde{b}}(I)\times W^{\tilde{\ell}, 0, s-\frac{1}{2}}(I)$, again with real-analytic dependence. In view of Theorem \ref{thm:main bilinear wave}, it suffices to show that $u \in S^{s,0, \tilde{b}}(I)$. Choose $\ell \les \ell' \les \tilde{\ell}$ such that
    $$ \max\big\{ \tfrac{d-4}{2} + \tilde{b}, s - 1 + \tilde{b} \big\} \les \ell' \les \min\big\{2s - \tfrac{d-2}{2}, s+b\big\}, \qquad (s, \ell') \not = \big( \tfrac{d-2}{2}, \tfrac{d-2}{2} + b \big).$$
An application of Theorem \ref{thm:main bilinear wave} gives
        $$ \| V \|_{W^{\ell', 0, s-\frac{1}{2}}(I)} \lesa \| V(0) \|_{H^{\tilde{\ell}}} + \| u \|_{S^{s, 0, b}(I)}^2 $$
and thus, via Theorem \ref{thm:main bilinear schro}, we conclude that
    $$ \| u \|_{S^{s, 0, \tilde{b}}(I)} \lesa \| u(0) \|_{H^s} + \| V \|_{W^{\ell', 0, s-\frac{1}{2}}(I)} \| u \|_{S^{s,0,0}(I)} \lesa \| u(0) \|_{H^s} + \big( \| V(0) \|_{H^{\tilde{\ell}}} + \| u \|_{S^{s,0, b}(I)}^2\big) \| u \|_{S^{s,0,0}(I)}.$$
Therefore $u \in S^{s, 0, \tilde{b}}(I)$ as required.

\subsection{Proof of Theorem \ref{thm:persistence}}\label{subsec:proof of persis thm} In view of the implications (i), (ii), and (iii), it suffices to show that if $(u, V) \in  C(I, H^{\frac{d-3}{2}}_x \times H^{\frac{d-4}{2}})$ is a solution to the Zakharov equation \eqref{eq:Zakharov 1st order} with $u \in L^2_t W^{\frac{d-3}{2}, 2^*}_x(I\times \RR^d)$, then $(u, V) \in S^{\frac{d-3}{2}, 0, 0}(I)\times W^{\frac{d-4}{2}, 0, \frac{d-4}{2}}(I)$. But this implication is contained in the argument used to prove uniqueness in Theorem \ref{thm:gwp endpoint}.

\section{Proofs of the main results}\label{sec:proofs-main}

\subsection{Proof of Theorem \ref{thm:lwp-ill}}\label{subsec:proof-lwp-ill}
Suppose that  $(s,\ell)$ satisfies \eqref{eqn:cond on s l}. Then, Theorem \ref{thm:gwp nonendpoint} or Theorem \ref{thm:gwp endpoint} imply well-posedness on a (small enough) interval $I\ni 0$.

Now, we prove the converse implication.
More precisely, we prove that the flow map of \eqref{eq:Zakharov 1st order} is not of class $C^2$ for $(s,\ell)$ which
do not satisfy \eqref{eqn:cond on s l}. Fix $(s,\ell)$ and assume the contrary.
Fix $t>0$ and consider
\begin{align*}
  I_\lambda (t)=&-i\int_0^te^{i(t-t')\Delta}\Big(\Re\big( e^{it'|\nabla|}g_\lambda\big) e^{it'\Delta}f_\lambda \Big)dt',\\
  J_\lambda (t)=&i\int_0^te^{i(t-t')|\nabla|}|\nabla|\Big(e^{it'\Delta}
h_\lambda \overline{e^{it'\Delta} h_\lambda }\Big)dt'
\end{align*}
for certain $\|g_\lambda\|_{H^\ell(\R^d)}\approx \|f_\lambda\|_{H^s(\R^d)}\approx \|h_\lambda\|_{H^s(\R^d)}\approx 1$, to be chosen.
$(I_\lambda, J_\lambda)$ corresponds to a second order directional derivative (G\^ateaux derivative) at the origin, which must be uniformly (in $\lambda$) bounded by our hypothesis.

We first prove lower bounds on $\ell$. Choose
\[
\widehat{g_\lambda}(\xi)=\lambda^{-\ell-\frac{d}{2}}\ind_{G_\lambda}(\xi), \; G_\lambda=\{\xi \in \R^d: \lambda \leq |\xi| \leq 2\lambda\}
\]
and
\[
\widehat{f_\lambda}(\xi)=\frac{\ind_{F_\lambda}(\xi)}{|\xi|^{\frac{d}{2}+s}\log(|\xi|)}, \; F_\lambda=\{\xi \in \R^d: 2 \leq |\xi| \leq \lambda/4\}.
\]
We compute
\[
  \int_{F_\lambda} \widehat{f_\lambda}(\xi) d\xi \approx \begin{cases} \frac{\lambda^{\frac{d}{2}-s}}{\log \lambda }& \text{ if }s<\frac{d}{2}\\
    \log\log \lambda  &\text{ if } s=\frac{d}{2}\\
    1 & \text{ if }   s>\frac{d}{2}
  \end{cases}.
\]
If $\tfrac54  \lambda\leq |\xi|\leq \tfrac32 \lambda$ and $\eta \in F_\lambda$, then $\xi-\eta\in G_\lambda$. Therefore,
\[
\|I_\lambda (t)\|_{H^s(\R^d)}\lesa \|g_\lambda\|_{H^\ell(\R^d)}\|f_\lambda\|_{H^s(\R^d)}, \;\text{ for all } \lambda \gg 1,
\]
implies
\[
   \lambda^{-2+s-\ell}\int_{F_\lambda}\widehat{f_\lambda}(\xi) d\xi\lesa \|I_\lambda (t)\|_{H^s(\R^d)}\lesa 1,\]
 which is true if and only if
 \[
   \begin{cases}
     \ell\geq \frac{d}{2}-2 &\text{ if } s<\frac{d}{2}\\
     \ell>\frac{d}{2}-2 &\text{ if } s=\frac{d}{2}\\
      \ell\geq s-2 &\text{ if } s>\frac{d}{2}
   \end{cases}.
 \]

Second, we prove lower bounds on $s$. Choose $h_\lambda=a_\lambda+b_\lambda$, where
\[
\widehat{a_\lambda}(\xi)=\lambda^{-s-\frac{d}{2}}\big(\ind_{A_\lambda}(\xi)+\ind_{-A_\lambda}(\xi)\big),
\; A_\lambda=\{\xi \in \R^d: |\xi-e_1\lambda | \leq \frac{1}{4}\lambda\}
\]
and
\[
\widehat{b_\lambda}(\xi)=\frac{\ind_{B_\lambda}(\xi)}{|\xi|^{\frac{d}{2}+s}\log(|\xi|)}, \; B_\lambda=\{\xi \in \R^d: 2 \leq |\xi| \leq \lambda/8\}.
\]
We compute
\[
  \int_{B_\lambda} \widehat{b_\lambda}(\xi) d\xi =\begin{cases} \frac{\lambda^{\frac{d}{2}-s}}{\log \lambda }& \text{ if }s<\frac{d}{2}\\
    \log\log \lambda  &\text{ if } s=\frac{d}{2}\\
    1 & \text{ if }   s>\frac{d}{2}
  \end{cases},
\]
as above. The spatial Fourier transform of $\Big(e^{it'\Delta}
a_\lambda \overline{e^{it'\Delta} a_\lambda }+e^{it'\Delta}
b_\lambda \overline{e^{it'\Delta} b_\lambda }\Big)$ is zero within the set
$C_\lambda=\{\xi \in \R^d: |\xi-\lambda e_1|\leq \tfrac18 \lambda\}.$
Further, if $\xi \in C_\lambda$ and $\eta \in
B_\lambda$, then $\xi-\eta\in A_\lambda$. Therefore, the bound
\[
\|J_\lambda (t)\|_{H^\ell(\R^d)}\lesa \|h_\lambda\|_{H^s(\R^d)}\|h_\lambda\|_{H^s(\R^d)}, \;\text{ for all } \lambda \gg 1,
\]
implies
\[
\Big(\int_{C_\lambda} \al \xi\ar^{2\ell} \Big|\mc{F} J^{ab}_\lambda (t)(\xi)\Big|^2d\xi\Big)^{\frac12} \lesa 1 , \;\text{ for all } \lambda \gg 1,
\]
where
\[
J^{ab}_\lambda (t)=\int_0^te^{i(t-t')|\nabla|}|\nabla|\Re\Big(e^{it'\Delta}
a_\lambda \overline{e^{it'\Delta} b_\lambda }\Big)dt'.
\]
Since
\[
\Big(\int_{C_\lambda} \al
\xi\ar^{2\ell} \Big|\mc{F} J^{ab}_\lambda
(t)(\xi)\Big|^2d\xi\Big)^{\frac12}\approx \lambda^{\ell-s-1}\int_{B_\lambda}\widehat{b_\lambda}(\xi) d\xi
\]
we must have
 \[
   \begin{cases}
     2s\geq \ell+\frac{d-2}{2} &\text{ if } s<\frac{d}{2}\\
     s>\ell-1 &\text{ if } s=\frac{d}{2}\\
      s\geq \ell-1 &\text{ if } s>\frac{d}{2}
   \end{cases}.
 \]

 \subsection{Proof of Theorem \ref{thm:small data gwp}}\label{subsec:proof-small-gwp}
 Suppose that $(s,\ell)$ satisfies \eqref{eqn:cond on s l} and $f \in H^s(\R^d)$ and $(g_0,g_1)\in H^\ell(\R^d)\times H^\ell(\R^d)$. Define $g=g_0-ig_1\in H^\ell(\R^d)$. Suppose that $\|f\|_{H^s}\leq \epsilon$. If $\epsilon>0$ is small enough (depending on $g$), due to the endpoint Strichartz estimate, we have that \eqref{eqn:thm gwp endpoint:smallness} is satisfied for $I=\R$, and Theorem \ref{thm:gwp endpoint} yields a unique global solution $(u,V)\in C(\R,H^s)\cap L^2_t W^{\frac{d-3}{2}, 2^*}_x(\R\times \RR^d)\times C(\R,H^\ell)$ to the Zakharov equation \eqref{eq:Zakharov 1st order}.
Note that $v=\Re V$, $|\nabla|^{-1}\partial_t v=\Im V$ have the same regularity. Also, by Theorem \ref{thm:persistence}, the additional regularity persits, i.e. $(u,V)\in S^{s,a,b}\times W^{\ell,a,s-\frac12}$ and we have real-analytic dependence. Further, this implies the scattering claim, as shown in the proof of Theorem \ref{thm:global model}.

\subsection*{Acknowledgements}
Financial support by the
  German Research Foundation (DFG) through the CRC 1283 ``Taming uncertainty and profiting from
  randomness and low regularity in analysis, stochastics and their
  applications'' is acknowledged.
  The authors would like to thank the anonymous referee for his valuable suggestions and Akansha Sanwal for spotting a number of typos. The authors thank Martin Spitz for pointing out errors in Theorems 7.6 and 7.7 of the published version. We have fixed the conditions (7.3) and (7.8), as well as the proofs, together with slight improvement of Propositions 6.1 and 6.2, in the current version of the paper.

\bibliographystyle{amsplain}
\bibliography{Zakharov}

\providecommand{\bysame}{\leavevmode\hbox to3em{\hrulefill}\thinspace}
\providecommand{\MR}{\relax\ifhmode\unskip\space\fi MR }
\providecommand{\MRhref}[2]{%
  \href{http://www.ams.org/mathscinet-getitem?mr=#1}{#2}
}
\providecommand{\href}[2]{#2}
\begin{thebibliography}{10}

\bibitem{added1988}
H{\'e}l{\`e}ne Added and St{\'e}phane Added, \emph{Equations of {L}angmuir
  turbulence and nonlinear {S}chr\"odinger equation: smoothness and
  approximation}, J. Funct. Anal. \textbf{79} (1988), no.~1, 183--210.
  \MR{950090 (89h:35273)}

\bibitem{Bejenaru2009}
I.~Bejenaru, S.~Herr, J.~Holmer, and D.~Tataru, \emph{On the 2{D} {Z}akharov
  system with {$L^2$}-{S}chr\"{o}dinger data}, Nonlinearity \textbf{22} (2009),
  no.~5, 1063--1089. \MR{2501036}

\bibitem{Bejenaru2015}
Ioan Bejenaru, Zihua Guo, Sebastian Herr, and Kenji Nakanishi,
  \emph{Well-posedness and scattering for the {Z}akharov system in four
  dimensions}, Anal. PDE \textbf{8} (2015), no.~8, 2029--2055. \MR{3441212}

\bibitem{Bejenaru2011}
Ioan Bejenaru and Sebastian Herr, \emph{Convolutions of singular measures and
  applications to the {Z}akharov system}, J. Funct. Anal. \textbf{261} (2011),
  no.~2, 478--506. \MR{2793120}

\bibitem{bourgain_wellposedness_1996}
Jean Bourgain and James~E. Colliander, \emph{On wellposedness of the {Z}akharov
  system}, Internat. Math. Res. Notices \textbf{1996} (1996), no.~11, 515--546.
  \MR{1405972 (97h:35206)}

\bibitem{Bourgain2014}
Jean Bourgain and Dong Li, \emph{{O}n an endpoint {K}ato-{P}once inequality},
  Differential and Integral Equations \textbf{27} (2014), no.~11/12,
  1037--1072.

\bibitem{Candy2021}
Timothy Candy, Sebastian Herr, and Kenji Nakanishi, \emph{Global wellposedness
  for the energy-critical {Z}akharov system below the ground state}, Adv. Math.
  \textbf{384} (2021), 107746. \MR{4246100}

\bibitem{Colin2004}
T.~Colin, G.~Ebrard, G.~Gallice, and B.~Texier, \emph{Justification of the
  {Z}akharov model from {K}lein-{G}ordon--wave systems}, Comm. Partial
  Differential Equations \textbf{29} (2004), no.~9-10, 1365--1401. \MR{2103840}

\bibitem{colliander_low_2008}
James~E. Colliander, Justin Holmer, and Nikolaos Tzirakis, \emph{Low regularity
  global well-posedness for the {Z}akharov and {K}lein-{G}ordon-{S}chr\"odinger
  systems}, Trans. Amer. Math. Soc. \textbf{360} (2008), no.~9, 4619--4638.
  \MR{2403699}

\bibitem{Deimling85}
Klaus Deimling, \emph{Nonlinear functional analysis}, Springer, Berlin [u.a.],
  1985.

\bibitem{Domingues2019}
L.~Domingues and R.~Santos, \emph{A note on {$C^2$} ill-posedness results for
  the {Z}akharov system in arbitrary dimension}, Trends Comput. Appl. Math.
  \textbf{24} (2023), no.~3, 505--519. \MR{4618362}

\bibitem{Fang2008}
Daoyuan Fang, Hartmut Pecher, and Sijia Zhong, \emph{Low regularity global
  well-posedness for the two-dimensional {Z}akharov system}, Analysis (Munich)
  \textbf{29} (2009), no.~3, 265--281. \MR{2568883}

\bibitem{ginibre_cauchy_1997}
Jean Ginibre, Yoshio Tsutsumi, and Giorgio Velo, \emph{On the {C}auchy problem
  for the {Z}akharov system}, J. Funct. Anal. \textbf{151} (1997), no.~2,
  384--436. \MR{1491547 (2000c:35220)}

\bibitem{ginibre_scattering_2006}
Jean Ginibre and Giorgio Velo, \emph{Scattering theory for the {Z}akharov
  system}, Hokkaido Math. J. \textbf{35} (2006), no.~4, 865--892. \MR{2289364
  (2007k:35347)}

\bibitem{glangetas_concentration_1994}
L\'eo Glangetas and Frank Merle, \emph{Concentration properties of blow-up
  solutions and instability results for {Z}akharov equation in dimension two.
  {II}}, Comm. Math. Phys. \textbf{160} (1994), no.~2, 349--389. \MR{1262202
  (95e:35196)}

\bibitem{glangetas_existence_1994}
\bysame, \emph{Existence of self-similar blow-up solutions for {Z}akharov
  equation in dimension two. {I}}, Comm. Math. Phys. \textbf{160} (1994),
  no.~1, 173--215. \MR{1262194 (95e:35195)}

\bibitem{GuoGanKongZhang2016}
Boling Guo, Zaihui Gan, Linghai Kong, and Jingjun Zhang, \emph{The {Z}akharov
  system and its soliton solutions}, Springer, Singapore; Science Press
  Beijing, Beijing, 2016. \MR{3560624}

\bibitem{Guo2014}
Zihua Guo, Sanghyuk Lee, Kenji Nakanishi, and Chengbo Wang, \emph{Generalized
  {S}trichartz estimates and scattering for 3{D}{Z}akharov system}, Comm. Math.
  Phys. \textbf{331} (2014), no.~1, 239--259. \MR{3232001}

\bibitem{Guo2014a}
Zihua Guo and Kenji Nakanishi, \emph{Small energy scattering for the {Z}akharov
  system with radial symmetry}, Int. Math. Res. Not. IMRN (2014), no.~9,
  2327--2342. \MR{3207369}

\bibitem{Guo2018}
\bysame, \emph{The {Z}akharov system in 4{D} radial energy space below the
  ground state}, Amer. J. Math. \textbf{143} (2021), no.~5, 1527--1600.
  \MR{4334403}

\bibitem{Guo2013}
Zihua Guo, Kenji Nakanishi, and Shuxia Wang, \emph{Global dynamics below the
  ground state energy for the {Z}akharov system in the 3{D} radial case}, Adv.
  Math. \textbf{238} (2013), 412--441. \MR{3033638}

\bibitem{Hani2013}
Zaher Hani, Fabio Pusateri, and Jalal Shatah, \emph{Scattering for the
  {Z}akharov system in 3 dimensions}, Comm. Math. Phys. \textbf{322} (2013),
  no.~3, 731--753. \MR{3079330}

\bibitem{holmer_local_2007}
Justin Holmer, \emph{Local ill-posedness of the 1{D} {Z}akharov system},
  Electron. J. Differential Equations (2007), No. 24, 22 pp. (electronic).
  \MR{2299578 (2007k:35465)}

\bibitem{Kato2017}
Isao Kato and Kotaro Tsugawa, \emph{Scattering and well-posedness for the
  {Z}akharov system at a critical space in four and more spatial dimensions},
  Differential Integral Equations \textbf{30} (2017), no.~9-10, 763--794.
  \MR{3656486}

\bibitem{Keel1998}
Markus Keel and Terence Tao, \emph{Endpoint {S}trichartz estimates}, Amer. J.
  Math. \textbf{120} (1998), no.~5, 955--980.

\bibitem{Kenig1995}
Carlos~E. Kenig, Gustavo Ponce, and Luis Vega, \emph{On the {Z}akharov and
  {Z}akharov-{S}chulman systems}, J. Funct. Anal. \textbf{127} (1995), no.~1,
  204--234. \MR{1308623}

\bibitem{Kishi2013}
Nobu Kishimoto, \emph{Local well-posedness for the {Z}akharov system on the
  multidimensional torus}, J. Anal. Math. \textbf{119} (2013), 213--253.
  \MR{3043152}

\bibitem{KishiMaeda2013}
Nobu Kishimoto and Masaya Maeda, \emph{Construction of blow-up solutions for
  {Z}akharov system on {$\mathbb{T}^2$}}, Ann. Inst. H. Poincar\'{e} Anal. Non
  Lin\'{e}aire \textbf{30} (2013), no.~5, 791--824. \MR{3103171}

\bibitem{Mas08}
Nader Masmoudi and Kenji Nakanishi, \emph{Energy convergence for singular
  limits of {Z}akharov type systems}, Invent. Math. \textbf{172} (2008), no.~3,
  535--583. \MR{2393080}

\bibitem{merle_blowup_1998}
Frank Merle, \emph{Blow-up phenomena for critical nonlinear {S}chr\"odinger and
  {Z}akharov equations}, Proceedings of the {I}nternational {C}ongress of
  {M}athematicians, Vol. III ({B}erlin, 1998), 1998, pp.~57--66 (electronic).
  \MR{1648140 (99h:35200)}

\bibitem{ozawa_existence_1992}
Tohru Ozawa and Yoshio Tsutsumi, \emph{Existence and smoothing effect of
  solutions for the {Z}akharov equations}, Publ. Res. Inst. Math. Sci.
  \textbf{28} (1992), no.~3, 329--361. \MR{1184829 (93k:35246)}

\bibitem{ozawa_nonlinear_1992}
\bysame, \emph{The nonlinear {S}chr\"odinger limit and the initial layer of the
  {Z}akharov equations}, Differential Integral Equations \textbf{5} (1992),
  no.~4, 721--745. \MR{1167491 (93d:76079)}

\bibitem{Ozawa1994}
\bysame, \emph{Global existence and asymptotic behavior of solutions for the
  {Z}akharov equations in three space dimensions}, Adv. Math. Sci. Appl.
  \textbf{3} (1993/94), no.~Special Issue, 301--334. \MR{1287933}

\bibitem{Schochet-Weinstein}
Steven~H. Schochet and Michael~I. Weinstein, \emph{The nonlinear
  {S}chr\"{o}dinger limit of the {Z}akharov equations governing {L}angmuir
  turbulence}, Comm. Math. Phys. \textbf{106} (1986), no.~4, 569--580.
  \MR{860310}

\bibitem{Sulem1999}
Catherine Sulem and Pierre-Louis Sulem, \emph{The nonlinear {S}chr\"{o}dinger
  equation}, Applied Mathematical Sciences, vol. 139, Springer-Verlag, New
  York, 1999, Self-focusing and wave collapse. \MR{1696311}

\bibitem{Texier2007}
Benjamin Texier, \emph{Derivation of the {Z}akharov equations}, Arch. Ration.
  Mech. Anal. \textbf{184} (2007), no.~1, 121--183. \MR{2289864}

\bibitem{zakharov_collapse_1972}
Vladimir~E. Zakharov, \emph{{C}ollapse of {L}angmuir waves}, Sov. Phys. JETP
  \textbf{35} (1972), no.~5, 908--914.

\end{thebibliography}
\end{document}